\documentclass[graybox]{svmult}

\usepackage{helvet}         
\usepackage{courier}        
\usepackage{type1cm}        
\usepackage{empheq}
\usepackage{makeidx}         
\usepackage{graphicx}        
\usepackage{multicol}        
\usepackage[bottom]{footmisc}
\usepackage{algorithm}
\usepackage[noend]{algpseudocode}

\usepackage{xcolor}

\makeindex             

\usepackage{a4,amssymb,epsfig}
\usepackage[all]{xy}

\usepackage{amsmath,amsthm}
\usepackage{commath}
\usepackage{color}
\usepackage{lscape}
\usepackage[bookmarks,bookmarksnumbered,plainpages=false]{hyperref}
\usepackage{colonequals}
\usepackage{subfig}

\newcommand{\StfR}{\mathbb{V}_p(\mathbb{R}^n)}
\newcommand{\R}{\mathcal{R}}

\newcommand{\di}{\mathrm{d}}
\newcommand{\Diffeom}{\mathrm{Diff}^+(I)}
\newcommand{\LieG}{G}
\newcommand{\LieA}{\mathfrak{g}}

\newcommand{\RightTrans}{R}

\newcommand{\SRVT}{\mathcal{R}}

\newcommand{\G}{G}
\newcommand{\g}{\mathfrak{g}}
\newcommand{\h}{\mathfrak{h}}
\newcommand{\m}{\mathfrak{m}}
\newcommand{\M}{\mathcal{M}}
\DeclareMathOperator{\id}{id}
\DeclareMathOperator{\Evol}{Evol}
\DeclareMathOperator{\Diff}{Diff}

\DeclareMathOperator{\GL}{GL}
\DeclareMathOperator{\SO}{SO}

\DeclareMathOperator{\SL}{SL}
\DeclareMathOperator{\Ad}{Ad}

\DeclareMathOperator{\Lf}{\mathbf{L}}

\DeclareMathOperator{\SRp}{\theta}
\DeclareMathOperator{\ev}{ev}

\newcommand{\SSpace}{\mathcal{S}}
\newcommand{\PP}{\mathcal{P}}

\newcommand{\coloneq}{\colonequals}
\newcommand{\Pomega}{\theta_\omega}

\spnewtheorem{prop}[theorem]{Proposition}{\bfseries}{\itshape}
\spnewtheorem{lem}[theorem]{Lemma}{\bfseries}{\itshape}
\spnewtheorem{ex}[theorem]{Example}{\bfseries}{\rmfamily}
\spnewtheorem{thm}[theorem]{Theorem}{\bfseries}{\itshape}
\spnewtheorem{defn}[theorem]{Definition}{\bfseries}{\rmfamily}
\spnewtheorem{rem}[theorem]{Remark}{\bfseries}{\rmfamily} 

\spnewtheorem{setup}{\nocaption}[subsection]{\bfseries}{\rmfamily}

\title*{Shape analysis on homogeneous spaces: a generalised SRVT framework}
\author{E.\ Celledoni, S.\ Eidnes, A.\ Schmeding}
\institute{Elena Celledoni \email{elena.celledoni@ntnu.no}\and S\o{}lve Eidnes \email{solve.eidnes@ntnu.no} \and Alexander Schmeding  \email{schmeding@tu-berlin.de}
 \at NTNU Trondheim, Institutt for matematiske fag, 7491 Trondheim, Norway}
\begin{document}

\maketitle

\abstract{Shape analysis is ubiquitous in problems of pattern and object recognition and has developed considerably in the last decade. The use of shapes is natural in applications where one wants to compare curves independently of their parametrisation. 
One computationally efficient approach to shape analysis is based on the Square Root Velocity Transform (SRVT). In this paper we propose a generalised SRVT framework for shapes on homogeneous manifolds. The method opens up for a variety of possibilities based on  different choices of Lie group action and giving rise to different Riemannian metrics. }
\section{Shapes on homogeneous manifolds}\label{sect: SRVT}

Shapes are  unparametrised curves, evolving on a vector space, on a Lie group or on a manifold. Shape spaces and spaces of curves are infinite dimensional Riemannian manifolds, whose Riemannian metrics are the essential tool to compare and analyse shapes. By combining infinite dimensional differential geometry, analysis and computational mathematics, shape analysis provides a powerful approach to a variety of applications. 

In this paper, we are concerned with the approach to shape analysis based on the Square Root Velocity Transform (SRVT), \cite{srivastava11sao}.  This method is effective and computationally efficient. On vector spaces, the SRVT maps parametrised curves  to appropriately scaled tangent vector fields along them. The transformed curves are compared computing geodesics in the $L^2$ metric, and the scaling can be chosen suitably to yield reparametrisation invariance, \cite{srivastava11sao}, \cite{bauer14cri}.  
Notably, applying a (reparametrisation invariant) $L^2$ metric directly on the original parametrised curves is not an option as it leads to vanishing geodesic distance on parametrised curves and on the quotient shape space \cite{michor05vgd,MR2891297}. As an alternative, higher order Sobolev type metrics were proposed \cite{michor06rgo}, even though they can be computationally demanding, since computing geodesics in this infinite dimensional Riemannian setting amounts in general to solving numerically partial differential equations. These geodesics are used in practice for finding distances between curves and for interpolation between curves. The SRVT approach, on the other hand, is quite practical because it allows the use of the $L^2$ metric on the transformed curves: distances between curves are just $L^2$ distances of the transformed curves, and geodesics between curves are ``{\it straight lines}" between the transformed curves. It is also possible to prove that this algorithmic approach corresponds (at least locally) to a particular Sobolev type metric, see \cite{bauer14cri, celledoni15sao}. 

In the present paper we propose a generalisation of the SRVT,  from vector spaces and Lie groups, \cite{srivastava11sao, bauer14cri}, to homogeneous manifolds.
This problem has been previously considered for manifold valued curves in \cite{su14sao,brigant}, but our approach is different, the main idea is to take advantage of the Lie group acting transitively on the homogeneous manifold. 
The Lie group action allows us to transport derivatives of curves to our choice of base point in the homogeneous manifold.
Then this information is lifted to a curve in the Lie algebra.
It is natural to require that the lifted curve 
does not depend on the representative of the class used to pull back the curve to the base point.

The main contribution of this paper is the definition of a generalised square root velocity transform framework using transitive Lie group actions for curves on homogeneous spaces. Different choices of Lie group actions will give rise to different metrics on the infinite dimensional manifold of curves on the homogeneous space, with different properties. These different metrics, their  geodesics and associated geometric tools for shape analysis 
can all be implemented in the computationally advantageous SRVT framework. 

We extend previous results for Lie group valued curves and shapes   \cite{celledoni15sao}, to the homogeneous manifold setting. 
Using ideas from the literature on differential equations on manifolds \cite{celledoni03oti}, we describe the main tools necessary for the definition of the SRVT and discuss the minimal requirements guaranteeing that the SRVT is well defined, section~\ref{sect: Con:SRVT}.  
On a general homogeneous manifold, the SRVT is obtained using a right inverse of the composition of the Lie group action with the evolution operator of the Lie group.   If the homogeneous manifold is reductive, there is an explicit way to construct this right inverse (based on a canonical $1$-form for the reductive space, cf.\ \ref{setup: oneform} - \ref{setup: alphared}), see also \cite{munthekaas15ioh}.  We prove smoothness of the defined SRVT in section~\ref{smoothness}.  Detailed examples on matrix Lie groups are provided in section~\ref{examplesmatrices}.

A Riemannian metric on the manifold of curves on the homogeneous space is obtained by pulling back the $L^2$ inner product of curves on the Lie algebra through the SRVT, Theorem~\ref{thm: pbmetric}. To ensure that the distance function obtained on the space of parametrised curves descends to a distance function on the shape space, it is necessary to prove equivariance with respect to the group of orientation preserving diffeomorphisms (reparametrization invariance), these results are presented in section~\ref{equivariance}. 

For the case of reductive homogeneous spaces, fixed the Lie group action, two different approaches are considered: one obtained pulling back the curves to the Lie algebra $\mathfrak{g}$ (Proposition~\ref{prop: defn:SRVT:hom}) and one obtained pulling back the curves to the reductive subspace $\mathfrak{m}\subset \mathfrak{g}$ (section~\ref{setup: twist}). The resulting distances are both reparametrization invariant, see Lemmata~\ref{lem: repinv} and \ref{lem: red:repinv}.  For the second approach it follows similarly to what shown in \cite{celledoni15sao} that the geodesic distance is globally defined by the $L^2$ distance, Proposition~\ref{prop: gL2dist}. 
We conjecture that also for general homogeneous manifolds, at least locally, the geodesic distance of the pullback metric is given by the $L^2$ distance of the curves transformed by the SRVT, see end of section~\ref{geometricproperties}. 
To illustrate the performance of the proposed approaches we compute geodesics between curves on the $2$-sphere (viewed as a homogeneous space with respect to the canonical $\SO(3)$-action), see Figure~\ref{fig:appetizer} for an example.
Numerical experiments show that the two algorithms perform differently when applied to curves on the sphere (section~\ref{numericalexperiments}). 


\begin{figure}[htbp]
\begin{center}
\includegraphics[width=0.3\textwidth]{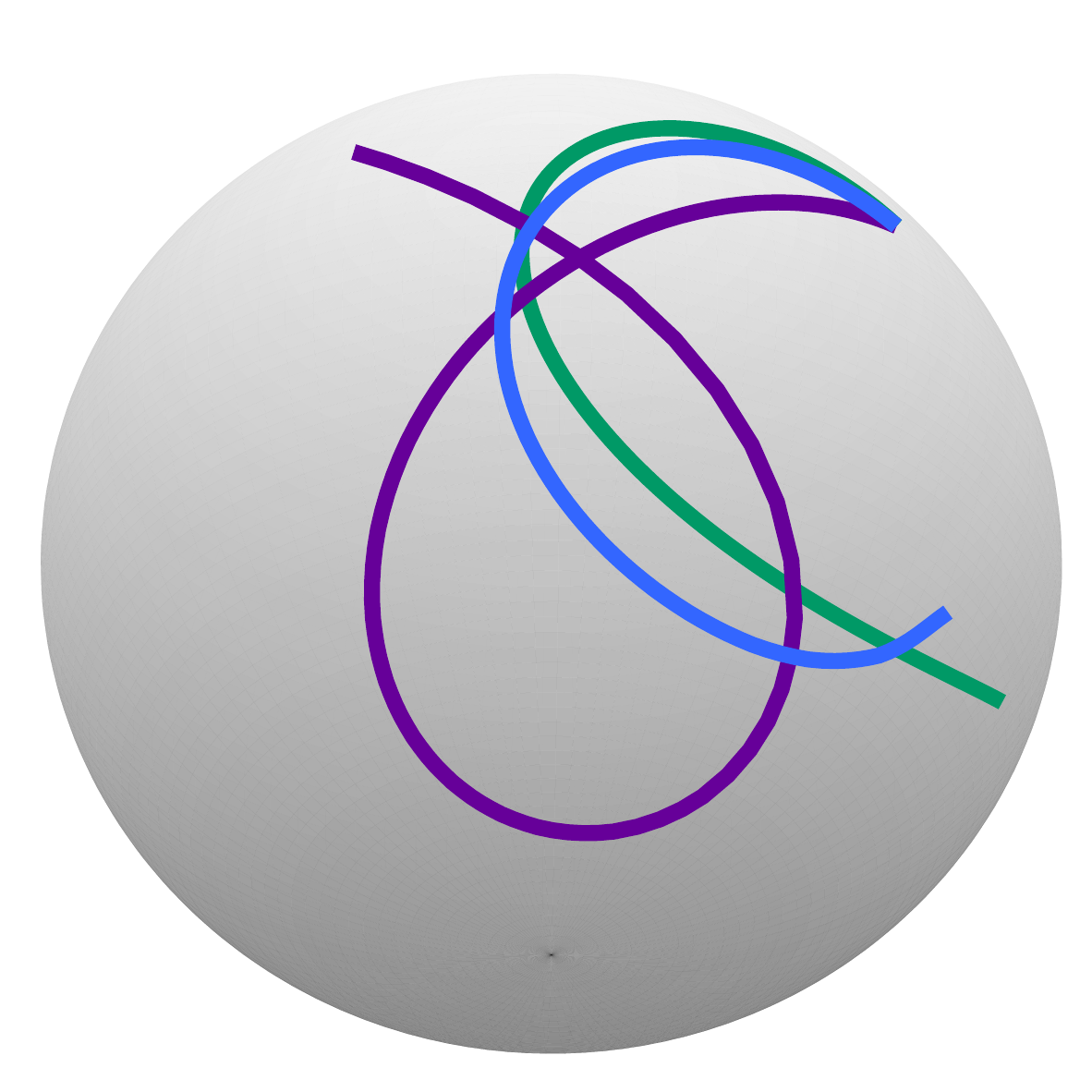}
\includegraphics[width=0.3\textwidth]{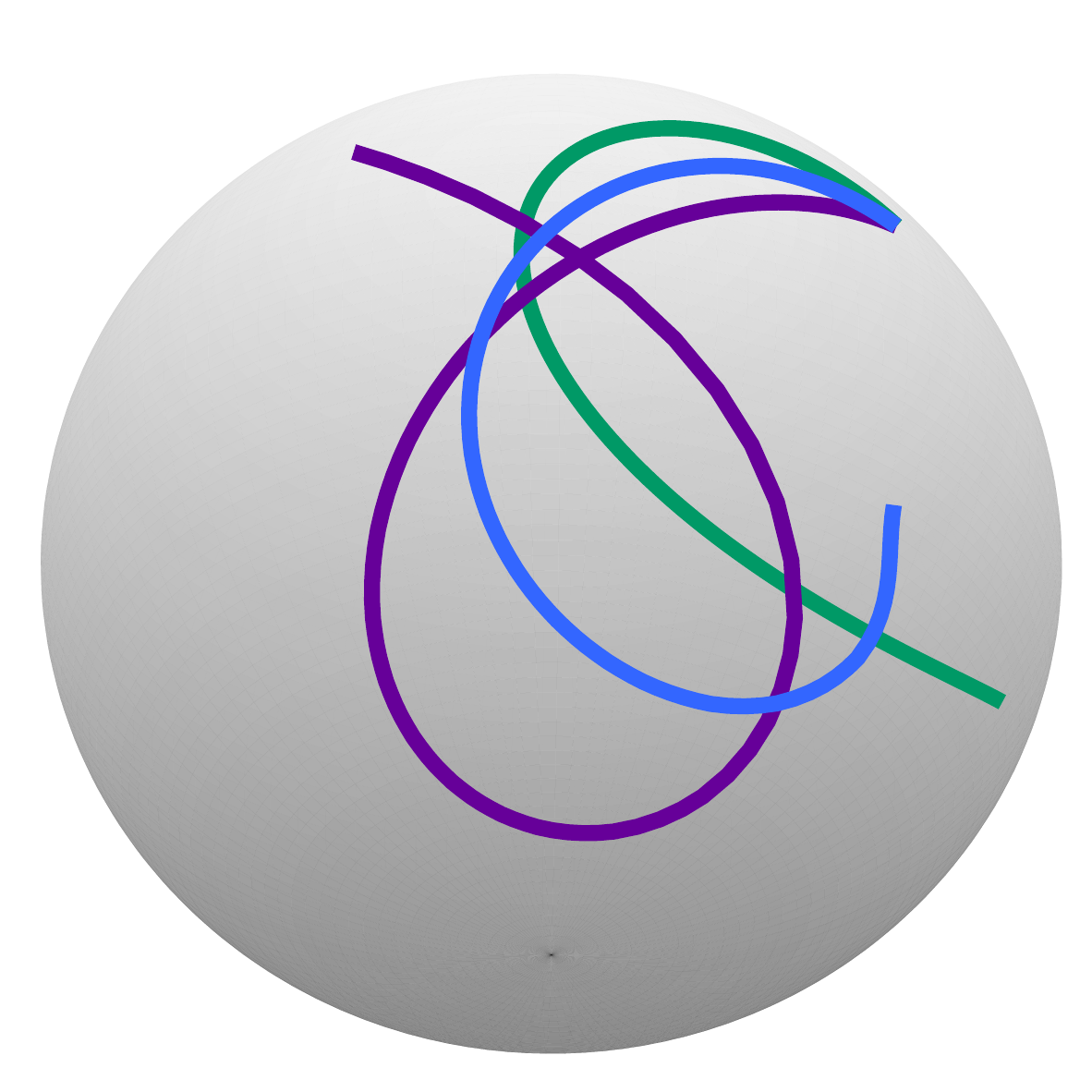}
\includegraphics[width=0.3\textwidth]{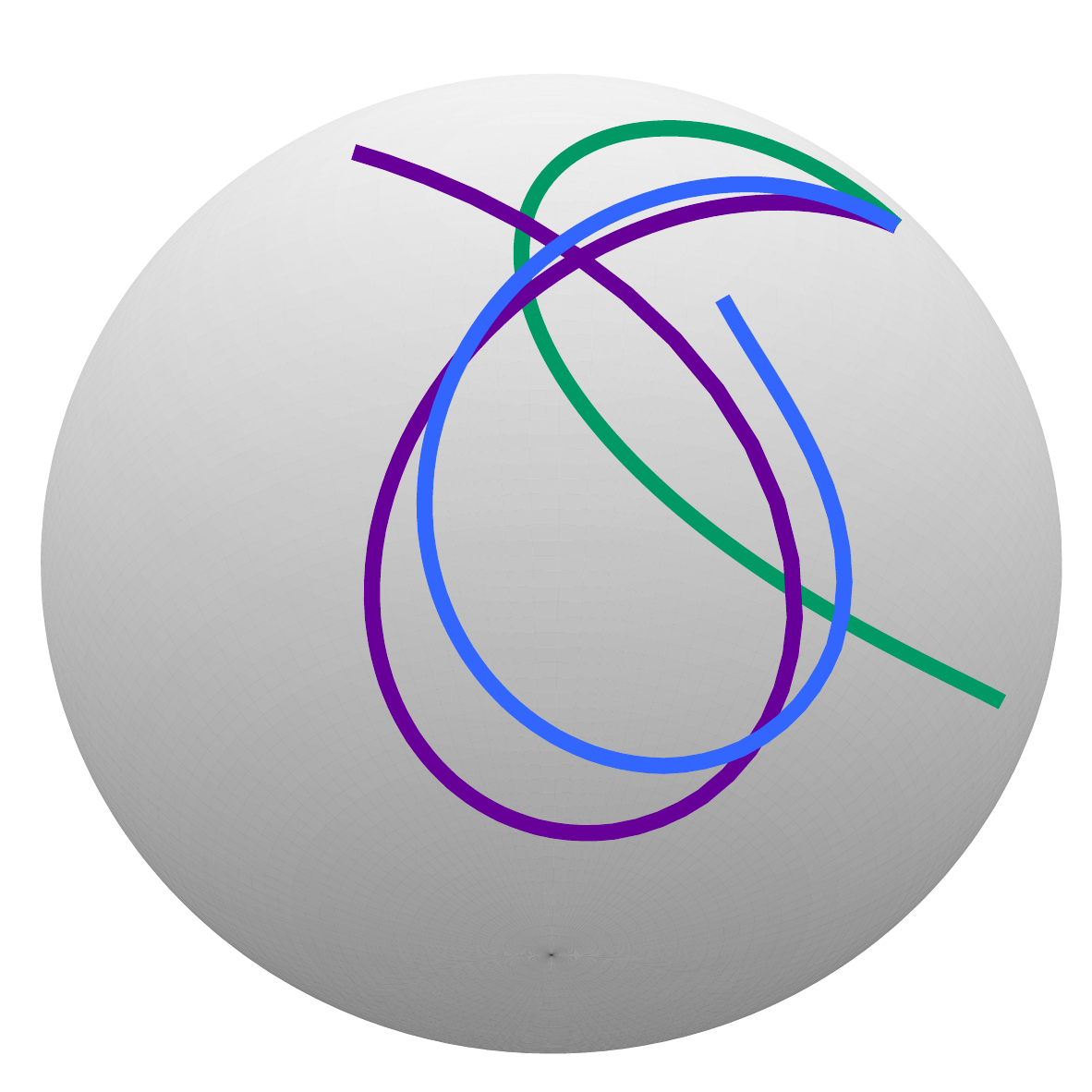}
\caption{The blue curve shows the deformation of the green curve into the purple one along a geodesic $\gamma \colon [0,1] \rightarrow \text{Imm} (I,S^2)$ plotted for the three times $\left\{\frac{1}{4},\frac{1}{2},\frac{3}{4}\right\}$ from left to right.}
\label{fig:appetizer}
\end{center}
\end{figure}
 This work appeared on the arXiv on the 5th of April 2017, later a related but different work from colleagues at Florida State University was completed and posted on the arXiv on the 9th of June 2017. The latter work has now appeared in \cite{SU}, see also the follow up \cite{BKS17}. Moreover, loc.cit.\ treats quotients by compact subgroups focuses on the existence of optimal reparametrisations.
 
\subsection{Preliminaries and notation}
 
 Fix a Lie group $\G$ with identity element $e$ and Lie algebra $\g$.\footnote{ In this paper we assume all Lie groups and Lie algebras to be finite dimensional. 
 Note however, that many of our techniques carry over to Lie groups modelled on Hilbert spaces, \cite{celledoni15sao}.} Denote by $R_g \colon \G \rightarrow \G$ and $L_g \colon G \rightarrow G$ the right resp.\ left multiplication by $g\in \G$.
 Let $H$ be a closed Lie subgroup of $G$ and $\M := G/H$ the quotient with the manifold structure turning $\pi \colon G \rightarrow G/H, g \mapsto gH$ into a submersion (see \cite[Theorem G (b)]{glockner15fos}).
 Then $\M$ becomes a homogeneous space for $G$ with respect to the (transitive) left action: 
 $$\Lambda \colon \G\times \M\rightarrow \M, \quad (g,kH) \mapsto (gk)H.$$ 
  For $c_0 \in \M$ we write
$\Lambda(g,c_0)=\Lambda_{c_0}(g)=g. c_0=\Lambda^{g}(c_0)$, i.e.\ $\Lambda_{c_0}:G\rightarrow \mathcal{M}$ (the orbit map of the orbit through $c_0$) and $\Lambda^{g}: \mathcal{M}\rightarrow \mathcal{M}$. 

\begin{setup}
We will consider smooth curves on $\M$ and describe them using the Lie group action. Namely for $c \colon [0,1] \rightarrow \M$ we choose a smooth lift $g \colon [0,1] \rightarrow G$ of $c$, i.e.:
$$c(t)=g(t). c_0,\quad c_0\in\M, \quad t\in[0,1] \quad (\text{the dot denotes the action of } G \text{ on } G/H).$$
In general, there are many different choices for a smooth lifts $g$.\footnote{Every homogeneous space $G/H$ is a principal $H$-bundle, whence there are smooth horizontal lifts of smooth curves (depending on some choice of connection, cf.\ e.g.\ \cite[Chapter 5.1]{MR2021152}).}  For brevity we will in the following write $I := [0,1]$.
\end{setup}

Later on we consider smooth functions on infinite-dimensional manifolds beyond the realm of Banach manifolds. Hence the standard definition for smooth maps (i.e.\ the derivative as a (continuous) map to a space of continuous operators) breaks down.
We base our investigation on the so called Bastiani calculus (see \cite{bastiani64}): A map $f \colon E \supseteq U \rightarrow F$ between Fr\'{e}chet spaces is smooth if all iterated directional derivatives exist and glue together to continuous maps.\footnote{In the setting of manifolds on Fr\'{e}chet spaces (with which we deal here) our setting of calculus is equivalent to the so called convenient calculus (see \cite{MR1471480}).  Convenient calculus defines a map $f$ to be smooth if it ``maps smooth curves to smooth curves", i.e.\ $f\circ c$ is smooth for any smooth curve $c$. This yields a calculus on infinite-dimensional spaces where smoothness does not necessarily imply continuity (though this does not happen on Fr\'{e}chet spaces), we refer to \cite{MR1471480} for a detailed exposition.
Note that both calculi can handle smooth maps on intervals $[a,b]$, see e.g.\ \cite[1.1]{hgreg2015} and \cite[Chapter 24]{MR1471480}.}
\begin{setup}
Let $M$ be a (possibly infinite-dimensional) manifold. By $C^\infty (I,M)$ we denote smooth functions from $I$ to $M$. Recall that the topology on these spaces, the compact-open $C^\infty$-topology, allows one to control a function and its derivatives. This topology turns $C^\infty (I,M)$ into an infinite-dimensional manifold (see e.g.\ \cite[Section 42]{MR1471480}).

Denote by $\mathrm{Imm} (I,M) \subseteq C^\infty (I,M)$ the set of smooth immersions (i.e.\ smooth curves $c \colon I \rightarrow M$ with $\dot{c} (t)\neq 0$) and recall from \cite[41.10]{MR1471480} that $\mathrm{Imm} (I,M)$ is an open subset of $C^\infty (I,M)$.
\end{setup}

\begin{setup}\label{setup: Evol}
 We further denote by $\Evol$ the evolution operator, which is defined as
\begin{align*}
&\Evol: C^\infty(I, \mathfrak{g}) \rightarrow \lbrace g \in C^\infty(I, G): g(0) = e \rbrace =: C^\infty_* (I, \LieG)\\
&\Evol (q) (t) := g(t) \qquad \text{where} \qquad 
\left\{\begin{array}{ccl}
\frac{\di }{\di t}\,g &=&\RightTrans_{g(t)*}(q(t)),\\[0.2cm]
g(0)&=&e
\end{array}\right.
\end{align*} and 
$\RightTrans_{g*} = T_e R_g$ is the tangent of the right translation.
%
Recall from \cite[Theorem A]{hgreg2015} that $\Evol$ is a diffeomorphism with inverse the \emph{right logarithmic derivative} 
  \begin{align*}
   \delta^r \colon C^\infty_* (I, \LieG) \rightarrow C^\infty (I,\LieA), \quad
   \delta^r g  := R_{g^*}^{-1} (\dot{g}).
  \end{align*}
\end{setup}

 \begin{setup}
 	We fix a Riemannian metric $(\langle \cdot, \cdot \rangle_g)_{g \in G}$ on $G$ which is right $H$-invariant (i.e.\ the maps $R_h, h \in H$ are Riemannian isometries). 
 	 Since $\mathcal{M}=G/H$ is constructed using the right $H$-action on $G$, an $H$-right invariant metric descends to a Riemannian metric on $\mathcal{M}$. We refer to \cite[Proposition 2.28]{MR2088027} for details and will always endow the quotient with this canonical metric to relate the Riemannian geometries. 
 
 Hence $H$-right invariance should be seen as a minimal requirement for the metric on $G$. Note that a natural way to obtain (right) invariant metrics is to transport a Hilbert space inner product from the Lie algebra by (right) translation in the group. This method yields a $G$-right invariant metric and we will usually work with such a metric induced by $\langle \cdot,\cdot\rangle$ on $\g$. Albeit it is very natural, $G$-invariance does not immediately add any benefits.
 In the following table we record properties of $H$, the Riemannian metric and of the canonical $G$-action on the quotient.
	\begin{table}[h!]
 		\centering
 		\caption{Riemannian metrics and dsectionecompositions of the Lie algebra}
 		\label{tab:table1}
 		\noindent\begin{tabular}{|c | l | l | l | l|}
 			\hline {\bf $H$ / $\h$} & {\bf metric on $G$ } & {\bf special decompositions of $\mathfrak{g}$} &  {\bf $G$-action on $\mathcal{M}$} \\ \hline
 			&&&\\
 			compact & $G$-left invariant,  & $\mathfrak{g}=\mathfrak{h}\oplus \mathfrak{h}^\perp$, the orthogonal     &  by isometries  \\
 			&  $H$-biinvariant &  complement $\mathfrak{h}^\perp$ is $\Ad(H)$-invariant &  \\ \hline
 			&&&\\
 			compact  & $G$-right invariant,& as above & only $H$ acts \\
 			&  $H$-biinvariant & &  by isometries \\ 
 			\hline &&&\\
 		 \ \	admits reductive \ \ \mbox{} & $G$-right invariant  & $\mathfrak{g}=\mathfrak{h}\oplus \mathfrak{m}$, $\m$ is $\Ad (H)$-invariant & not by isometries \\
 			complement in $\g$\footnotemark & & $\g =\h \oplus \h^\perp$, where in general $\m \neq \h^\perp$&  \\ 
 			\hline &&& \\ 
 			& $G$-right invariant & $\mathfrak{g}=\mathfrak{h}\oplus \mathfrak{h}^\perp$  but $\mathfrak{h}^\perp$ is not  & not by isometries\\
 			& & $\Ad (H)$ invariant & \\ \hline
 		\end{tabular}\\
 	\end{table}
\footnotetext{$\h = \Lf (H)$ admits a \emph{reductive complement} $\m$, if $\m$ is an $\Ad (H)$-invariant subspace and $\g  = \h \oplus \m$ as vector spaces, cf.\ \ref{setup: adj}. Then $\mathcal{M}=G/H$ is a reductive homogeneous space.}
\end{setup}

 \begin{setup} Let $f \colon M \rightarrow N$ be a smooth map and denote postcomposition by
 	$$\theta_f \colon C^{\infty}(I, M)\rightarrow C^{\infty}(I,N),\qquad c\mapsto f\circ c.
 	$$
 Note that $\theta_f$ is smooth as a map between (infinite-dimensional) manifolds. 
 \end{setup}


 \begin{setup}[The SRVT on Lie groups]\label{setup:SRVT}
  For a Lie group $\G$ with Lie algebra $\g$, consider an immersion $c \colon I \rightarrow \G$.
 The square root velocity transform of $c$ is
 \begin{equation}\label{SRVT: LIE}
 \R \colon \mathrm{Imm} (I,G) \rightarrow C^\infty (I, \g\setminus \{0\}), \qquad \R (c ) = \frac{\delta^r (c)}{\sqrt{\norm{\dot{c}}}} = \frac{\left(R_{c(t)}^{-1}\right)_* (\dot{c})}{\sqrt{\norm{\dot{c}}}} ,
 \end{equation}
 where the norm $\|\cdot \|$ is induced by a right $G$-invariant Riemannian metric, \cite{celledoni15sao}.  
 The SRVT consists of the composition of three maps: 
 \begin{itemize}
\item  {\it differentiation} $D \colon C^\infty (I,\G) \rightarrow C^\infty (I,T\G), D(c) := \dot{c}$, 
\item {\it transport} $ \alpha \colon C^\infty (I, T\G) \rightarrow C^\infty (I, \g),\quad  \gamma \mapsto (R_{\pi_{TG} \circ \gamma}^{-1})_* (\gamma)$ and 
\item {\it scaling}  $\mathrm{sc} \colon C^\infty (I,\mathfrak{g} \setminus \{0\}) \rightarrow C^\infty (I,\mathfrak{g}\setminus \{0\}),\quad   q \mapsto \left(t \mapsto \frac{q(t)}{\sqrt{\|q(t)\|}}\right)$.
\end{itemize}
 The scaling by the square root of the norm of the velocity is crucial to obtain a parametrisation invariant Riemannian metric, see \cite{celledoni15sao} and Lemma \ref{lem: repinv}. 
 \end{setup}
  
\section{Definition of the SRVT for homogeneous manifolds}\label{sect: Con:SRVT}

Our aim is to construct the SRVT for curves with values in the homogeneous manifold $\mathcal{M}$.
It was crucial in our investigation of the Lie group case \cite{celledoni15sao} that the right-logarithmic derivative inverts the evolution operator, see~\ref{setup: Evol}. 
To mimic this behaviour we introduce a version of the evolution for homogeneous manifolds. 

\begin{defn}
 Fix $c_0 \in \mathcal{M}$ and denote by $C^\infty_{c_0} (I,\mathcal{M})$ all smooth curves $c \colon I \rightarrow \mathcal{M}$ with $c(0)=c_0$.
 Then we define
 \begin{displaymath}
  \rho_{c_0}\colon C^{\infty}(I,\mathfrak{g}) \rightarrow C^{\infty}_{c_0}(I,\mathcal{M}), \quad \rho_{c_0}(q)= \Lambda_{c_0} (\Evol (q) (t)) = \Lambda (\Evol(q)(t),c_0).
 \end{displaymath}
\end{defn}

\begin{rem}\label{rem: gen:EVOL}
Fix $q \in C^\infty (I,\g)$ and $c_0 \in \mathcal{M}$ and denote by $g(t) = \Evol(q)(t)$.
Then  
\begin{displaymath}
\rho_{c_0}(q) := c(t) \qquad \text{where}\qquad \left\{\begin{array}{ccl}
   \frac{\dif}{\dif t}c(t)&=& T_e\Lambda_{c(t)}(\,q(t)\,),\\[0.2cm]
   c(0) &=& c_0.
  \end{array}\right.
\end{displaymath}
\end{rem}
\begin{proof}
In fact
\begin{align*}
 \frac{\dif}{\dif t}\rho_{c_0}(q)(t) &= T_{g(t)} \Lambda_{c_0} \left(\frac{\dif}{\dif t} g(t)\right) =  T_{g(t)} \Lambda_{c_0}((R_{g(t)})_*(q(t))) =
 T_{g(t)} \Lambda_{c_0}\circ (R_{g(t)})_{*}\,(q(t))\\
&= T_e(\Lambda_{c_0}\circ R_{g(t)})(q(t)) 
=T_e \Lambda_{\Lambda_{c_0} (g(t))} (q(t)) 
  = T_e \Lambda_{\rho_{c_0} (q)(t)} (q(t)),
\end{align*}
with
$T_{g(t)}\Lambda_{c_0} \colon T_{g(t)} G\rightarrow T_{\Lambda_{c_0}(g(t))}\mathcal{M} = T_{\rho_{c_0}(q)(t)} \mathcal{M},\qquad T_e \Lambda_{c(t)} \colon \mathfrak{g}\rightarrow T_{c(t)}\mathcal{M}.$
\end{proof}

Hence we can interpret $\rho_{c_0}$ as a version of the evolution operator $\Evol$ for homogeneous manifolds.
\begin{ex}
Consider the two dimensional unit sphere $\mathcal{M}=S^2$ in $\mathbb{R}^3$.  Consider the action of $\SO(3)$ on $S^2$ by matrix-vector multiplication: $\Lambda: \SO(3)\times S^2\rightarrow S^2$, $\Lambda(Q,u)=Q\cdot u$. Assume $c_0:=e_1$ the first canonical vector in $\mathbb{R}^3$, then given a curve in the Lie algebra of skew-symmetric matrices $q(t)\in \mathfrak{so}(3)$,  $\rho_{e_1}(q(t))=y(t)$, where $y(t)$ satisfies $\dot{y}=q(t)y$ with $y(0)=e_1$.
\end{ex}


We want to construct a section of the submersion $\rho_{c_0}$ to mimic the construction for Lie groups, see also \cite[Proposition 2.2]{celledoni03oti}. 
As we have seen in the Lie group case, the SRVT factorises into a derivation map, a map transporting the derivative to the Lie algebra and a scaling in the Lie algebra. 
For homogeneous spaces, we can make sense of this procedure if we can replace the transport from the Lie group case by a map which transports derivatives from the tangent bundle of the homogeneous manifold to the Lie algebra. Thus we search for a map $\alpha \colon C^\infty (I,T\mathcal{M}) \rightarrow C^\infty (I,\g)$ such that the following diagram commutes:
\begin{displaymath}
 \begin{xy}
  \xymatrix{
        C^\infty_{c_0} (I,\mathcal{M}) \ar@/_ 1cm/[rrr]^{\id_{ C^\infty_{c_0} (I,\mathcal{M})}} \ar[r]^D & C^\infty (I,T\mathcal{M}) \ar[r]^\alpha & C^\infty (I,\g) \ar[r]^{\rho_{c_0}} & C^\infty_{c_0} (I,\mathcal{M})   \\
        &&&&
  }
\end{xy}
\end{displaymath}
Moreover, in the Lie group case we see that the mapping $\alpha \circ D$ maps the submanifold of immersions into the subset $C^\infty (I,\mathfrak{g} \setminus \{ 0\})$. We will require this property in general, as derivatives of immersions should vanish nowhere and this property should be preserved by the transport $\alpha$. 
The next definition details necessary properties of $\alpha$.

\begin{defn}[Square root velocity transform]\label{defn:SRVT}
Let $c_0 \in \mathcal{M}$ be fixed and define the closed submanifold\footnote{As $\mathrm{Imm} (I,\mathcal{M}) \subseteq C^\infty (I,\mathcal{M})$ is open and the evaluation map $\text{ev}_{0} \colon \mathrm{Imm} (I,\mathcal{M}) \rightarrow \mathcal{M}$ is a submersion, $\PP_{c_0} = \text{ev}_0^{-1} (c_0)$ is a closed submanifold of $\mathrm{Imm} (I,\mathcal{M})$ (cf.\ \cite{glockner15fos}).} $\PP_{c_0} := \{ c \in \mathrm{Imm} (I,\mathcal{M}) \mid c(0)=c_0\} = \mathrm{Imm} (I,\mathcal{M}) \cap C^\infty_{c_0} (I,\mathcal{M})$ of $C^\infty (I,\M)$. 
Assume there is a smooth $\alpha\colon C^{\infty}(I,T\M)\rightarrow C^\infty (I,\g)$, such that 
%
\begin{align}
\rho_{c_0} \circ \alpha \circ D = \id_{C^\infty_{c_0} (I,\mathcal{M})} \text{ and }  \label{alpha: cond1} \\
\alpha \circ D (\PP_{c_0}) \subseteq C^\infty (I,\LieA \setminus \{0\}).\label{alpha: cond2}
\end{align} 
Then we define the \emph{square root velocity transform} on $\mathcal{M}$ at $c_0$, with respect to $\alpha$ as
$$ \mathcal{R} \colon \PP_{c_0} \rightarrow C^\infty (I,\g \setminus \{0\}),\quad \mathcal{R}(c):=\frac{\alpha(\dot{c}) }{ \sqrt{ \| \alpha(\dot{c}) \|}},$$
where $\|\cdot\|$ is the norm induced by the right invariant Riemannian metric on the Lie algebra.
We will see in Lemma \ref{lem: SRVT:diff} that $\SRVT$ is smooth.
\end{defn}

The SRVT allows us to transport curves (via $\alpha$) from the homogeneous manifold to curves with values in a fixed vector space (i.e.\ the Lie algebra $\g$). 
{\it The crucial property here is that $\alpha \circ D$ is a right-inverse of $\rho_{c_0}$}, and we note that our construction depends strongly on the choice of the map $\rho_{c_0}$. 

\begin{ex} \label{Liegroupex} Let $\G$ be a Lie group and $H = \{e\}$  the trivial subgroup (with $e$ the Lie group identity). Then $\G = \G/\{e\}$ is a  homogeneous manifold
and $\rho_e = \Evol$. Taking $\alpha(v)=(R_g^{-1})_* (v) $, we reproduce the definition of the SRVT on Lie groups~\ref{setup:SRVT}.
However, contrary to $\Evol$, $\rho_{c_0}$ is not invertible if the subgroup $H$ (with $\mathcal{M} = \G/H$) is non-trivial, but we might still be able to find a right inverse. 
\end{ex}

\begin{ex}
We have $T_uS^2:=\{ v\in \mathbb{R}^3\, | \,  v\cdot u=0\}$ where we have denoted with ``$\,\cdot \,$'' the Euclidean inner product in $\mathbb{R}^3$. Then we can write 
$$v=(vu^T-uv^T)u,\qquad \forall v\in T_uS^2$$ and we can define the map
$$\alpha:v\in T_uS^2 \mapsto vu^T-uv^T\in \mathfrak{so}(3).$$
For $c$ a curve evolving on $S^2$ with $c(0)=e_1$, we have $\rho_{e_1}(\alpha(\dot{c}))=c,$
so $\alpha \circ D$ is the right inverse of $\rho_{e_1}$. The SRVT is then
$$\mathcal{R}(c)=\frac{\dot{c}c^T-c\dot{c}^T}{\sqrt{\|\dot{c}c^T-c\dot{c}^T\|}},$$
and $\| \cdot\|$ is the norm deduced by the usual Frobenius inner product of matrices (the scaled negative Killing form in $\mathfrak{so}(3)$ see table in example~\ref{ex: semisimple}). See section~\ref{examplesmatrices} and~\ref{numericalexperiments}, for further details and more examples.
\end{ex}


The definition of $\alpha$ and the SRVT in Definition \ref{defn:SRVT} depend on the initial point $c_0 \in \mathcal{M}$. 
In many cases our choices of $\alpha$ satisfy \eqref{alpha: cond1} for every $c_0 \in \mathcal{M}$, i.e.\ $\alpha$ satisfies 
  \begin{displaymath}
   \rho (c (0), \alpha (\dot{c})):= \rho_{c(0)} (\alpha (\dot{c})) = c \quad \text{for all } c \in C^\infty (I,\mathcal{M}). 
  \end{displaymath}
Further, the SRVT also depends on the choice of the left-action $\Lambda \colon \G \times \mathcal{M} \rightarrow \mathcal{M}$.
A different action will yield a different SRVT. 
For example, there are several ways to interpret a Lie group as a homogeneous manifold with respect to different group actions.
One of these recovers exactly the SRVT from \cite{celledoni15sao} (see Example \ref{Liegroupex}).
See \cite[Section 5.1]{munthekaas15ioh} for more information on Lie groups as homogeneous spaces, e.g.\ by using the Cartan-Schouten action.

\begin{rem}
Fix $c\in C^\infty (I,\mathcal{M})$ to obtain a smooth map $\Lambda_c \colon C^\infty (I,G) \rightarrow C^\infty (I,\mathcal{M})$, $f \mapsto (t \mapsto \Lambda (f,c) (t)$ \cite[Corollary 11.10 1. and Theorem 11.4]{MR583436}. 
Further we recall from \cite[Theorem 42.17]{MR1471480} that $C^{\infty}(I,T\M)\cong TC^{\infty}(I,\M)$. Identifying the tangent space over the constant  $e \colon I \rightarrow G$ (taking everything to the unit) we obtain 
\[
 T_e \Lambda_{c} \colon C^\infty (I,\g)\rightarrow  T_cC^{\infty}(I,\M), \quad q \mapsto \left( t \mapsto T_e \Lambda_{c(t)} (q(t)) \right).
\]
If $T_e \Lambda_{c}$ was invertible (which it will not be in general), we could use it to define $\alpha$. 
\end{rem}
 

\subsection{Smoothness of the SRVT}\label{smoothness}
One of the most important properties of the square root velocity transform is that it allows us to transport curves from the manifold to curves in the Lie algebra, and this operation is smooth and invertible. The details are summarised in the following two lemmata.
Following \cite[Lemma 3.9]{celledoni15sao}, we consider the smooth scaling maps 
  \begin{equation} \label{eq: scaling}\begin{aligned}
                     \mathrm{sc} \colon C^\infty (I,\mathfrak{g} \setminus \{0\}) &\rightarrow C^\infty (I,\mathfrak{g}\setminus \{0\}),\quad   q \mapsto \left(t \mapsto \frac{q(t)}{\sqrt{\|q(t)\|}}\right),  \\
                     \mathrm{sc}^{-1} \colon C^\infty (I,\mathfrak{g} \setminus \{ 0\})&\rightarrow C^\infty (I,\mathfrak{g} \setminus \{ 0\}),\quad q \mapsto (t \mapsto q(t) \| q(t)\|).
                   \end{aligned}
  \end{equation}

\begin{lem}\label{lem: rhoprop} Fix $c_0 \in \mathcal{M}$, then
\begin{enumerate}
 \item $C^\infty_{c_0} (I,\mathcal{M})$ is a closed and split submanifold\footnote{A submanifold $N$ of a (possibly infinite-dimensional) manifold $M$ is called \emph{split} if it is modeled on a closed subvectorspace $F$ of the model space $E$ of $M$, such that $F$ is complemented, i.e.\ $E = F \oplus G$ as topological vector spaces (see \cite[Section 1]{glockner15fos}).} of $C^\infty (I,\mathcal{M})$,
 \item $\rho_{c_0} \colon C^{\infty}(I,\mathfrak{g}) \rightarrow C^{\infty}_{c_0}(I,\mathcal{M})$ is a smooth surjective submersion. \label{lem: rhoprop_b}
\end{enumerate}
\end{lem}

\begin{proof}
 \begin{enumerate}
  \item Note that $C^{\infty}_{c_0}(I,\mathcal{M})$ is the preimage of $c_0$ under the evaluation map 
  $$\ev_0 \colon C^\infty (I,\mathcal{M}) \rightarrow \mathcal{M},\quad c \mapsto c(0).$$
One can show, similarly to the proof of \cite[Proposition 4.1]{celledoni15sao} that $\ev_0$ is a submersion.
 Hence, \cite[Theorem C]{glockner15fos} implies that $C^\infty_{c_0} (I, \mathcal{M})$ is a closed submanifold of $C^\infty (I,\mathcal{M})$.
 \item Recall that $\rho_{c_0} = \theta_{\Lambda_{c_0}} \circ \Evol$ with $\theta_{\Lambda_{c_0}} \colon C^{\infty} (I,\G) \rightarrow C^\infty (I,\mathcal{M}) , f \mapsto \Lambda_{c_0} \circ f$.
  
 As $\mathcal{M}$ is a homogeneous space, $\pi \colon \G \rightarrow \mathcal{M}$ is a surjective submersion.
 Hence \cite[Chapter 5.1]{MR2021152} implies that $\theta_{\pi} \colon C^\infty (I,\G) \rightarrow C^\infty (I,\mathcal{M})$ is surjective. 
 Further, the Stacey-Roberts Lemma \cite[Lemma 2.4]{1706.04816v1} asserts that $\theta_\pi$ is a submersion.
 Picking $g \in \pi^{-1} (c_0)$, we can also write $\theta_{\Lambda_{c_0}} (f)= \pi \circ R_g \circ f = \theta_\pi (\theta_{R_g} (f))$.
 Thus $\theta_{\Lambda_{c_0}} = \theta_{\pi} \circ \theta_{R_g}$ is a surjective submersion and
\begin{displaymath}
 \theta_{\Lambda_{c_0}}^{-1} (C^\infty_{c_0} (I,\mathcal{M})) = C^\infty_* (I,\G)  = \{c \in C^\infty (I,\G) \mid c(0)=e\}.
\end{displaymath}
 By \cite[Theorem C]{hgreg2015}, $\theta_{\Lambda_{c_0}}$ restricts to a smooth surjective submersion $ C^\infty_* (I,\G) \rightarrow C^\infty_{c_0} (I,\mathcal{M})$. 
Finally, since $\Evol \colon C^\infty (I,\g) \rightarrow C^\infty_* (I,\G)$ is a diffeomorphism (cf.\ \ref{setup: Evol}), $\rho_{c_0} = \theta_{\Lambda_{c_0}} \circ \Evol$ is a smooth surjective submersion.
\qedhere
 \end{enumerate}
\end{proof}
  
\begin{lem}\label{lem: SRVT:diff}
 Fix $c_0 \in \mathcal{M}$ and let $\alpha$ be as in Definition \ref{defn:SRVT}.
 Then the square root velocity transform $\mathcal{R} = \mathrm{sc} \circ \alpha \circ D$ constructed from $\alpha$ is a smooth immersion $\mathcal{R} \colon \PP_{c_0} \rightarrow C^\infty (I,\g \setminus \{0\})$. 
\end{lem}

\begin{proof}
 The map $D \colon C^\infty (I,\mathcal{M}) \rightarrow C^\infty (I,T\mathcal{M}), c \mapsto \dot{c}$ is smooth by Lemma \ref{lem: smoothness}.
 Hence on $\PP_{c_0}$, the restriction of $D$ is smooth. As a composition of smooth maps, $\SRVT = \mathrm{sc} \circ \alpha \circ D|_{\PP_{c_0}}$ is also smooth.
 
 Since $\mathrm{sc} \colon C^\infty (I, \g \setminus \{0\}) \rightarrow C^\infty (I,\g \setminus \{0\})$ is a diffeomorphism, it suffices to prove that $\alpha \circ D|_{\PP_{c_0}}$ is an immersion.
 As we are dealing with infinite-dimensional manifolds, it is not sufficient to prove that the derivative of $\alpha \circ D|_{\PP_{c_0}}$ is injective (which is evident from \eqref{alpha: cond1}).
 Instead we have to construct immersion charts for $x \in \PP_{c_0}$, i.e.\ charts in which $\alpha \circ D$ is conjugate to an inclusion of vector spaces.\footnote{See \cite{glockner15fos} for more information on immersions between infinite-dimensional manifolds.}
 
 To construct these charts, recall from \eqref{alpha: cond1} that $f := \alpha \circ D|_{\PP_{c_0}}$ is a right-inverse to $\rho_{c_0}$.
 In Lemma \ref{lem: rhoprop} we established that $\rho_{c_0}$ is a surjective submersion which restricts to a submersion $\rho_{c_0}^{-1} (\PP_{c_0}) \rightarrow \PP_{c_0}$ by \cite[Theorem C]{glockner15fos}.
 Fix $x \in \PP_{c_0}$ and use the submersion charts for $\rho_{c_0}$. 
 By \cite[Lemma 1.2]{glockner15fos} there are open neighborhoods $x \in U_x \subseteq \PP_{c_0}$ and $f (x) \in U_{f(x)} \subseteq \rho_{c_0}^{-1} (\PP_{c_0})$ together with a smooth manifold $N$ and a diffeomorphism $\theta \colon U_{x} \times N \rightarrow U_{f(x)}$ such that $\rho_{c_0} \circ \theta (u,n)= u$. Thus $\theta^{-1} \circ f|_{U_x} = (\id_{U_x} , f_2)$ for a smooth map $f_2 \colon U_x \rightarrow U_{f_x}$.
 Hence $\theta^{-1} \circ f|_{U_x}$ induces a diffeomorphism onto the split submanifold $\Gamma (f_2) := \{ (y,f_2(y)) \mid y \in U_x\} \subseteq U_x \times U_{f_x}$.  
 Following \cite[Lemma 1.13]{glockner15fos}, we see that $f = \alpha \circ D|_{U_x}$ is an immersion.
 As $x$ was arbitrary, the SRVT $\SRVT$ is an immersion.
\end{proof}

Exploiting that $\SRVT$ is an immersion, we transport Riemannian structures and distances from $C^\infty (I,\g \setminus \{0\})$ to $\PP_{c_0}$ by pullback.
Note that the image of the SRVT for a homogeneous space is in general only an immersed submanifold of ${\nobreak C^\infty (I,\g \setminus \{0\})}$. 
For reductive homogeneous spaces, a certain SRVT will always yield a smooth embedding (see Lemma \ref{lem: SRVT:redemb}).
We investigate now the Riemannian structure on $\PP_{c_0}$.

\subsection{The Riemannian geometry of the SRVT}\label{geometricproperties}

As a first step, we construct a Riemannian metric using the $L^2$ metric on $C^\infty (I,\g)$.

\begin{defn}
 Endow $C^\infty (I,\g)$ with the $L^2$ inner product 
 \begin{displaymath}
  \langle f ,g \rangle_{L^2} = \int_0^1 \langle f(t) , g(t)\rangle \di t,
 \end{displaymath}
where $\langle \cdot , \cdot \rangle$ is induced by the right $H$-invariant Riemannian metric of $\G$ on $\g$.
\end{defn}

The $L^2$ inner product induces a weak Riemannian metric. The $L^2$-geodesics are straight lines, i.e.\ a curve $c(t) \in C^\infty (I,\g)$ is a $L^2$-geodesic if and only if for every $t$, $s \mapsto c(t)(s)$ is a straight line in the vector space $\g$.
In Lemma \ref{lem: SRVT:diff} the square root velocity transform was identified as an immersion, which we now turn into a Riemannian immersion by pulling back the $L^2$ metric. 
Arguing as in the proof of \cite[Theorem 3.11]{celledoni15sao} one obtains the following formula for this pullback metric.

\begin{thm}\label{thm: pbmetric}
 Let $c \in \PP_{c_0}$ and consider $v,w \in T_c \PP_{c_0}$, i.e.\ $v ,w \colon I \rightarrow T\mathcal{M}$ are curves with $v(t),w(t) \in T_{c(t)}\mathcal{M}$.
 The pullback of the $L^2$ metric on $C^\infty(I, \LieA \setminus \{0\})$ under the SRVT to the manifold of immersions $\PP_{c_0}$ is given by: 
  \begin{equation}\label{Eq:ElasticMetric} \begin{aligned}
   G_c^\SRVT(v,w) = \int_I \frac{1}{4} &\left\langle D_s v, u_c\right\rangle \left\langle D_s w, u_c\right\rangle  \\ &+ \left\langle D_s v-u_c \left\langle D_s v,u_c \right\rangle , D_s w-u_c\left\langle D_s w,u_c \right\rangle \right\rangle  \dif s,
   \end{aligned}
  \end{equation}
  where $D_s v := T_c (\alpha \circ D)(v)/ \norm{ \alpha (\dot{c}) }$, $u_c := \alpha (\dot{c})/\norm{\alpha (\dot{c})}$ is the (transported) unit tangent vector of $c$, and $\dif s = \norm{\alpha (\dot{c}(t))} \dif t$.
 The pullback of the $L^2$ norm is given by
  \begin{displaymath}
       G_c^\SRVT (v,v) = \int_I \frac{1}{4} \left\langle D_s v, u_c\right\rangle^2 + \left\|D_s v-u_c\left\langle D_s v, u_c \right\rangle\right\|^2 \dif s.
       \end{displaymath}
\end{thm}

The formula for the pullback metric in Theorem \ref{thm: pbmetric} depends on $\alpha$ and its derivative.
However, notice that we always obtain a first order Sobolev metric which measures the derivative $D_s v$ of the vector field over a curve $c$.   
\smallskip

The distance on $\PP_{c_0}$ will now be defined as the geodesic distance of the first order Sobolev metric $G^{\SRVT}$, i.e.\ of the pullback of an $L^2$ metric. Thus we just need to pull the $L^2$ geodesic distance on $\SRVT (\PP_{c_0})$ 
back using the SRVT.
But, in general, the geodesic distance of two curves on the submanifold $\SRVT (\PP_{c_0})$ with respect to the $L^2$ metric will not be the $L^2$ distance of the curves (see e.g.\ \cite[Section 2]{MR3584579}). 
The question is now, under which conditions is the geodesic distance at least locally given by the $L^2$ distance. 
Note first that the image of the SRVT will in general not be an open submanifold of $C^\infty (I,\g)$ (this was the key argument to derive the geodesic distance in \cite[Theorem 3.16]{celledoni15sao}).
As a consequence we were unable to derive a general result describing the links between the geodesic distance by $G^{\SRVT}$  on $\PP_{c_0}$ and the SRVT algorithmic approach for homogeneous manifolds.
Nonetheless, we conjecture that at least locally the geodesic distance should be given by the $L^2$ distance (note that $\rho_{c_0}^{-1} (\PP_{c_0})$ is an open set, whence the geodesic distance is locally given by the $L^2$ distance). 
On the other hand, for reductive homogeneous spaces (discussed in Section \ref{sect: SRVT:reductive}), an auxiliary map can be used to obtain a geodesic distance which globally coincides with the transformed $L^2$ distance.

\subsection{Equivariance of the Riemannian metric}\label{equivariance}

Often in applications, one is interested in a metric on the shape space 
  \begin{displaymath}
   \SSpace_{c_0} := \PP_{c_0} / \Diff^+ (I) = \text{Imm}_{c_0} (I,\mathcal{M}) / \Diff^+ (I),
  \end{displaymath}
where $\Diff^+ (I)$ is the group of orientation preserving diffeomorphisms of $I$ acting on $\PP_{c_0}$ from the right (cf.\ \cite{bauer_overview_2014}).
To assure that the distance function $d_{\PP_{c_0}}$ descends to a distance function on the shape space, we need to require that it is invariant with respect to the group action.

\begin{defn}
 Let $d \colon \PP_{c_0} \times \PP_{c_0} \rightarrow [ 0 , \infty [$ be a metric. 
 Then $d$ is \emph{reparametrisation invariant} if
 \begin{equation}\label{Eq:ReparametrizationInvariance}
        d(f, h) = d(f \circ \varphi, g \circ \varphi) \quad \forall \varphi \in \Diffeom.
\end{equation}
In other words $d$ is invariant with respect to the diagonal (right) action of $\Diffeom$ on $\PP_{c_0} \times \PP_{c_0}$.
\end{defn}

Let $[f], [g] \in \SSpace$ be equivalence classes and pick arbitrary representatives $f \in [f]$ and $g \in [g]$.
If $d$ is a reparametrisation invariant, we define a metric on $\SSpace$ as
\begin{equation}\label{Eq:DistanceShape}
    d_\mathcal{S}([f], [g]) := \inf_{\varphi \in \Diffeom} d (f, g \circ \varphi).
\end{equation}
Since $d$ is reparametrisation invariant, the definition of $d_\SSpace$ makes sense (cf.\ \cite[Lemma 3.4]{celledoni15sao}).
To obtain a metric on $\SSpace$, we need reparametrisation invairance of $$ d_{\PP_{c_0}} \colon \PP_{c_0} \times \PP_{c_0} \rightarrow \R , \qquad  d_{\PP_{c_0}} (f,g) := \sqrt{\int_{0}^1 \norm{\SRVT (f)(t) - \SRVT(g)(t)}^2 \D t}.$$ 

\begin{lem}\label{lem: repinv}
 Let $\SRVT$ be the square root velocity transform with respect to $c_0 \in \mathcal{M}$ and $\alpha \colon C^\infty (I, T\mathcal{M}) \cong TC^\infty (I, \mathcal{M})  \rightarrow C^\infty (I,\g)$.
 Then $d_{\PP_{c_0}}$ is reparametrisation invariant if $\alpha$ is a $C^\infty (I,\g)$-valued $1$-form on $C^\infty (I,\mathcal{M})$, e.g.\ if $\alpha = \theta_\omega$ for a $\g$-valued $1$-form on $\M$.
 \end{lem}
 \begin{proof}
  Consider $\varphi \in \Diffeom$ and $f,g \in \PP_{c_0}$. Then a computation yields
  \begin{displaymath}
   \SRVT (f\circ \varphi) = \frac{\alpha (\dot{f} \circ \varphi \cdot \dot{\varphi})}{\sqrt{\norm{\alpha (\dot{f} \circ \varphi \cdot \dot{\varphi})}}} = \frac{\alpha (\dot{f} \circ \varphi) \cdot \dot{\varphi}}{\sqrt{\norm{\alpha (\dot{f} \circ \varphi) \cdot \dot{\varphi}}}} = ( \SRVT (f)\circ \varphi) \cdot \sqrt{\dot{\varphi}} ,
  \end{displaymath}
  where we have used that $\alpha$ is fibre-wise linear as a $1$-form.
 Thus we can now compute 
  \begin{displaymath}
      d_{\PP_{c_0}} (f \circ \varphi,g\circ \varphi) =  \sqrt{\int_I \norm{\SRVT (f)\circ \varphi (t) - \SRVT (g) \circ \varphi (t)}^2 \dot{\varphi}(t) \di t} = d_{\PP_{c_0}} (f,g). \qedhere
    \end{displaymath}
 \end{proof}

The condition on $\alpha$ from Lemma \ref{lem: repinv} is satisfied in all examples of the SRVT considered in the present paper.
For example, for a reductive homogeneous case (see Section \ref{sect: SRVT:reductive}), we can always choose $\alpha$ as the pushforward of a $\g$-valued $1$-form. 

\section{SRVT for curves in reductive homogeneous spaces}\label{sect: SRVT:reductive}

A fundamental problem in our approach to shape spaces with values in homogeneous spaces is that we need to somehow lift curves from the homogeneous space to the Lie group.
Ideally, this lifting process should be compatible with the Riemannian metrics on the spaces.
Note that for our purposes it suffices to lift the derivatives of smooth curves to curves in the Lie algebra of the Lie group.
Hence we need a suitable Lie algebra valued $1$-form, which turns out to exist for reductive homogeneous spaces, cf.\ e.g.\ \cite[Chapter X]{MR0238225} (see also \cite{munthekaas15ioh} for a recent account)

\begin{setup}\label{setup: adj} Recall that $\Ad (g) := T_e \text{conj}_{g}$, where $\text{conj}_{g} = L_{g} \circ R_{g^{-1}}$ denotes conjugation $\text{conj}_{g}\colon G\rightarrow G$.
Suppose $\m$ is a subspace of $\g$ such that 
$\g=\h\oplus \m.$\\
	Let $\omega_e \colon T_{eH} \mathcal{M} \rightarrow \m$ be the inverse of $T_e\pi|_{\m} \colon \g \supseteq \m \rightarrow T_{eH}\mathcal{M}$. 	Identify $\g = T_eG$ and observe that $T_e \pi \colon \g \rightarrow T_{eH} \mathcal{M}$ induces an isomorphism $T_e \pi|_{\m} \colon \m \rightarrow T_{eH} \mathcal{M}$.
	
	By definition $\pi \circ R_h = \pi$ holds for all $h \in H$. Now the group actions of $G$ on itself by left and right multiplication commute and we observe that 
	\begin{equation}\label{eq: comm} 
	\text{for all } g \in G \quad \pi \circ L_g = \Lambda^g \circ \pi \quad \text{ and }\,\, T_e \pi \circ \Ad (h) = T\Lambda^h \circ T_e \pi \text{ for } h \in \h.
	\end{equation}

\end{setup}
\begin{setup}

	We will from now on assume that $\mathcal{M}$ is a reductive homogeneous manifold. This means that the subalgebra $\h$ admits a \emph{reductive complement}, i.e.\ a vector subspace $\m \subseteq \g$ such that
	\begin{displaymath}
	\g = \h \oplus \m \text{ and } \Ad (h) . \m \subseteq \m \text{ for all } h \in H.
	\end{displaymath}
	If it exists, a reductive complement will in general not be unique. However, we choose and fix a reductive complement $\m$ for $\h$. 
\end{setup}

\begin{setup}\label{setup: oneform}
	As a reductive complement, $\m$ is closed with respect to the adjoint action of $H$. 
	Hence one deduces (cf.\ \cite[Lemma 4.6]{munthekaas15ioh} for a proof) that $\omega_e$ is $H$-invariant with respect to the adjoint action, i.e.\
	\begin{displaymath}
	\omega_e (T \Lambda^h (v))= \Ad (h). \omega_e (v) \quad \text{for all } v \in T_{eH}\mathcal{M} \text{ and } h \in H.
	\end{displaymath}
	Thus the following map is well-defined: 
	\begin{displaymath}
	\omega \colon T\mathcal{M} \rightarrow \g , \quad v\mapsto \Ad (g).\omega_{e} (T\Lambda^{g^{-1}} (v)) \quad \text{ for all } v \in T_{gH} \mathcal{M}. 
	\end{displaymath}
	From the definition it is clear that $\omega$ is a smooth $\g$-valued $1$-form on $\mathcal{M}$. 
	Moreover, $\omega$ is even $G$-equivariant with respect to the canonical and adjoint action: 
	\begin{equation}\label{eq: equivariance}
	\omega (T\Lambda^k (v)) = \Ad (k). \omega (v) \quad \text{for all } v\in T \mathcal{M} \text{ and } k \in \G.
	\end{equation}
	Note that $\omega$ depends by construction on our choice of reductive complement $\m$. 
	However, we will suppress this dependence in the notation. As noted in \cite[Section 4.2]{munthekaas15ioh}, the $1$-forms $\omega$ correspond bijectively to reductive structures on $\G/H$.\footnote{Note that there might be different reductive structures on a homogeneous manifold. We refer to \cite[Section 5.1]{munthekaas15ioh} for examples and further references.}
\end{setup}

\begin{setup}\label{setup: alphared}
	Let $\omega$ be the $1$-form constructed in \ref{setup: oneform}. 
	Then we define the map 
	$$\Pomega \colon C^\infty (I,T\mathcal{M}) \rightarrow C^\infty (I,\g),\quad f \mapsto \omega \circ f .$$
	Note first that $\Pomega$ is smooth by \cite[Theorem 42.13]{MR1471480}.
	We will prove that $\Pomega$ indeed satisfies \eqref{alpha: cond1} and \eqref{alpha: cond2}, whence $\alpha = \Pomega$ yields an SRVT as in \ref{defn:SRVT}. 
\end{setup}

To motivate the computations, let us investigate an important special case.
%

\begin{ex}\label{ex: LIEGP}
Similarly to example~\ref{Liegroupex}, let $\G$ be a Banach Lie group and $H = \{e\}$ the trivial subgroup. Then $\G = \G/\{e\}$ can be viewed as a reductive homogeneous manifold with $\m = \g$, $\pi = \id_\G$ and $\omega_e = \id_{\g}$.
	From the definition of $\omega$ we obtain $\omega (v) = \Ad(g). (L^{g^{-1}})_* (v) = (R_g^{-1})_* (v) = \kappa^r (v),$
	where $\kappa^r$ denotes the right Maurer-Cartan form, \cite[Section 38]{MR1471480} or \cite[Section 5.1]{munthekaas15ioh}.
	In particular, for $c \colon I \rightarrow \G$ we have $\SRp (c)= \kappa^r (\dot{c}) = \delta^r (c)$ (right logarithmic derivative).
	As we have $\Evol \circ \,\delta^r (c) = c$ for a curve starting at $e$.
	
	The SRVT for reductive spaces coincides thus with the SRVT for Lie group valued shape spaces as outlined in \ref{setup:SRVT}.
\end{ex}

Albeit Example \ref{ex: LIEGP} is quite trivial as a homogeneous space, it highlights a general principle of the construction for reductive homogeneous spaces.

\begin{rem}
We here provide an alternative interpretation for $\Pomega \circ D$: 
A smooth curve $c \colon I \rightarrow \mathcal{M}$ admits a smooth horizontal lift $\tilde{c} \colon I \rightarrow G$ depending on a choice of connection for the principal bundle $G \rightarrow \mathcal{M}$ \cite[Chapter 5.1]{MR2021152}. 
For a reductive homogeneous manifold we construct a horizontal lift $\tilde{c}$ using the canonical invariant connection (depending on the reductive complement, see \cite[X.2]{MR0238225}). 
Now we take the (right) Darboux derivative (aka right logarithmic derivative) of $\tilde{c} \colon I \rightarrow G$ (see \cite[3.\S 5]{Sharpe}). Then unraveling the definitions similar to Examples \ref{Liegroupex} and \ref{ex: LIEGP}, one can show that $\delta^r (\tilde{c}) = \Pomega \circ D (c)$ holds for the $1$-form $\Pomega$ as in \ref{setup: alphared}.
Thus for a reductive homogeneous space the proposed SRVT can be viewed (up to scaling) as the Darboux derivative of a horizontal lift of a curve in $\mathcal{M}$. Note that this interpretation justifies again to view $\rho_{c_0}$ as a generalised version of the evolution operator $\Evol$ (which inverts the right logarithmic derivative, see Remark \ref{rem: gen:EVOL}).
\end{rem}

A rich source for reductive homogeneous spaces are quotients of semisimple Lie groups. We recall now some of the main examples. 

\begin{ex}\label{ex: semisimple}
	Let $G$ be a semisimple Lie group and $H$ a Lie subgroup of $G$ which is also semisimple. Then the homogeneous space $\mathcal{M} = G/H$ is reductive. 
	A reductive complement of $\mathfrak{h}$ in $\mathfrak{g}$ is the orthogonal complement $\mathfrak{h}^\perp$ with respect to the Cartan-Killing form on $\mathfrak{g}$ (recall that the Killing form of a semisimple Lie algebra is non-degenerate by Cartan's criterion \cite[I.\S 7 Theorem 1.45]{MR1920389}).
	For example, this occurs for $G = \SL (n)$ and $H  = \SL(n-p)$ or $G = \SO (n)$ and $H= \SO(n-p)$ (where $1 \leq p < n$), since by \cite[I.\S 8 and I.\S 18]{MR1920389} the following properties hold:
	\begin{center}
		\begin{tabular}{|c | c | c | c | }
		   \hline {\bf Lie group $\G$ } & {\bf compact? } & {\bf semisimple?} &  {\bf Killing form $B(X,Y)$ on $\g$} \\
			\hline
			$\SO (n)$ & yes & yes (for $n\geq 3$) & $(n-2)\text{Tr} (XY)$\\ \hline
			$\SL (n)$ & no  & yes 			& $2n \text{Tr}(XY)$\\ \hline
			$\GL (n)$ & no  & no & $2n \text{Tr}(XY) - 2 \text{Tr}(X)\text{Tr}(Y)$\\ \hline
		\end{tabular}
	\end{center}
	Here $\text{Tr}$ denotes the trace of a matrix. All main examples in this paper are reductive.
\end{ex}
\begin{prop}\label{prop: alphaomega:prop}
	Let $\mathcal{M} =\G/H$ be a reductive homogeneous space, $c_0 \in \mathcal{M}$, $\omega$ and $\Pomega$ as in \ref{setup: alphared}. 
	Consider $D \colon C^\infty_{c_0} (I ,\mathcal{M} ) \rightarrow C^\infty (I,T\mathcal{M}),  c \mapsto \dot{c}$.
	Then  
	\begin{displaymath}
	\rho_{c_0} \circ \Pomega \circ D = \id_{C^\infty_{c_0} (I,\mathcal{M})}.
	\end{displaymath}
\end{prop}

\begin{proof}
	As a shorthand write $\SRp := \Pomega \circ D$.
	We establish in Lemma \ref{lem: identity} the identity 
	\begin{equation}\label{claim}
	\id_{C^\infty_{eH} (I,\mathcal{M})}= \rho_{eH} \circ \SRp = \Lambda_{eH} \circ \Evol \circ \SRp = \pi \circ \Evol \circ \SRp.
	\end{equation} 
	Let now $c \in C^\infty_{c_0} (I,\mathcal{M})$ with $c_0 = g_0 H$. 
	Then we obtain $\Lambda^{g_0^{-1}} \circ c \in C^\infty_{eH} (I,\mathcal{M})$ and   
	\begin{align*}
	\rho_{c_0} \circ \SRp (c) &\stackrel{\hphantom{\eqref{eq: equivariance}}}{=}  (\Lambda_{c_0} \circ \Evol) \circ \Pomega (\dot{c}) = \Lambda_{c_0} \circ \Evol \circ \,\omega (T\Lambda^{g_0} T\Lambda^{g_0^{-1}}\dot{c})\\
	&\stackrel{\eqref{eq: equivariance}}{=} \Lambda_{c_0} \circ \Evol ( \Ad (g_0).\omega (T\Lambda^{g_0^{-1}}\dot{c})) = \Lambda_{c_0} \circ \Evol ( \Ad (g_0).\SRp (\Lambda^{g_0^{-1}} \circ c)).
	\end{align*}
	Recall from \cite[1.16]{hgreg2015} that for a Lie group morphism $\varphi$ one has the identity $\Evol \circ \Lf(\varphi) = \varphi \circ \Evol$.
	By definition, $\Ad (g) = \Lf (\text{conj}_{g}) := T_e \text{conj}_{g}$, where $\text{conj}_{g} = L_{g} \circ R_{g^{-1}}$ denotes the conjugation morphism.
	Insert this into the above equation: 
	\begin{align*}
	\rho_{c_0} \circ \SRp (c) &=  \Lambda_{c_0} \circ \Evol \circ \SRp (c) = \Lambda_{c_0} \circ L_{g_0} \circ R_{g_0^{-1}} \circ \Evol (\SRp (\Lambda^{g_0^{-1}} \circ c))\\
	&= \pi \circ L_{g_0}\Evol (\SRp (\Lambda^{g_0^{-1}} \circ c)) = \Lambda^{g_0} \circ \pi \circ \Evol (\SRp (\Lambda^{g_0^{-1}} \circ c))\\
	&\stackrel{\eqref{claim}}{=}\Lambda^{g_0} \circ \Lambda^{g_0^{-1}} \circ c =c.
	\end{align*}
	In passing to the second line we used that left and right multiplication maps commute and that $\Lambda_{c_0} (R_{g_0^{-1}} (k)) = \Lambda_{c_0} (kg_0^{-1}) = kg_0^{-1}c_0 = kg_0^{-1}g_0H = \pi (k)$.
\end{proof}

\begin{prop}\label{prop: defn:SRVT:hom}
	Let $\mathcal{M} =\G/H$ be a reductive homogeneous space, $c_0 \in \mathcal{M}$, $\omega$ and $\Pomega$ as in \ref{setup: alphared}. 
	Then $\Pomega$ satisfies \eqref{alpha: cond1} and \eqref{alpha: cond2}, whence for a reductive homogeneous space we can define the SRVT as 
	\begin{displaymath}
	\SRVT (c) := \frac{\Pomega (\dot{c})}{\sqrt{\norm{\Pomega (\dot{c})}}} \quad \text{for } c \in \text{\upshape Imm} (I,\mathcal{M})
	\end{displaymath}
\end{prop}

\begin{proof}
	In Proposition \ref{prop: alphaomega:prop} we have already established \eqref{alpha: cond1}.
	To see that \eqref{alpha: cond2} also holds for $\Pomega$, observe first that for $v \in T_{gH} \mathcal{M}$, we have 
	$\omega (v) = \Ad (g). \omega_e (T\Lambda^{g^{-1}} (v))$.
	Since $\omega_e \circ T\Lambda^{g^{-1}} \colon T_{gH} \mathcal{M} \rightarrow \m$ and $\Ad (g) \colon \g \rightarrow \g$ are linear isomorphisms, we see that $\omega (v) = 0$ if and only if $v=0_{gH}$.
	As $\Pomega$ is post-composition by $\omega$, $\Pomega$ satisfies \eqref{alpha: cond2}. 
\end{proof} 

\subsection{Riemannian geometry and the reductive SRVT}

In the reductive space case, it is easier to describe the image of the square root velocity transform. 
It turns out that the image is a split submanifold with a global chart. Using this chart, we can also obtain information on the geodesic distance.

The idea is to transform the image of the SRVT such that it becomes $C^\infty (I, \m \setminus \{0\})$, where $\m$ is again the reductive complement. 
Pick $g_0 \in \pi^{-1}(c_0)$ and use the adjoint action of $G$ and the evolution $\Evol \colon C^\infty (I,\g) \rightarrow C^\infty (I,G)$ to define 
\begin{displaymath}
\Psi_{g_0} (q) := -\Ad (g_0\Evol (q)^{-1}).q \quad \text{ for } q \in C^\infty (I,\g)
\end{displaymath}
where the dot denotes pointwise application of the linear map $\Ad (\Evol(q)^{-1})$.
Then $\Psi_{g_0}$ is a diffeomorphism with inverse $\Psi_{g_0^{-1}}$ (see Lemma \ref{lem: PSi0diff}). 
We will now see that $\Psi_{g_0^{-1}}$ maps the image of the SRVT to $C^\infty (I, \m \setminus \{0\})$.

\begin{lem}\label{lem: SRVT:redemb}
	Choose $c_0 \in \mathcal{M}$ in the reductive homogeneous space $\M$, and let $\omega$ and $\Pomega$, $D$ be as in Proposition \ref{prop: alphaomega:prop}.
	Then $\textup{Im } \Pomega \circ D$ is a split submanifold of $C^\infty (I,\g \setminus \{0\})$ modelled on $C^\infty (I,\m)$ and $\Pomega \circ D$ is a smooth embedding.
	In particular, $\SRVT (\PP_{c_0}) = \Psi_{g_0} (C^\infty (I, \m \setminus \{0\}))$ is a split submanifold of $C^\infty (I,\g \setminus \{0\})$ and $\SRVT$ is a smooth embedding.
\end{lem}

\begin{proof}
	As $\g = \h \oplus \m$, we have $C^\infty (I,\g) = C^\infty (I,\h\oplus \m) \cong C^\infty (I,\h) \oplus C^\infty (I,\m)$. 
	Thus $C^\infty (I,\m \setminus \{0\})$ is a closed and split submanifold of $C^\infty (I,\g \setminus \{0\})$.
	Fix $g_0 \in G$ with $\pi(g_0)=c_0$ and note that $\Psi_{g_0}$ restricts to a diffeomorphism $C^\infty (I,\g \setminus \{0\}) \rightarrow C^\infty (I,\g \setminus \{0\})$ by Lemma \ref{lem: PSi0diff}.
	Now as $\Psi_{g_0} (C^\infty (I,\m\setminus \{0\}) ) = \text{Im } \Pomega \circ D$ (cf.\ Lemma \ref{lem: straightening}), the image $\textup{Im } \Pomega \circ D$ is a closed and split submanifold of $C^\infty (I,\g \setminus \{0\})$.
	Further, we deduce from Lemma \ref{lem: straightening} that $\rho_{c_0}|_{\text{Im} \Pomega \circ D}$ is smooth with $\Pomega \circ D \circ \rho_{c_0}|_{\text{Im} \Pomega \circ D} = \id_{\text{Im} \Pomega \circ D}$.
	As also $\rho_{c_0} \circ \Pomega = \id_{\text{Imm}_{c_0} (I,\mathcal{M})}$, we see that $\Pomega$ is a diffeomorphism onto its image.
	Thus $\Pomega \circ D$ is indeed a smooth embedding.
	
	Since the scaling maps are diffeomorphisms $C^\infty (I,\g \setminus \{0\}) \rightarrow C^\infty (I,\g \setminus \{0\})$, the assertions on the image of $\SRVT$ and on $\SRVT$ follow directly from the assertions on $\Pomega$. 
\end{proof}

\
\begin{setup}[Reductive SRVT]\label{setup: twist}
	Let $\M$ be a reductive homogeneous space with reductive complement $\m$ and $\theta_\omega \colon C^\infty (I,T\M) \rightarrow C^\infty (I,\g),\ f \mapsto \omega \circ f$ be constructed with respect to the $1$-form $\omega$ from \ref{setup: oneform}.
	Then $\Psi_{g_0^{-1}} \circ \theta_\omega (\PP_{c_0}) = C^\infty (I,\m \setminus \{0\})$ (see Appendix \ref{app: aux}).
	Now one constructs a version of the SRVT for reductive spaces via 
	\begin{displaymath}
	\SRVT_\m \colon \PP_{c_0} \rightarrow C^\infty (I, \m \setminus \{0\}) , \quad f\mapsto \frac{\Psi_{g_0^{-1}} \circ \theta_\omega (\dot{f})}{\sqrt{\lVert\Psi_{g_0^{-1}} \circ \theta_\omega (\dot{f})\rVert}}.
	\end{displaymath}
	We call this map \emph{reductive SRVT}, to distinguish it from the usual SRVT. Contrary to the SRVT, the reductive SRVT will go into the reductive complement, but it will not be a section of $\rho_{c_0}$. Instead it is a section of $\rho_{c_0} \circ \Psi_{g_0}$. 
	Finally, we note that by construction (cf.\ Lemma \ref{lem: PSi0diff}) the image of the reductive SRVT is $C^\infty (I,\m \setminus \{0\})$.
\end{setup}

Arguing as in Theorem \ref{thm: pbmetric}, we also obtain a first order Sobolev metric by pullback with the reductive SRVT. 
In general this Riemannian metric will not coincide with the pullback metric obtained from the SRVT.
The advantage of the reductive SRVT is that we have full control over its image, which happens to be an open subset (of a subspace of $C^\infty (I,\g)$).
Since $C^\infty (I,\g)$ with respect to the $L^2$ inner product is a flat space (in the sense of Riemannian geometry), it follows that at least locally the geodesic distance on the image of the SRVT is given by the distance 
$$d_{\PP_{c_0},\m} (f,g) := d_{L^2} (\SRVT_\m(f),\SRVT_\m (g)).$$
However, we argue as in \cite[Theorem 3.16]{celledoni15sao} to obtain the following result.

\begin{prop}\label{prop: gL2dist}
	If $\mathrm{dim}\ \mathfrak{h} + 2 < \mathrm{dim}\ \g$, then the geodesic distance of $(\SRVT (\PP_{c_0}), \langle \cdot,\cdot\rangle_{L^2})$ coincides with the $L^2$ distance.
	In this case the geodesic distance on $\PP_{c_0}$ induced by the pullback metric \eqref{Eq:ElasticMetric} (with respect to the reductive SRVT) is given by 
	$
	d_{\PP_{c_0},\m} (f,g) = \sqrt{\int_I \norm{\SRVT_\m(f)(t) - \SRVT_\m (g)(t)}^2\di t} .
	$
\end{prop}

Note that the modification by the reductive SRVT is highly non-linear, e.g.\ in the Lie group case, Example \ref{ex: LIEGP}, we obtain:
\begin{ex}
	Let $G$ be a Lie group, $c \in ^\infty (I, G)$ and $\delta^l (c) = c^{-1}\dot{c}$. Then 
	\begin{displaymath}
	\Psi (\delta^r (c)) = -\Ad (\Evol (\delta^r (c))^{-1}).\delta^r (c) = -\Ad (c^{-1}). \dot{c} c^{-1} = - \delta^l (c).
	\end{displaymath}
	Recall from \cite[38.4]{MR1471480} that $\Evol (-\delta^l (c))(t) = (c(t))^{-1}$. 
	In the Lie group case, the reductive SRVT modifies the formulae to compute distances and interpolations between the pointwise inverses of curves instead of the curves themselves. 
	In particular, this shows that the reductive SRVT will not be a section of $\rho_{c_0}$.  
\end{ex}

In particular, we have to prove a version of Lemma \ref{lem: repinv} for the reductive SRVT.

\begin{lem}\label{lem: red:repinv}
	For a reductive space, $d_{\PP_{c_0},\m}$ is reparametrisation invariant.
\end{lem}

\begin{proof}
	For $\SRVT_\m$ we use $\Psi_{g_0^{-1}} \circ \theta_\omega$ instead of $\alpha = \theta_\omega$. Consider $f \in \PP_{c_0}$ and $\varphi \in \Diffeom$ to compute as in Lemma \ref{lem: repinv}: 
	$\Psi_{g_0}^{-1} (\theta_\omega (f\circ \varphi)) = \Psi_{g_0}^{-1} (\dot{\varphi} \cdot \theta_\omega (f) \circ \varphi)$.
	Now $$\Evol (q) \circ \varphi  = \Evol (\dot{\varphi} \cdot q\circ \varphi)\underbrace{\Evol (q) (\varphi (0))}_{=e \text{ since }\varphi(0)=0} = \Evol (\dot{\varphi} \cdot q\circ \varphi)$$
	follows from \cite[p.\ 411]{MR1471480}. Linearity of the adjoint action yields $\Psi_{g_0}^{-1} (\theta_\omega (f\circ \varphi)) = (\Psi_{g_0}^{-1} (\theta_\omega (f)) \circ \varphi )\cdot \dot{\varphi}$. 
	Inserting this in \eqref{eq: equivariance} yields reparametrisation invariance.
\end{proof}

\section{The SRVT on matrix Lie groups}\label{examplesmatrices}

In order to illustrate our definition of the SRVT in different instances of homogeneous manifolds, we consider in what follows two examples of quotients of finite dimensional matrix Lie groups (for $n\geq 3$ and $p <n$):
\begin{itemize}
	\item[1.] $\SO (n) /(\SO (n-p)\times \SO (p))$ (see \ref{ss: Grass}). 
	\item[2.] $\SO (n) /\SO (n-p)$ (see \ref{ss: Stiefel}).
\end{itemize}
Note that in both cases the quotients are reductive homogeneous spaces.
To prepare our investigation, we will now collect some information on relevant tangent spaces for the matrix Lie groups.
These examples are relevant in applications \cite{absil08oao}. 

\subsection{Tangent space of $G/H$ and tangent map of $G\rightarrow G/H$}

For  $G$ and $H$ finite dimensional (matrix) Lie groups, we here describe the tangent space of $G/H$ at a prescribed point  $c_0$ and the tangent mapping of the canonical projection $\pi:G\rightarrow G/H$. We have seen that any curve $c(t)$ on $G/H$, $c(0)=c_0$, can be expressed non-uniquely by means of  a curve on the Lie group $c(t)=\pi(g(t))$. 
For matrix Lie groups, the elements of $G/H$ are equivalence classes of matrices. 
Let the elements of $G$, $g\in G$, be $n\times n$ matrices, then the group multiplication coincides with matrix multiplication. We
 identify elements of $H\subset G$ with matrices
 \begin{equation}
 \label{isotropyGroup}
 h=\left[ \begin{array}{cc}
I & 0\\
0 & \Gamma
\end{array} \right] , \quad \text{$\Gamma$ a $(n-p)\times(n-p)$ matrix and $I$ the $p\times p$ identity.}
\end{equation} 



We obtain $T_{g_0}\pi:T_{g_0}G\rightarrow T_{\pi(g_0)}G/H$, $v\mapsto w$, by differentiating $c(t)=\pi(g(t))$. Assuming $g(0)=g_0$, $\pi(g_0)=c_0$, $\dot{g}(0)=v\in T_{g_0}G$,  we have 
\begin{displaymath}
w:=T_{\pi(g_0)}(v)=\left. \frac{d}{dt}\right|_{t=0}\pi(g(t))=\left\{ \left. \frac{d}{dt}\right|_{t=0}\tilde{g}(t)\, |\, \tilde{g}(t)=g(t)\, h(t), \quad h (t)\in H\right\}.
\end{displaymath}
Assuming $\dot{g}(t)=A(t)g(t)$, where $A(t) \in \mathfrak{g}$,  $v=A_0g_0=g_0\,\mathrm{Ad}_{g_0^{-1}}(A_0)$,   and  assuming also that $\frac{d}{dt} h(t)=B(t)h(t), \, B(t)\in \mathfrak{h},\quad B(0)=B_0$, in analogy to \eqref{log: deriv}, we get 
\begin{equation}
\label{diffcurveinM}
 \quad \frac{d}{dt} \tilde{g}(t)= 
 \left(A(t)+\mathrm{Ad}_{g(t)}(B(t)) \right)\,g(t)h(t)=g(t)\left(\mathrm{Ad}_{g(t)^{-1}}(A(t))+B(t) \right)\,h(t),
 \end{equation}
so we obtain  
\begin{eqnarray*}
w:=T_{\pi(g_0)}(v)&=&
 \left\{ \tilde{w}=\left. \frac{d}{dt}\right|_{t=0}\tilde{g}(t) \,|\, \tilde{w}=(A_0+ \mathrm{Ad}_{g_0}(B_0 ))\,g_0\, h, \quad h\in H, B_0\in\mathfrak{h}\right\} \\
&=& \left\{ \tilde{w}=\left. \frac{d}{dt}\right|_{t=0}\tilde{g}(t) \,|\, \tilde{w}=g_0\,(\mathrm{Ad}_{g_0^{-1}}(A_0)+B_0)\, h,\quad h\in H, B_0\in\mathfrak{h}\right\}, 
\end{eqnarray*}
which gives a  description of the tangent vector $w\in T_{c_0}G/H$ as well as the characterisation of $T\pi$ for matrix Lie groups. 
Suppose that we fix a complementary subspace $\mathfrak{m}$ of $\mathfrak{h}$, $\mathfrak{g}=\mathfrak{h}\oplus\mathfrak{m}$, then there is a unique isotropy element $B_0\in\mathfrak{h}$ such that $\mathrm{Ad}_{g_0^{-1}}(A_0)+B_0\in \mathfrak{m}$. 

Repeating this procedure for each value of $t$ along a curve $c(t)$, we can assume $c(t)=\pi (g(t))$ and $w(t) \in T_{\pi(g(t))}G/H$, $w(t)=(A(t)+\mathrm{Ad}_{g(t)}(B(t)) )c(t)$ with $A(t)\in\mathfrak{g}$, $B(t)\in \mathfrak{h}$, such that $\mathrm{Ad}_{g(t)^{-1}}(A(t))+B(t)\in \mathfrak{m}$, then we can define 
$$\alpha : T_{\pi(g(t))}G/H\rightarrow \mathrm{Ad}_{g(t)}(\mathfrak{m}),\qquad \alpha (w(t))=A(t)+\mathrm{Ad}_{g(t)}(B(t)).$$
This map corresponds to the map $\theta_{\omega}$ of \ref{setup: alphared} with $\omega$ as described in \ref{setup: oneform}.
If $\m$ is reductive, this map is well defined (independently of the choice of representative $g(t)$ of $c(t)=\pi (g(t))$).
We refer to Table~\ref{tab:table1} for different, possible choices of $\mathfrak{m}$ and their implications. In the following examples $\m$ is reductive and $H$ is compact.

\subsection{SRVT on the Stiefel manifold: $\SO (n) / \SO (n-p)$.}\label{ss: Stiefel}

In this section we consider the case when $G=\mathrm{SO}(n)$ and $H=\mathrm{SO}(n-p)\subset \mathrm{SO}(n)$, where the elements of $\mathrm{SO}(n-p)$ are  of the type \eqref{isotropyGroup} with $\Gamma$ a $(n-p)\times (n-p)$ orthogonal matrix with determinant equal to $1$. We consider the canonical left action of 
$\mathrm{SO}(n)$ on the quotient $\mathrm{SO}(n) / \mathrm{SO}(n-p)$. This homogeneous manifold can be identified with the Stiefel manifold, $\M=\StfR$, i.e. the set of $p$-orthonormal frames in $\mathbb{R}^n$ (real matrices $n\times p$ with orthonormal columns). We will in the following denote by $[U,U^{\perp}]$ the elements of $\mathrm{SO}(n)$ where we have collected in $U$ the first $p$ orthonormal columns and in $U^{\perp}$ the last $n-p$. Multiplication from the right by an arbitrary element in the isotropy subgroup  $\mathrm{SO}(n-p)$ gives $[U,U^{\perp}\Gamma]$, leaving the first $p$ columns unchanged and orthonormal to the last $n-p$, for all choices of $\Gamma$. Here $U$ alone represents the whole coset of $[U,U^{\perp}]$.
When thought of as a map from $\mathrm{SO}(n)$ to $\mathrm{SO}(n)/\mathrm{SO}(n-p)$, the projection $\pi\colon\mathrm{SO}(n)\rightarrow \mathrm{SO}(n)/\mathrm{SO}(n-p)$ is 
$$\pi([U,U^{\perp}])=\{\tilde{g}\in \mathrm{SO}(n)\, | \, \tilde{g}=[U,U^{\perp}\Gamma],\quad \forall \, \Gamma\in \mathrm{SO}(n-p)\}.$$
Otherwise, when thought of as a map from $\pi:\mathrm{SO}(n)\rightarrow \StfR$, the canonical projection  
conveniently becomes
$\pi([U,U^{\perp}])=[U,U^{\perp}]\,I_p=U,$
where $I_p$ is the $n\times p$ matrix whose columns are the first $p$ columns of the $n\times n$ identity matrix. The equivalence class of the group identity element $\pi(e)$ is  identified with the $n\times p$ matrix $I_p$. Similarly the tangent mapping of the projection $\pi$,  
{
$$T\pi: T\mathrm{SO}(n)\rightarrow T\mathrm{SO}(n)/\mathrm{SO}(n-p),\qquad v\in T_g\mathrm{SO}(n)\mapsto w\in T_{\pi(g)}\mathrm{SO}(n)/\mathrm{SO}(n-p),$$
}
 with $g=[ U , U^{\perp} ]$, $v=A\, [ U , U^{\perp} ] \in T_{[U,U^{\perp}]}\mathrm{SO}(n)$ and $A\in \mathfrak{so}(n)$), can be realised as

\begin{equation}
\label{quotientSO} 
 T_{[U,U^{\perp}]}\pi(A\,[U,U^{\perp}]) =\left\{\tilde{w}\in T_{[U,U^{\perp}\Gamma ]}\SO(n)\, \middle|\, \begin{aligned} \tilde{w}=[U,U^{\perp}](\mathrm{Ad}_{[U,U^{\perp}]^T}(A)+B)\,\Gamma ,\\ \forall \,\Gamma \in\mathrm{SO}(n-p),\, B\in\mathfrak{so}(n-p) \end{aligned}\right\},
\end{equation}

while $T\pi: T\mathrm{SO}(n)\rightarrow T\StfR$
by multiplication from the right by $I_p$, and 
\begin{equation}
\label{tanvecStiefel}
T_{[U,U^{\perp}]}\pi(A\,[U,U^{\perp}])=A\,[U,U^{\perp}]\,I_p=AU\in T_U \StfR .
\end{equation}
Alternatively, a characterisation of tangent vectors can be obtained by differentiation of curves on $\StfR$. 
We
have then that
$$T_Q\StfR=\{ V \hskip0.2cm n\times p \hskip0.2cm \mathrm{matrix}\hskip0.1cm|\hskip0.1cm
Q^{\mathrm{T}}V\hskip0.2cm p\times p
\hskip0.2cm \text{
    skew-symmetric} \}.$$

\begin{prop}\cite{celledoni03oti}
Any tangent vector $V$ at $Q\in\StfR$ can be written as
\begin{align}
\label{eq2.1}%
V&=(FQ^{\mathrm{T}}-QF^{\mathrm{T}})\, Q,\\ 
F:&=V-Q\,\frac{Q^{\mathrm{T}}V}{2}\in T_Q\M.
\end{align}
\end{prop}
\noindent And notice that replacing $F$ with $F:=V-Q(\frac{Q^{\mathrm{T}}V}{2}+S)$, where $S$ is an arbitrary $p\times p$ symmetric matrix, does not affect \eqref{eq2.1}.

We proceed by using the representation (\ref{eq2.1}) of $T_Q\StfR$  and the framework described in Definition \ref{defn:SRVT} for
defining an SRVT on the Stiefel manifold. 
Consider
\begin{align}
\label{SRVT1}
f_Q:T_Q\M\rightarrow T_Q\M, &&
f_Q(V)&=V-Q\frac{Q^{\mathrm{T}}V}{2}, \\ 
\label{SRVT2}
a_Q:T_Q\M\rightarrow \mathfrak{m}_Q\subset \g, &&
a_Q(V)&=f_Q(V)Q^{\mathrm{T}}-Qf_Q(V)^{\mathrm{T}}.
\end{align}
The SRVT of a curve $Y(t)$ on the Stiefel manifold is a curve on $\mathfrak{so}(n)$ defined by
\begin{equation}
\label{SRVTStiefel}
\mathcal{R}(Y):=\frac{a_{Y}(\dot{Y})}{\sqrt{\|a_{Y}(\dot{Y})\|}}=\frac{f_Y(\dot{Y})Y^{\mathrm{T}}-Yf_Y(\dot{Y})^{\mathrm{T}} }{\sqrt{\|f_Y(\dot{Y})Y^{\mathrm{T}}-Yf_Y(\dot{Y})^{\mathrm{T}}\|}}.
\end{equation}

As the Stiefel manifold is a reductive homogeneous space, we can define a reductive SRVT in this case.
Denoting with $[Q,Q^{\perp}]$ a representative in $SO(n)$ of the equivalence class identified by $Q$ on $\StfR$,
we observe that
$$V=\mathrm{Ad}_{[Q Q^{\perp}]}(G)I_p\quad \text{with}\quad G:=[Q Q^{\perp}]^TFI_p^T-I_pF^T[Q Q^{\perp}].$$
Assuming the right invariant metric on $SO(n)$ is the negative Killing form, then we observe that $G$ belongs to the orthogonal complement of the subalgebra $\mathfrak{so}(n-p)$ in $\mathfrak{so}(n)$ with respect to this inner product. As stated in Table~\ref{tab:table1}, this orthogonal complement is the reductive complement, i.e. $\mathfrak{m}=\mathfrak{so}(n-p)^{\perp}$, and $\mathrm{Ad}_{\SO(n-p)}(\mathfrak{so}(n-p)^{\perp})\subset \mathfrak{so}(n-p)^{\perp}$. The elements of such an orthogonal complement $\mathfrak{so}(n-p)^{\perp}$ are  matrices $W\in \mathfrak{so}(n)$ of the form
\begin{equation}
\label{reductiveStiefel}
W=\left[  \begin{array}{cc}
\Omega & \Sigma^T\\
-\Sigma & 0
\end{array}\right],
\end{equation}
with $\Omega \in \mathfrak{so}(p)$ and $\Sigma$ an arbitrary $(n-p)\times p$ matrix. 
Consider  the maps
\begin{align}
\label{SRVTm1}
\tilde{f}_Q:T_Q\M\rightarrow T_{I_p}\M, &&
\tilde{f}_Q(V)&=[Q Q^{\perp}]^TV-I_p\frac{Q^{\mathrm{T}}V}{2}, \\ 
\label{SRVTm2}
\tilde{a}_Q:T_Q\M\rightarrow \mathfrak{m}\subset \g, &&
\tilde{a}_Q(V)&=\tilde{f}_Q(V)I_p^{\mathrm{T}}-I_p\tilde{f}_Q(V)^{\mathrm{T}},
\end{align}
and we observe that $\tilde{a}_Q(V)\in\mathfrak{m}$. Then the reductive SRVT is
\begin{equation}
\label{SRVTStiefelred}
\mathcal{R}_{\mathfrak{m}}(Y):=\frac{\tilde{a}_{Y}(\dot{Y})}{\sqrt{\|\tilde{a}_{Y}(\dot{Y})\|}}=\frac{\tilde{f}_Y(\dot{Y})I_p^{\mathrm{T}}-I_p\tilde{f}_Y(\dot{Y})^{\mathrm{T}} }{\sqrt{\|\tilde{f}_Y(\dot{Y})I_p^{\mathrm{T}}-I_p\tilde{f}_Y(\dot{Y})^{\mathrm{T}}\|}}.
\end{equation}

\subsection{SRVT on the Grassmann manifold:  $\mathrm{SO} (n) / (\mathrm{SO} (n-p)\times \mathrm{SO} (p)) $.}\label{ss: Grass}

In this section we consider the case when $G=\mathrm{SO}(n)$ and $H=\mathrm{SO}(n-p)\times \mathrm{SO}(p) \subset \mathrm{SO}(n)$ where the elements of $\mathrm{SO}(n-p)\times \mathrm{SO}(p)$ are  of the type 
\begin{equation}
 \label{isotropyGroupGR}
h=\left[ \begin{array}{cc}
\Lambda & 0\\
0 & \Gamma
\end{array} \right],
\end{equation}
with $\Lambda$ a $p\times p$ matrix and $\Gamma$ an $(n-p)\times (n-p)$ matrix, both orthogonal with determinant equal to $1$. We consider the canonical left action of 
$\mathrm{SO}(n)$ on the quotient $\mathrm{SO}(n) / (\mathrm{SO}(n-p)\times \mathrm{SO}(p) )$.  This homogeneous manifold can be identified with a quotient of the Stiefel manifold $\StfR / \mathrm{SO}(p) $ with equivalence classes $[Q]=\{\tilde{Q}\in \StfR \, |\, \tilde{Q}=Q\,\Lambda, \,\,\Lambda\in \mathfrak{so}(p)\}$. 
We denote such a manifold here with $\mathbf{G}_{p,n}(\mathbb{R})$\footnote{An alternative  representation of $\mathbf{G}_{p,n}$ is given by considering symmetric matrices $P$, $n\times n$, with $\mathrm{rank}(P)=p$ and $P^2=P$, \cite{huper07otg}.}. 
The reductive subspace is $\mathfrak{m}=(\mathfrak{so}(p)\times \mathfrak{so}(n-p))^{\perp}$ with elements as in \eqref{reductiveStiefel} but with $\Omega=0$.
Imposing a choice of isotropy $B\in \mathfrak{so}(p)\times \mathfrak{so}(n-p)$ such that $(\mathrm{Ad}_{[Q,Q^{\perp}]^T}(A)+B) \in \mathfrak{m}$ leads to the following characterisation of tangent vectors. 

\begin{prop}
Any tangent vector $V$ at $Q\in\mathbf{G}_{p,n}(\mathbb{R})$ is an $n\times p$ matrix such that $Q^TV=0$, and $V$ can be expressed in
the form \eqref{eq2.1}
with $F=V$. 
\end{prop}
The proof follows from \eqref{diffcurveinM} assuming $g(t)=[Q(t) Q(t)^{\perp}]\in\SO(n)$, and $h(t)$ of  the form \eqref{isotropyGroupGR}, imposing the stated choice of isotropy, and projecting the resulting curves on $\StfR$ by post-multiplication by $I_p$. 

We proceed by using (\ref{eq2.1}) but with $F=V$. 
Define $a_Q:T_Q\M\rightarrow \g$ as in \eqref{SRVT2} with $f_Q:T_Q\M\rightarrow T_Q\M,$ the identity map $f_Q(V)=V$. 
Suppose that $Y(t)$ is a curve on the Grassmann manifold, then the SRVT of $Y$ is a curve on $\mathfrak{so}(n)$  and takes the form \eqref{SRVTStiefel} which here becomes
\begin{equation}
\label{SRVTGrassmann}
\mathcal{R}(Y):=\frac{\dot{Y}Y^{\mathrm{T}}-Y\dot{Y}^{\mathrm{T}} }{\sqrt{\|\dot{Y}Y^{\mathrm{T}}-Y\dot{Y}^{\mathrm{T}}\|}}.
\end{equation}
%
The reductive SRVT is defined by \eqref{SRVTStiefelred} with 
$$\tilde{f}_Q(V)=[Q,Q^{\perp}]^TV=\left[ \begin{array}{c}
O\\
(Q^{\perp})^TV
\end{array}\right]$$ 
and $\tilde{a}_Q$ as in \eqref{SRVTm2}, which implies $\tilde{a}_Q(V)\in\mathfrak{m}$.

\section{Numerical experiments}\label{numericalexperiments}
To demonstrate an application of the SRVT introduced in this paper, we present a simple example of interpolation between two curves on the unit $2$-sphere. In the following we describe some implementation details for this example.

\begin{setup}[Preliminaries]
We will use Rodrigues' formula for the Lie group exponential,
\begin{equation*}
\exp(\hat{x}) = I + \frac{\sin{(\alpha)}}{\alpha}\hat{x} + \frac{1- \cos{(\alpha)}}{\alpha^2}\hat{x}^2, \quad \alpha = \lVert x \rVert_2,\quad x = 
 \begin{pmatrix}
  x_1 \\
  x_2 \\
  x_3
 \end{pmatrix}
\mapsto
\hat{x} = 
 \begin{pmatrix}
  0 & -x_3 & x_2 \\
  x_3 & 0 & -x_1 \\
  -x_2 & x_1 & 0
 \end{pmatrix}
\end{equation*}
where $x\mapsto \hat{x}$  defines 
an isomorphism between vectors in $\mathbb{R}^3$ and $3 \times 3$ skew-symmetric matrices in $\mathfrak{so}(3)$. 
\end{setup}

\begin{setup}[Interpolated curves]\label{ssect: Interp}
Given a continuous curve $c(t), t \in [t_0,t_N]$ on the Stiefel manifold $\SO(3)/\SO(2)$, which is diffeomorphic to $\text{S}^2$, we replace $c(t)$ with the curve $\bar{c}(t)$ interpolating between $N+1$ values $\bar{c}_i =c(t_i)$, with $t_0 < t_1 < ... < t_N$, as follows:
\begin{equation}
\bar{c}(t) := \sum_{i=0}^{N-1} \chi_{[t_i,t_{i+1})}(t) \exp{\left(\frac{t-t_i}{t_{i+1}-t_i}\left(v_i \bar{c}_i^\text{T}- \bar{c}_i v_i^\text{T} \right)\right)}\bar{c}_i,
\label{eq:capprox}
\end{equation}
where $\chi$ is the characteristic function, $\exp$ is the Lie group exponential, and $v_i$ are approximations to $\left.\frac{d}{d t}c(t)\right|_{t=t_i}$ found by solving the equations
\begin{align}
\label{eqforvi}
\bar{c}_{i+1} = \exp{\left(v_i \bar{c}_i^\text{T}- \bar{c}_i v_i^\text{T}\right)}\bar{c}_i \\
\text{constrained by} \quad 
v_i^\text{T}\bar{c}_i = 0.
\label{eq:constrains}
\end{align}
The $v_i$, $i=1,\dots, N$, can be found explicitly, by a simple calculation.
We observe that if $\kappa = \bar{c}_i \times v_i$, then $\hat{\kappa} = v_i \bar{c}_i^\text{T}- \bar{c}_i v_i^\text{T}$. By (\ref{eq:constrains}), we have that $\lVert \bar{c}_i \times v_i \rVert_2 = \lVert \bar{c}_i \rVert_2 \lVert v_i \rVert_2 = \lVert v_i \rVert_2$, where the last equality follows because we assume the sphere to have radius $1$, and so
$\lVert \bar{c}_i\rVert_2 = \bar{c}_i^\text{T}\bar{c}_i = 1$. Using Rodrigues' formula, from \eqref{eqforvi} we obtain 
$$
\bar{c}_{i+1} = 
\frac{\sin{(\lVert v_i \rVert_2)}}{\lVert v_i \rVert_2}v_i + \cos{\left(\lVert v_i \rVert_2\right)} \bar{c}_i.
$$
Thus $\bar{c}_i^\text{T}\bar{c}_{i+1} = 1- \cos{\left(\lVert v_i \rVert_2\right)}$ and  so $\lVert v_i \rVert_2 = \arccos{\left(\bar{c}_i^\text{T}\bar{c}_{i+1}\right)}$ leading to
\begin{align*}
v_i = \left(\bar{c}_{i+1} - \bar{c}_i^\text{T}\bar{c}_{i+1}\bar{c}_i\right) \frac{\arccos{\left(\bar{c}_i^\text{T}\bar{c}_{i+1}\right)}}{\sqrt{1-\left(\bar{c}_i^\text{T}\bar{c}_{i+1}\right)^2}}.
\end{align*}
Inserting this into (\ref{eq:capprox}) gives
\begin{equation}
\bar{c}(t) = \sum_{i=0}^{N-1} \chi_{[t_i,t_{i+1})}(t) \exp{\left(\frac{t-t_i}{t_{i+1}-t_i} \frac{\arccos{\left(\bar{c}_i^\text{T}\bar{c}_{i+1}\right)}}{\sqrt{1-\left(\bar{c}_i^\text{T}\bar{c}_{i+1}\right)^2}} \left(\bar{c}_{i+1} \bar{c}_i^\text{T} - \bar{c}_i \bar{c}_{i+1}^\text{T} \right)\right)}\bar{c}_i.
\label{eq:capprox2}
\end{equation}
\end{setup}
\begin{setup}[The SRVT and its inverse]\label{ssect: SRVT}
By Definition~\ref{defn:SRVT} and formulae \eqref{SRVT1}, \eqref{SRVT2} and \eqref{SRVTStiefel}, the SRVT  of the curve \eqref{eq:capprox2} is a piecewise constant function $\bar{q}(t)$ in $\mathfrak{so}(3)$, 
taking values $\bar{q}_i = \bar{q}(t_i), \,i = 0,...,N-1$, where $\bar{q}_i =\left.\mathcal{R}(\bar{c})\right|_{t=t_i}$ is given by
\begin{align}
\bar{q}_i=\frac{v_i \bar{c}_i^\text{T}- \bar{c}_i v_i^\text{T}}{\lVert v_i \bar{c}_i^\text{T}- \bar{c}_i v_i^\text{T} \rVert^\frac{1}{2}}=\frac{\arccos^\frac{1}{2}{\left(\bar{c}_i^\text{T}\bar{c}_{i+1}\right)}}{\left(1-\left(\bar{c}_i^\text{T}\bar{c}_{i+1}\right)^2\right)^\frac{1}{4}\lVert \bar{c}_{i+1} \bar{c}_i^\text{T} - \bar{c}_i \bar{c}_{i+1}^\text{T}  \rVert^\frac{1}{2}} \left(\bar{c}_{i+1} \bar{c}_i^\text{T} - \bar{c}_i \bar{c}_{i+1}^\text{T} \right).
\label{eq:numSRVT}
\end{align}
Here the norm $\lVert \cdot \rVert$ is induced by the negative (scaled) Killing form, which for skew-symmetric matrices corresponds to the Frobenius inner product, $\lVert A \rVert = \sqrt{\text{tr}(A A^\text{T})}$.

The inverse SRVT is then given by (\ref{eq:capprox2}), with  
\begin{equation*}
\bar{c}_{i+1} = \exp{\left(\lVert\bar{q}_i\rVert\bar{q}_i\right)}\bar{c}_i, \quad i = 1,...,N-1, \quad \bar{c}_0 = c(t_0).
\end{equation*}
\end{setup}
\begin{setup}[The reductive SRVT]\label{ssect: redSRVT}
Since $\Evol(a_{\bar{c}_i}(v_i)) = \exp{(a_{\bar{c}_i}(v_i))}$, the reductive SRVT (\ref{setup: twist}) becomes then
\begin{align}
\left.\mathcal{R}_{\mathfrak{m}}(\bar{c})\right|_{t=t_i} = \left.\mathcal{R}(\tilde{c})\right|_{t=t_i}= \frac{\arccos^\frac{1}{2}{\left(\tilde{c}_i^\text{T}\tilde{c}_{i+1}\right)}}{\left(1-\left(\tilde{c}_i^\text{T}\tilde{c}_{i+1}\right)^2\right)^\frac{1}{4}\lVert \tilde{c}_{i+1} \tilde{c}_i^\text{T} - \tilde{c}_i \tilde{c}_{i+1}^\text{T}  \rVert^\frac{1}{2}} \left(\tilde{c}_{i+1} \tilde{c}_i^\text{T} - \tilde{c}_i \tilde{c}_{i+1}^\text{T} \right),
\label{eq:twistedSRVT}
\end{align}
with
\begin{align*}
\tilde{c}_i = [U,U^\perp]^\text{T}_i \bar{c}, \quad i=0,...,N, \qquad
[U,U^\perp]_{i+1} = \exp{(a_{\bar{c}_i}(v_i))}[U,U^\perp]_{i} \quad  i=0,...,N-1,
\end{align*}
where $[U,U^\perp]_0$ can be found e.g. by $QR$-factorization of $c(t_0)$.
\end{setup}\textbf{•}
\begin{setup}[Curve blending on the 2-sphere]
We wish to compute the geodesic in the shape space of curves on the sphere between the two curves $\bar{c}^1(t)$ and $\bar{c}^2(t)$. Following \cite{celledoni15sao}, we use a dynamic programming algorithm to solve the optimization problem (\ref{Eq:DistanceShape}) (see \cite{sebastian03, bauer2015} for a detailed description on the use of dynamic programming for shapes):

\begin{algorithm}
\caption{REPARAMETRISATION\cite[Section 3.2]{bauer2015}}\label{alg:reparam}
\begin{algorithmic}
\State Given $\bar{q}^1(t), \bar{q}^2(t), N, \left\lbrace t_i\right\rbrace_{i=0}^N$
\For {$i, j \in \{0,\dots,N \}$}
    \State $c_\text{min} = \infty$
    \For {$k \in \{0,...,i-1\}, l \in \{0,...,j-1\}$}
    		\State {$c_\text{loc} = \int_{t_i}^{t_k} \lvert \bar{q}^1(t) - \bar{q}^2(t_l+\frac{t_j-t_l}{t_i-t_k}t)\rvert^2 \di t$}
    		\If {$\Psi^m(k,l) = \Psi \circ \cdots \circ \Psi (k,l) = (0,0)$ for some $m \geq 0$}
    			\State $z = 0$
    		\Else
    			\State $z = \infty$
    		\EndIf
    		\State $c = c_\text{loc} + A_{k,l} + z$
    		\If {$c < c_\text{min}$}
    			\State $c_\text{min} = c$
    			\State $\Psi(i,j) = (k,l)$
    		\EndIf
    	\EndFor
    	\State $A_{i,j} = c_\text{min}$
\EndFor
\State Create two vectors of indices $(p,q)$ by setting $(p_0,q_0) = (N,N)$ and
\State $(p_{m+1},q_{m+1}) = \Psi(p_m,q_m)$ until $(p_{m+1},q_{m+1}) = (0,0)$
\State Flip $(p,q)$ so it starts at $(0,0)$ and ends at $(N,N)$
\For {$i \in \{0,...,N \}$}
	\State $s_i = t_{q_j} + (t_{q_{j+1}}-t_{q_j}) \frac{t_i -t_{p_j}}{t_{p_{j+1}}-t_{p_j}} \text{ for }\, j$ s.t. $p_j \leq i < p_{j+1}$
\EndFor
\State Return $s = \left\lbrace s_i\right\rbrace_{i=0}^N$
\end{algorithmic}
\end{algorithm}

With this approach, 
we reparametrise optimally the curve $\bar{c}^2(t)$ while minimizing its distance to $\bar{c}^1(t)$. This distance is measured by taking the $L^2$ norm of $\bar{q}^1(t) - \bar{q}^2(t)$ in the Lie algebra. In the discrete case, this reparametrisation yields an optimal set of grid points $\left\{ s_i \right\}_{i=0}^N$, where $s_0 = t_0 < s_1 < ... < s_N = t_N$, from which we find $\bar{c}'^2_i = \bar{c}^2(s_i)$ by (\ref{eq:capprox2}). See \cite{celledoni15sao} for further details.

We interpolate between $\bar{c}^1(t)$ and $\bar{c}'^2(t)$ by performing a linear convex combination of their SRV transforms $\bar{q}^1(t)$ and $\bar{q}'^2(t)$, and then by taking the inverse SRVT of the result. We obtain
\begin{equation*}
\bar{c}_\text{int}(\bar{c}_1,\bar{c}'_2,\theta) = \mathcal{R}^{-1}\left(\left(1-\theta\right)\mathcal{R}(\bar{c}_1) + \theta \,\mathcal{R}(\bar{c}'_2)\right), \qquad \theta \in [0,1].
\end{equation*}

Examples are reported in Figures \ref{fig:curves1}, \ref{fig:curves2} and \ref{fig:curves3}, where interpolation between two curves is performed with and without reparametrisation. We show curves resulting from using both (\ref{eq:numSRVT}) and (\ref{eq:twistedSRVT}), and compare these to the results obtained when employing the SRVT introduced in \cite{celledoni15sao} on curves in $\text{SO}(3)$ which are then traced out by a vector in $\mathbb{R}^3$ to match the curves in $\text{S}^2$ studied here.

\end{setup}
\begin{setup}[Conclusions]
We have proposed generalisations of the SRVT approach to curves and shapes evolving on homogenous manifolds using Lie group actions. Different Lie group actions lead to different Riemannian metrics in the infinite dimensional manifolds of curves and shapes opening up for a variety of possibilities which can all be implemented in the same generalised SRVT framework. We have presented only a few preliminary examples here, and further tests and analysis will be the subject of future work.
A number of open questions related to the properties of the pullback metrics through the SRVT, to the performance of the algorithms when using different Lie group actions, to the comparison of the SRVT and the reductive SRVT and to the implementation of the approach in examples of non reductive homogeneous manifolds will be addressed in future research. 
\end{setup}

\section{Acknowledgment} 
\vskip-0.3cm
This work was supported by the Norwegian Research Council, and by the European Union's Horizon 2020 research and innovation programme under the Marie Sklodowska-Curie, grant agreement No.
691070.


\begin{figure}[htbp]
\begin{center}
\subfloat[From left to right: Two curves on the sphere, their original parametrisations, the reparametrisation minimizing the distance in $\text{SO}(3)$ and the reparametrisation minimizing the distance in $S^\text{2}$, using the reductive SRVT (\ref{eq:twistedSRVT}).]{
\includegraphics[width=0.24\textwidth]{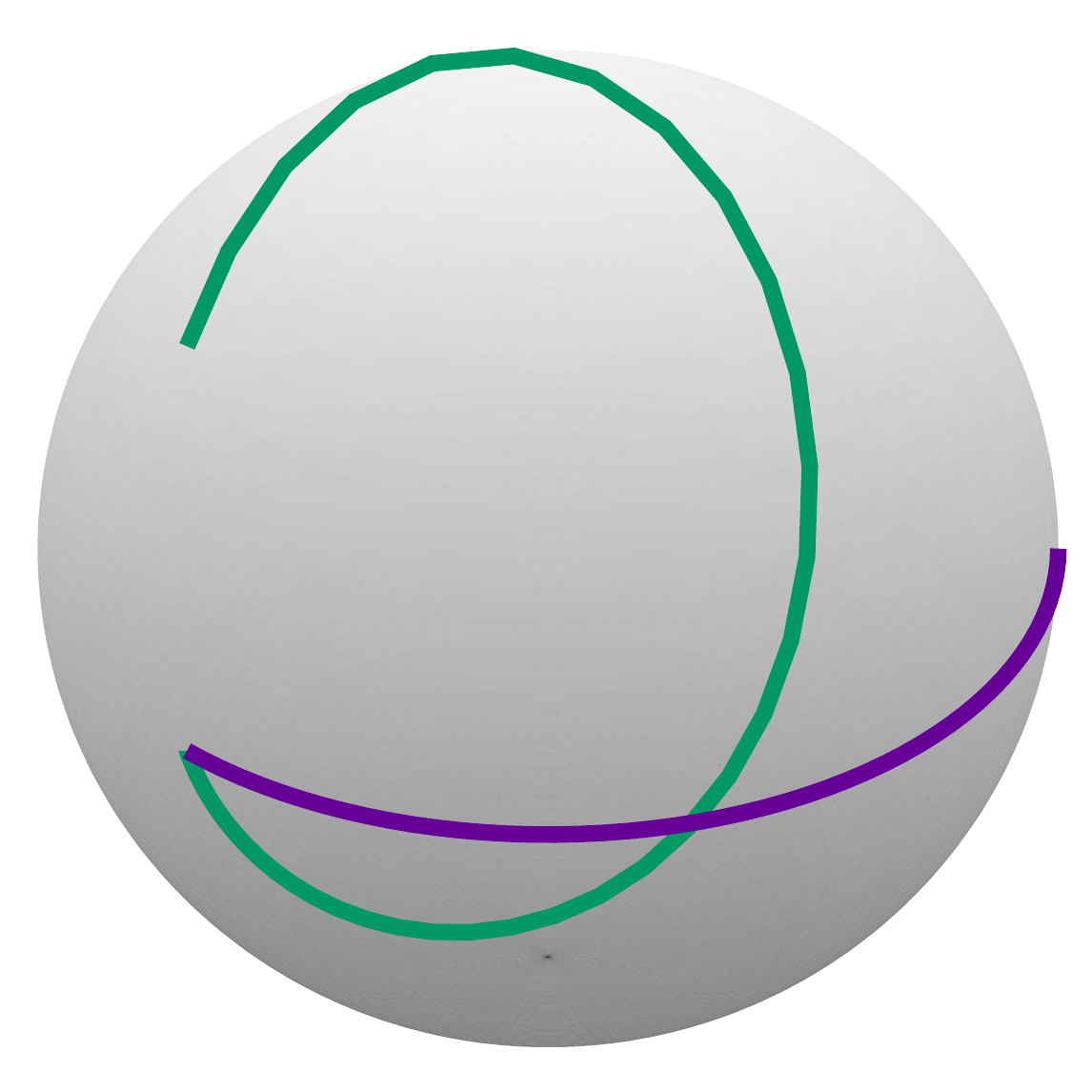}
\includegraphics[width=0.24\textwidth]{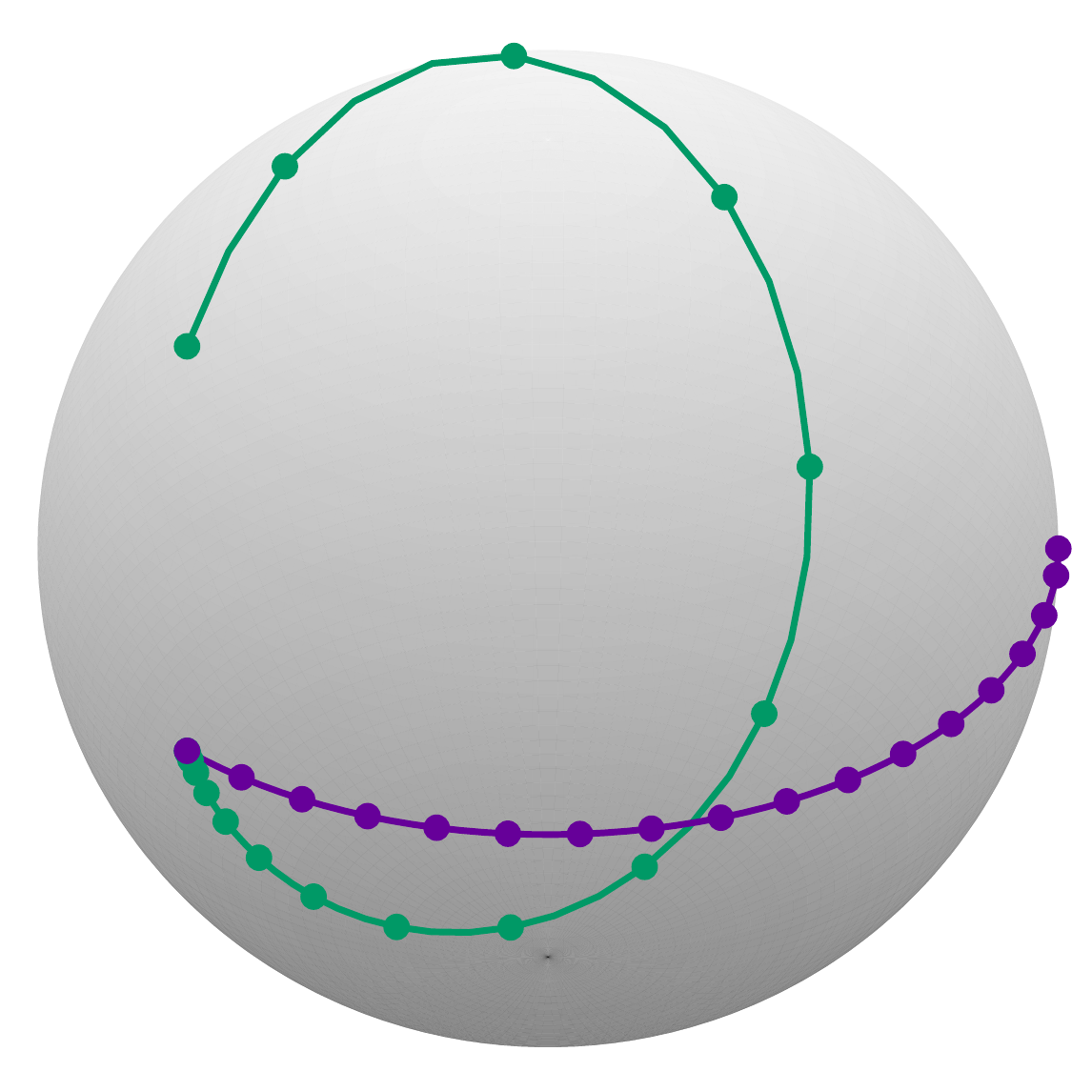}
\includegraphics[width=0.24\textwidth]{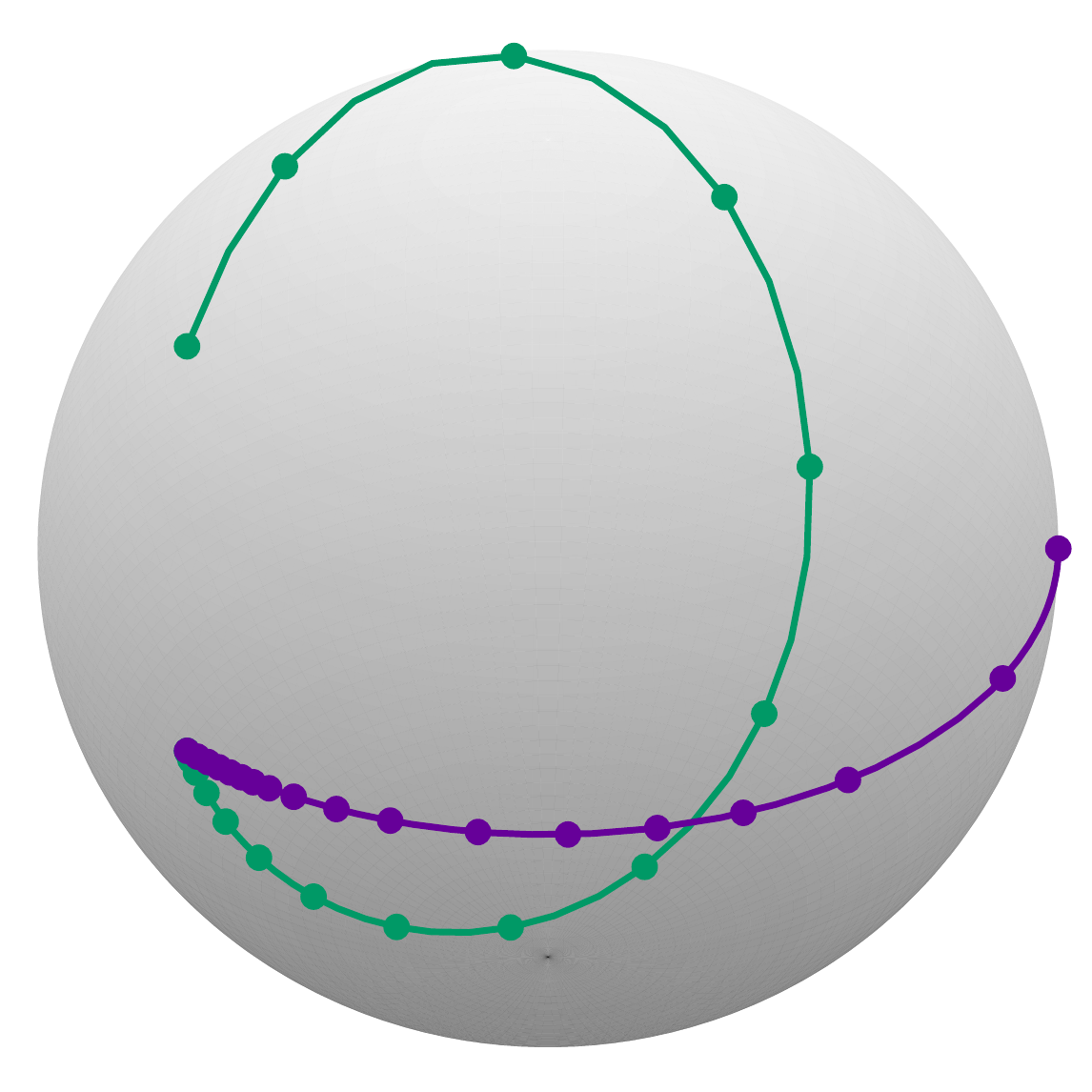}
\includegraphics[width=0.24\textwidth]{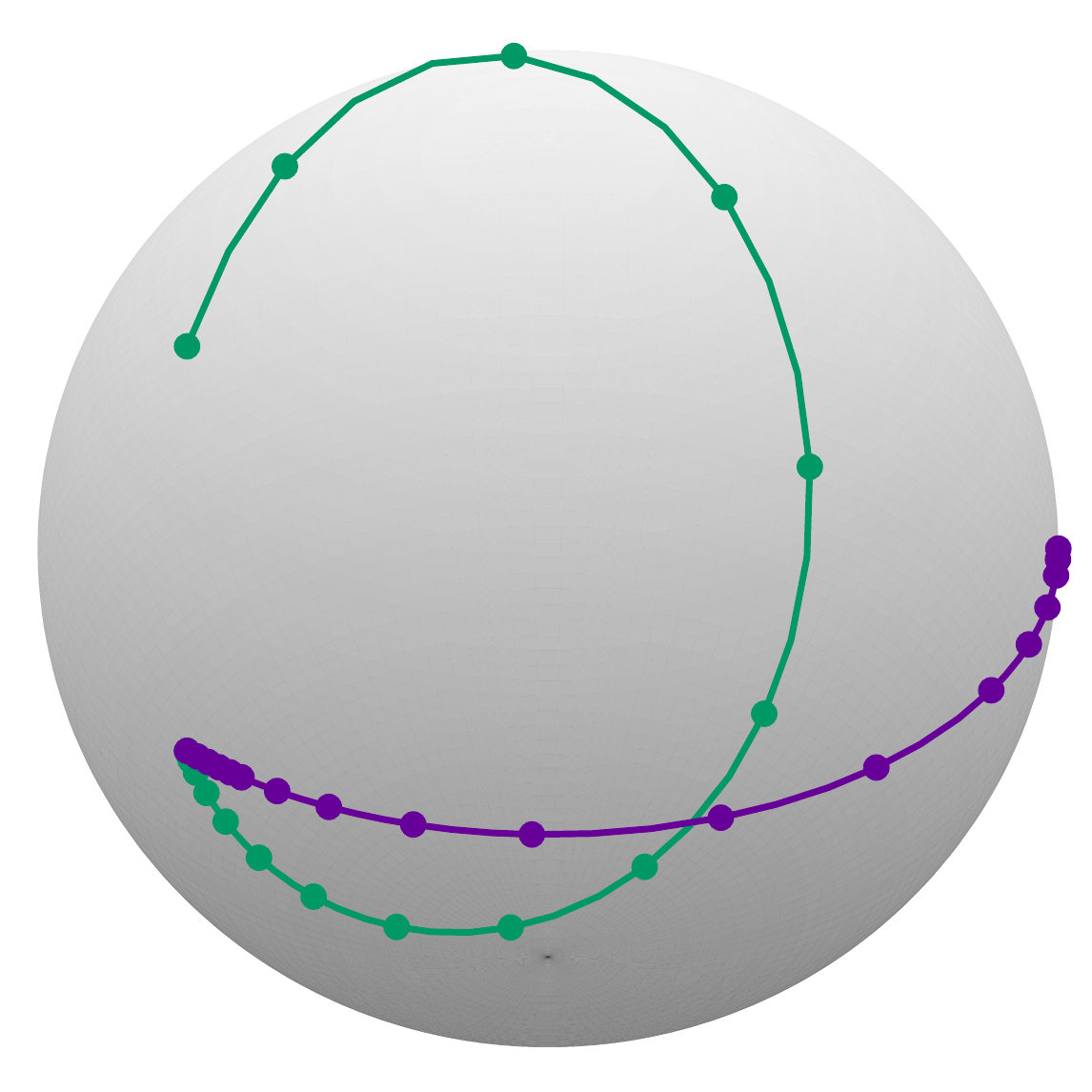}}\\
\subfloat[The interpolated curves at times $t = \left\{\frac{1}{4},\frac{1}{2},\frac{3}{4}\right\}$, from left to right, before reparametrisation, on $\text{SO}(3)$ (yellow, dashed line) and $\text{S}^2$ (blue, solid line).]{
\includegraphics[width=0.33\textwidth]{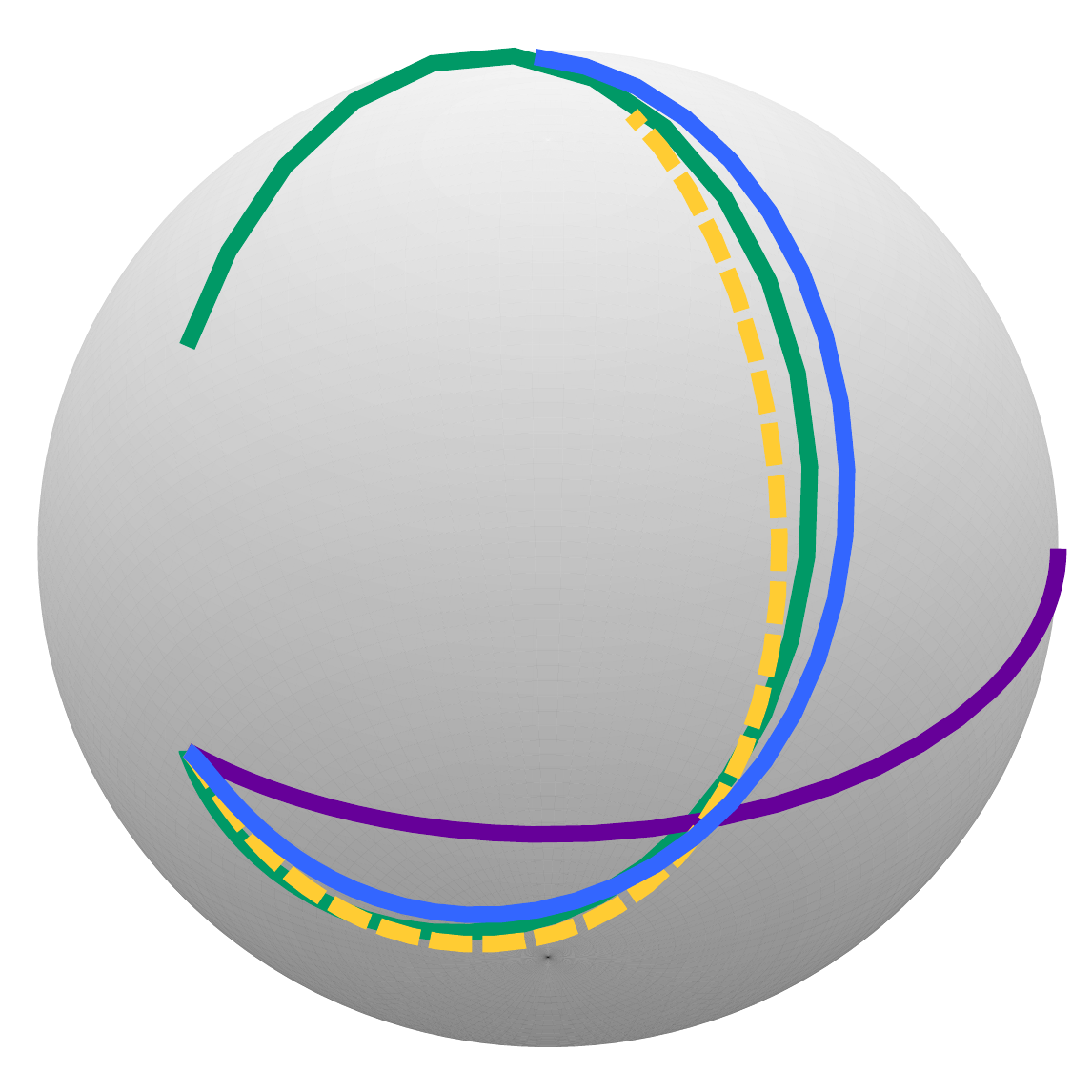}
\includegraphics[width=0.33\textwidth]{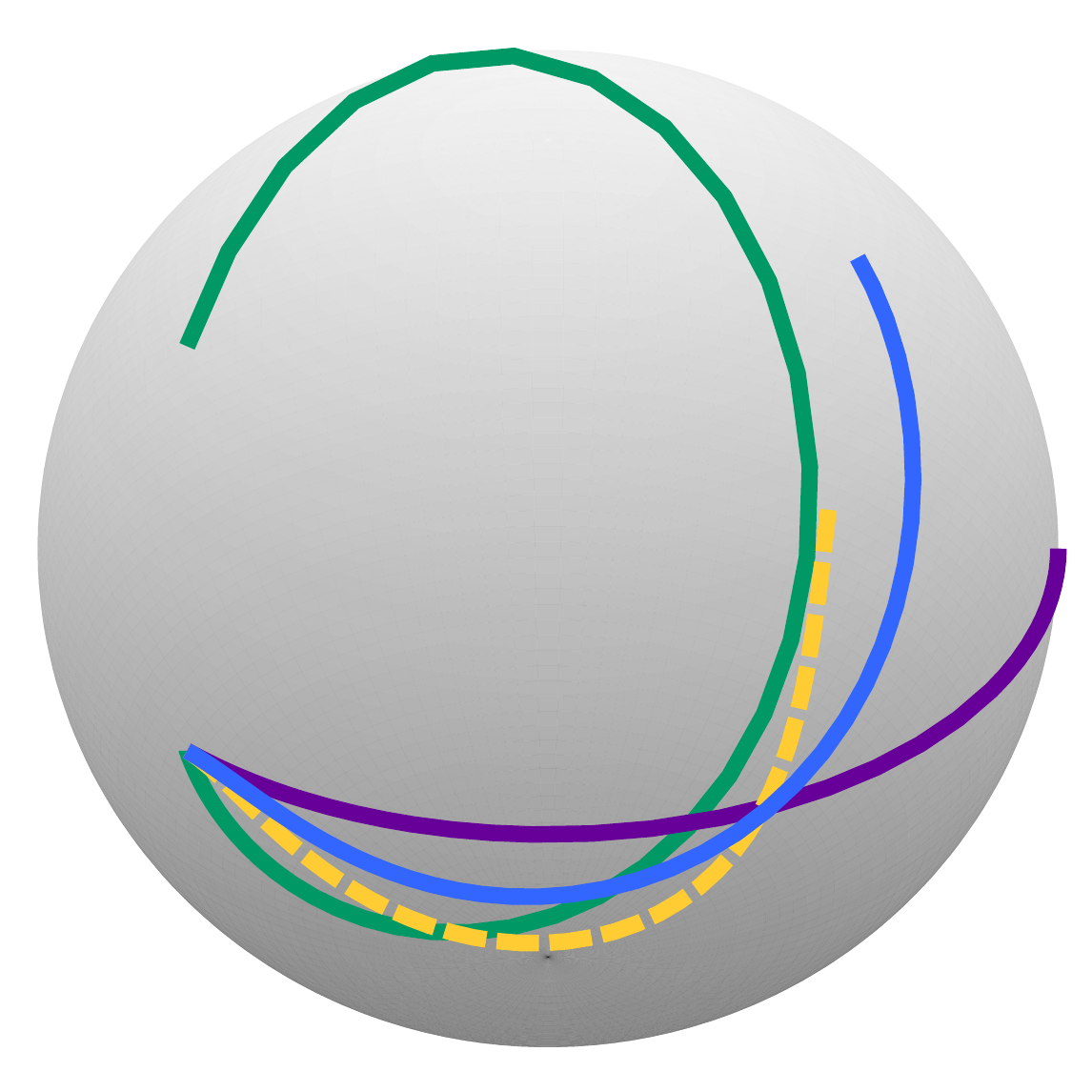}
\includegraphics[width=0.33\textwidth]{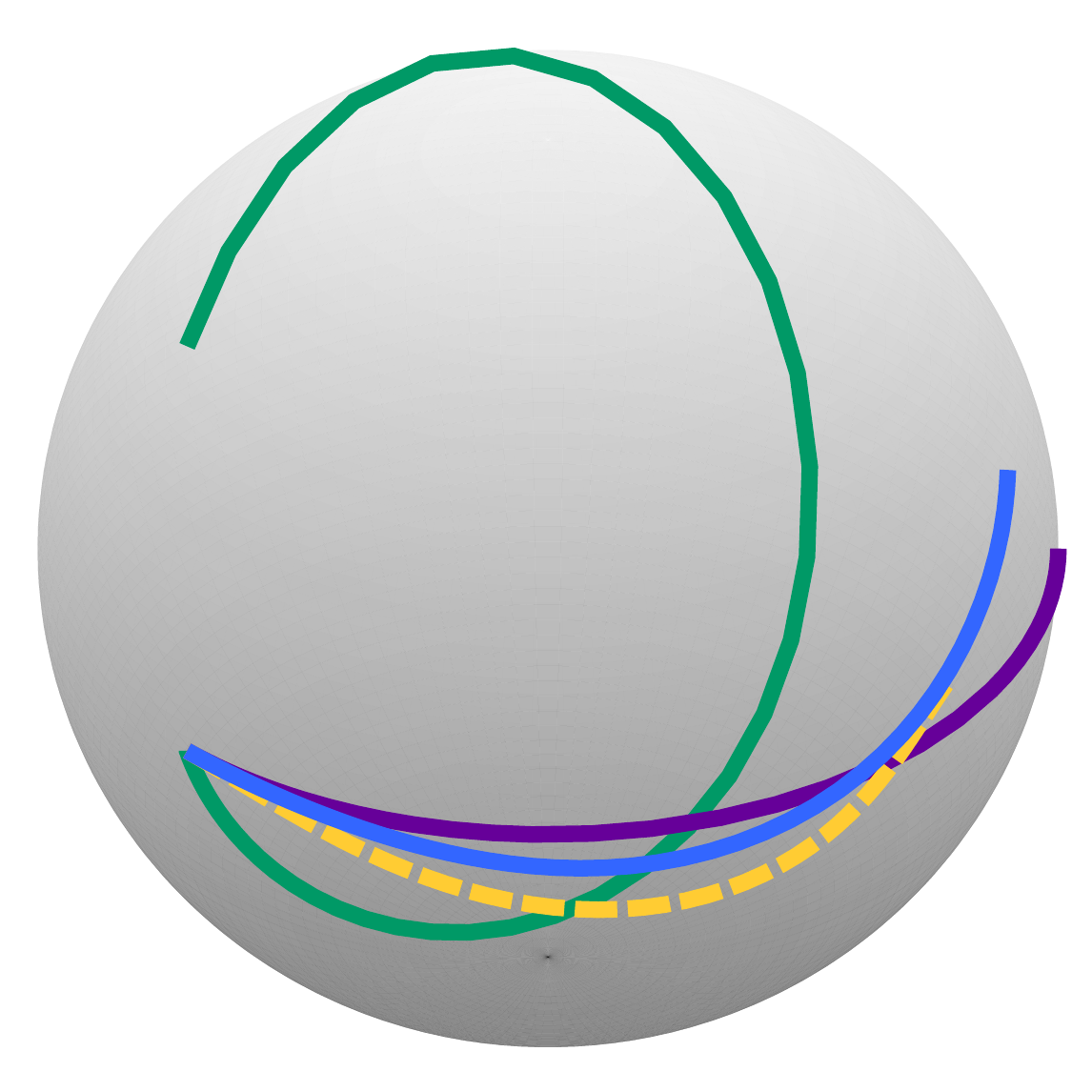}}\\
\subfloat[The interpolated curves at times $t = \left\{\frac{1}{4},\frac{1}{2},\frac{3}{4}\right\}$, from left to right, after reparametrisation, on $\text{SO}(3)$ (yellow, dashed line) and $\text{S}^2$ (blue, solid line).]{
\includegraphics[width=0.33\textwidth]{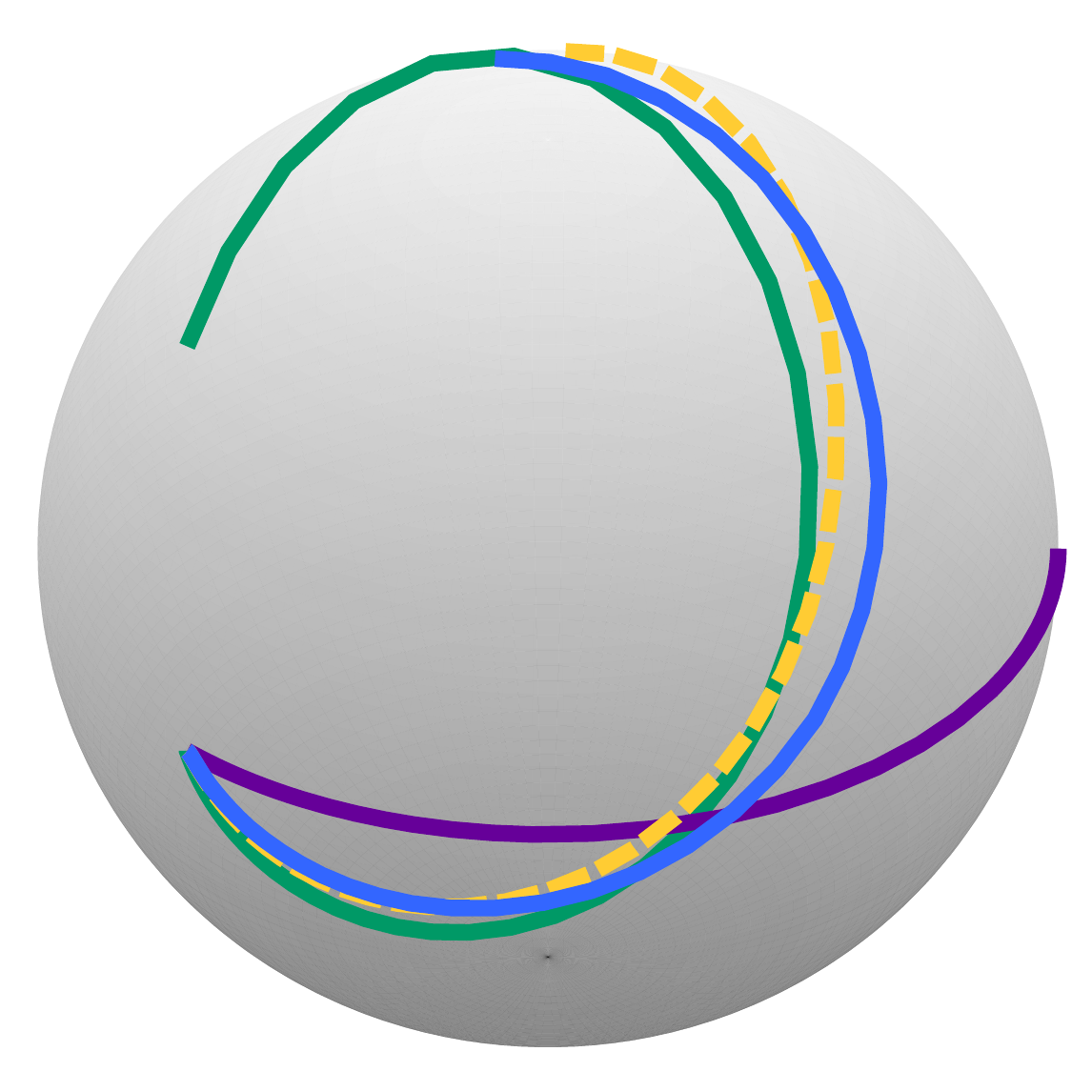}
\includegraphics[width=0.33\textwidth]{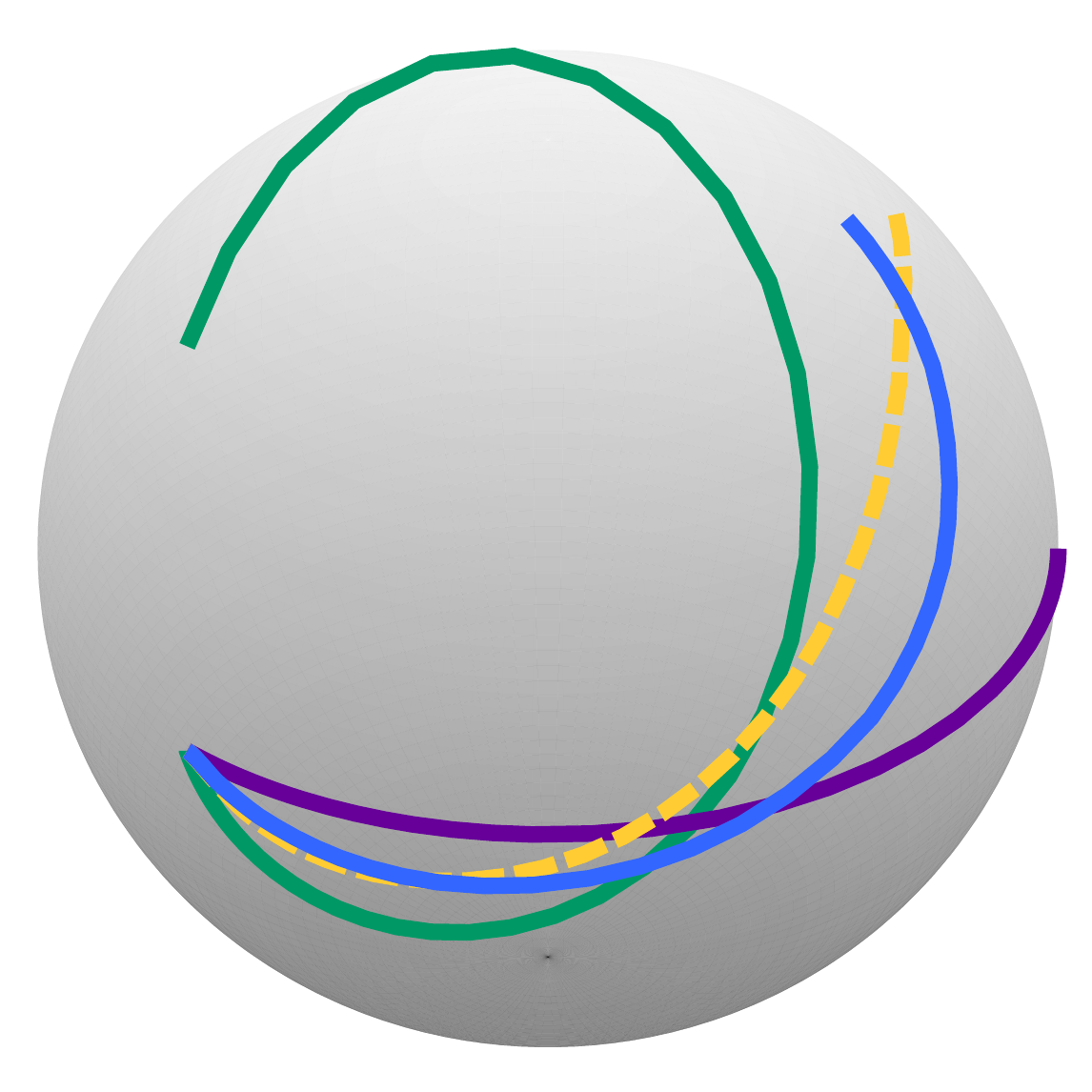}
\includegraphics[width=0.33\textwidth]{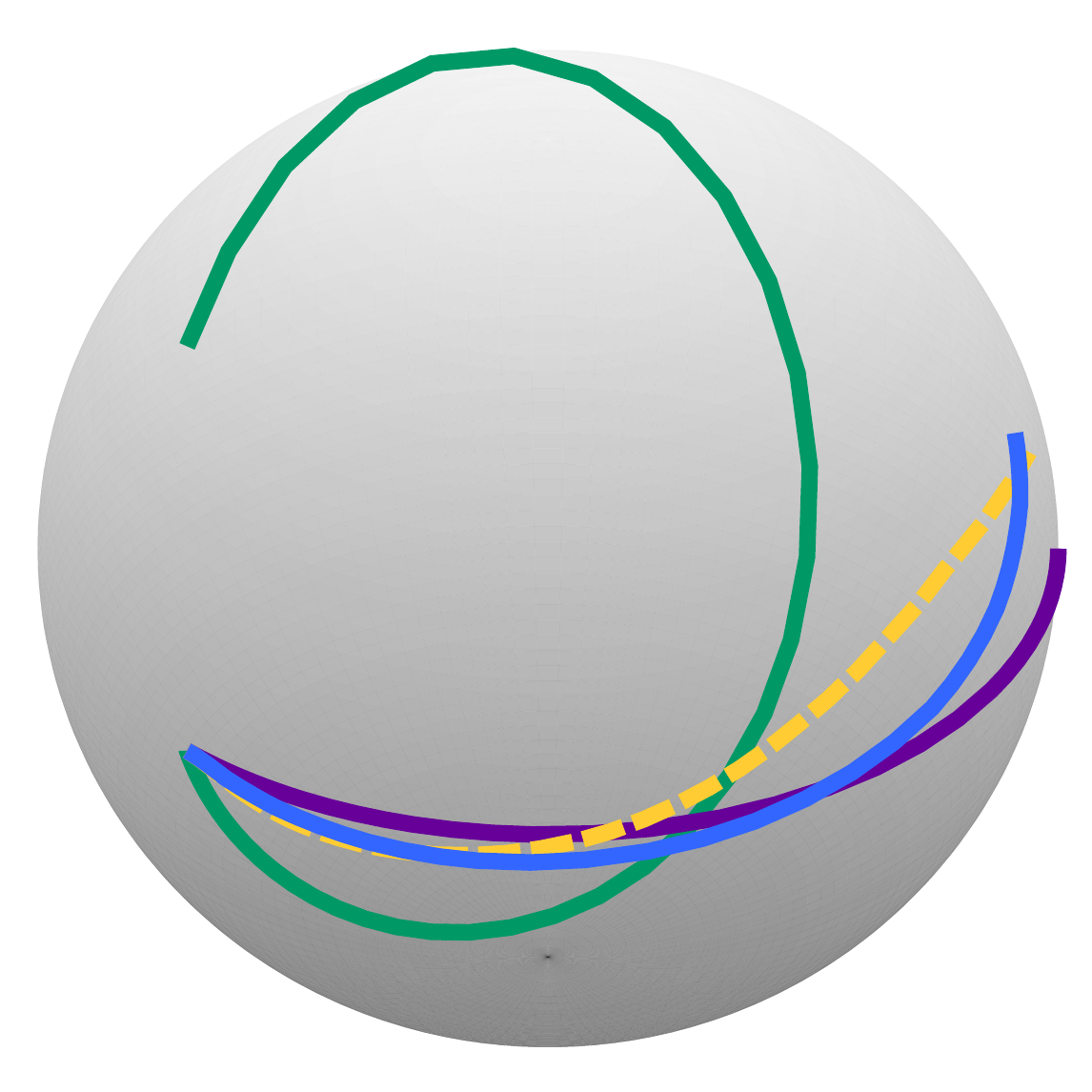}}
\caption{Interpolation between two curves on $\text{S}^2$, with and without reparametrisation, obtained by the reductive SRVT (\ref{eq:twistedSRVT}). The results obtained by using the SRVT (\ref{eq:numSRVT}) are not identical to these, but in this case very similar, and therefore omitted. The results are compared to the corresponding SRVT interpolation between curves on $\text{SO}(3)$, which are then mapped to $\text{S}^2$ by multiplying with the vector $(1,0,0)^\text{T}$. The curves are $c^1(t) =  R_x(\pi t^3) R_y(\pi t^3) R_y(\pi t^3/2)  \cdot (1,0,0)^\text{T}$ and $c^2(t) =  R_z(3 \pi t/4) R_x(\pi t) \cdot (1,0,0)^\text{T}$ for $t \in [0,1]$, where $R_x(t)$, $R_y(t)$ and $R_z(t)$ are the rotation matrices in $\text{SO}(3)$ corresponding to rotation of an angle $t$ around the $x$-, $y$- and $z$-axis, respectively. 
}
\label{fig:curves1}
\end{center}
\end{figure}

\begin{figure}[htbp]
\begin{center}
\subfloat[From left to right: Two curves on the sphere, their original parametrisations and the reparametrisation minimizing the distance in $S^\text{2}$, using the reductive SRVT (\ref{eq:twistedSRVT}).]{
\includegraphics[width=0.33\textwidth]{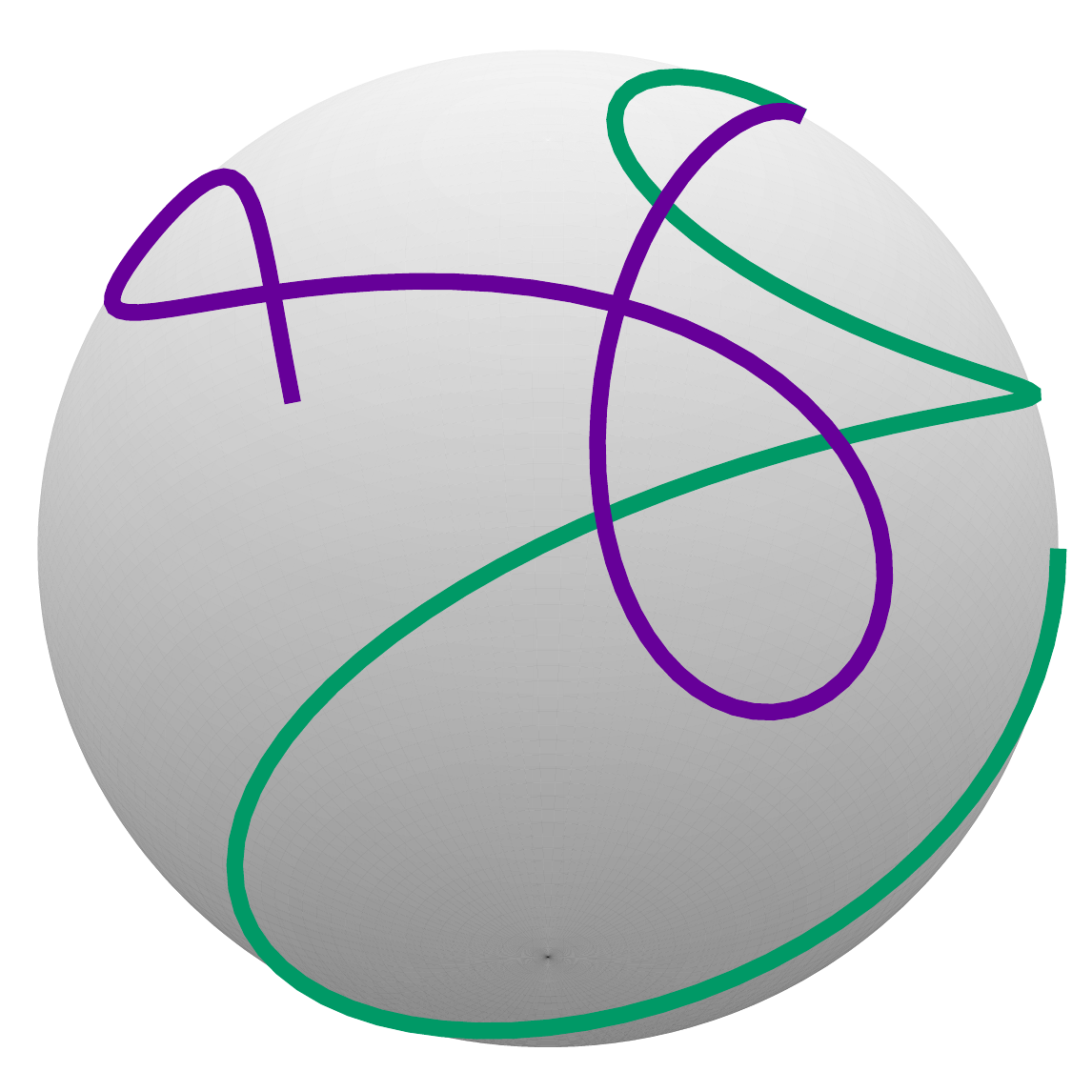}
\includegraphics[width=0.33\textwidth]{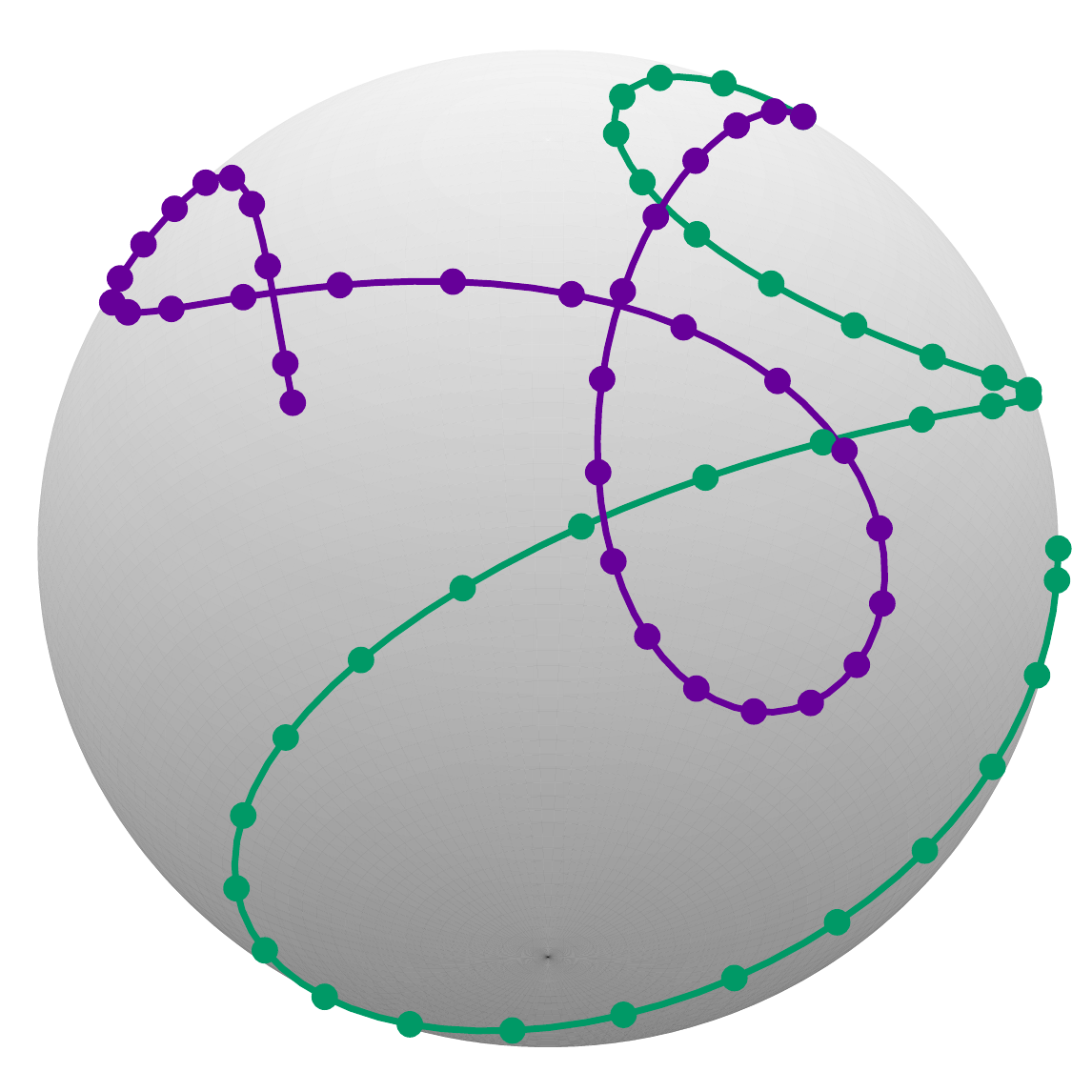}
\includegraphics[width=0.33\textwidth]{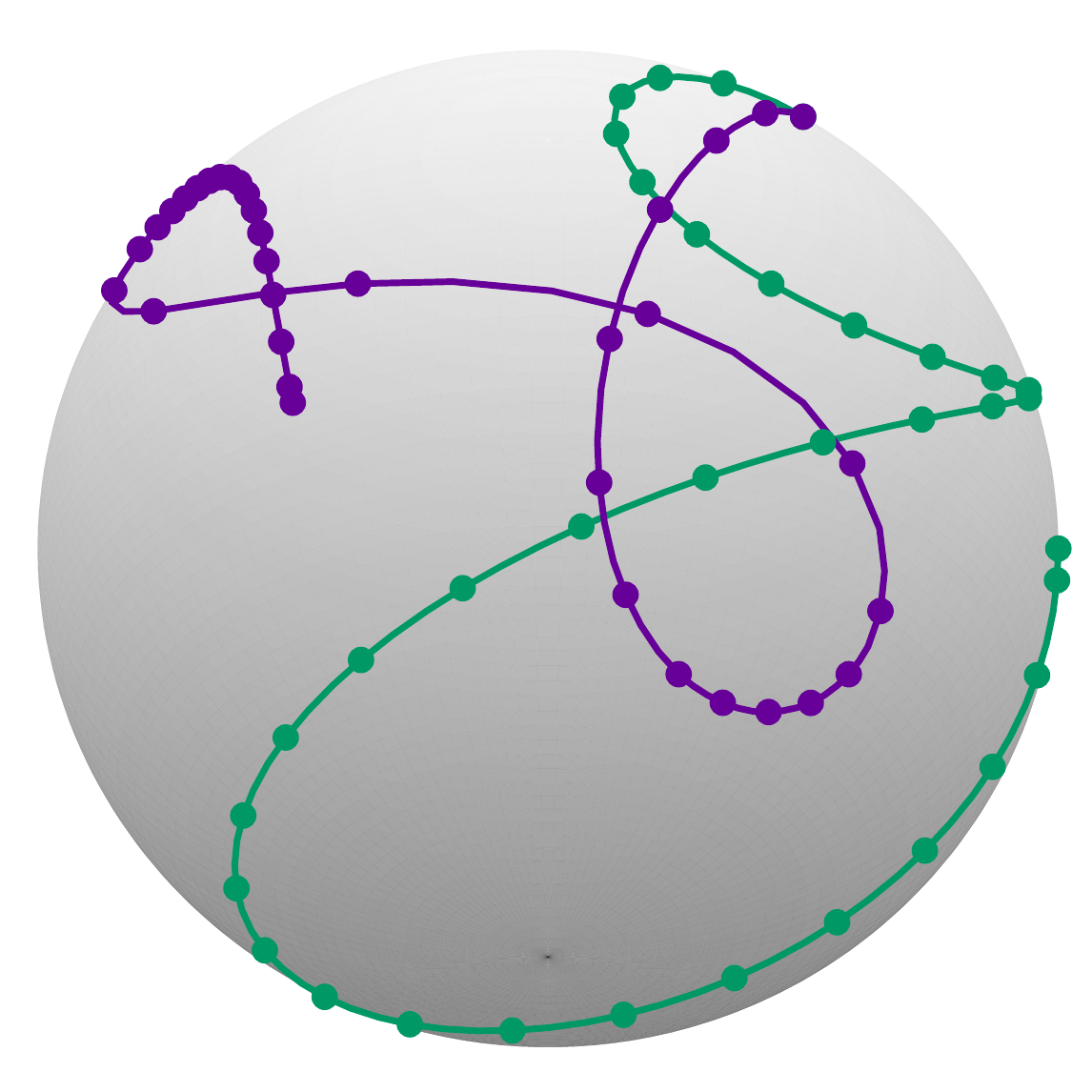}}\\
\subfloat[The interpolated curves at times $t = \left\{\frac{1}{4},\frac{1}{2},\frac{3}{4}\right\}$, from left to right, before reparametrisation.]{
\includegraphics[width=0.33\textwidth]{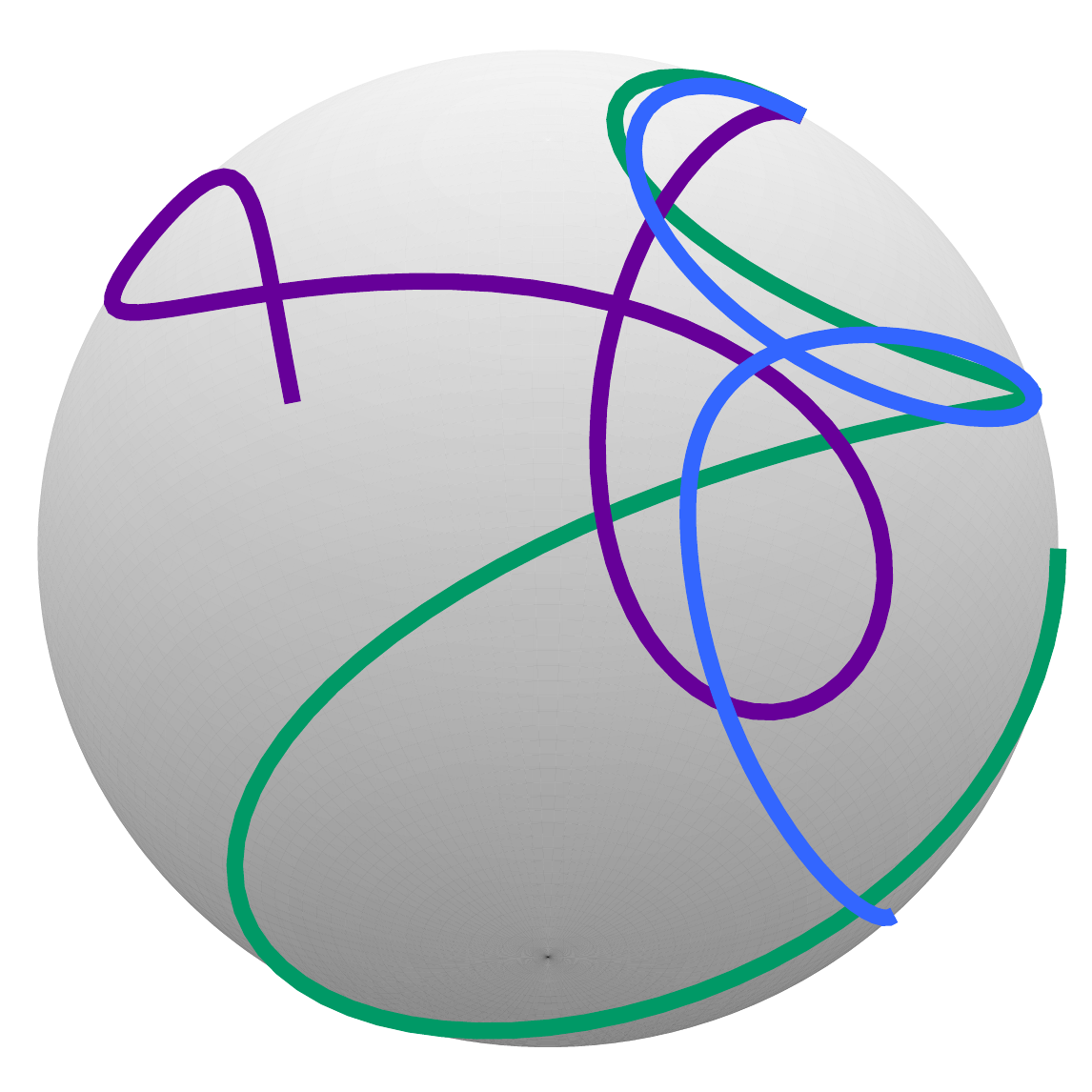}
\includegraphics[width=0.33\textwidth]{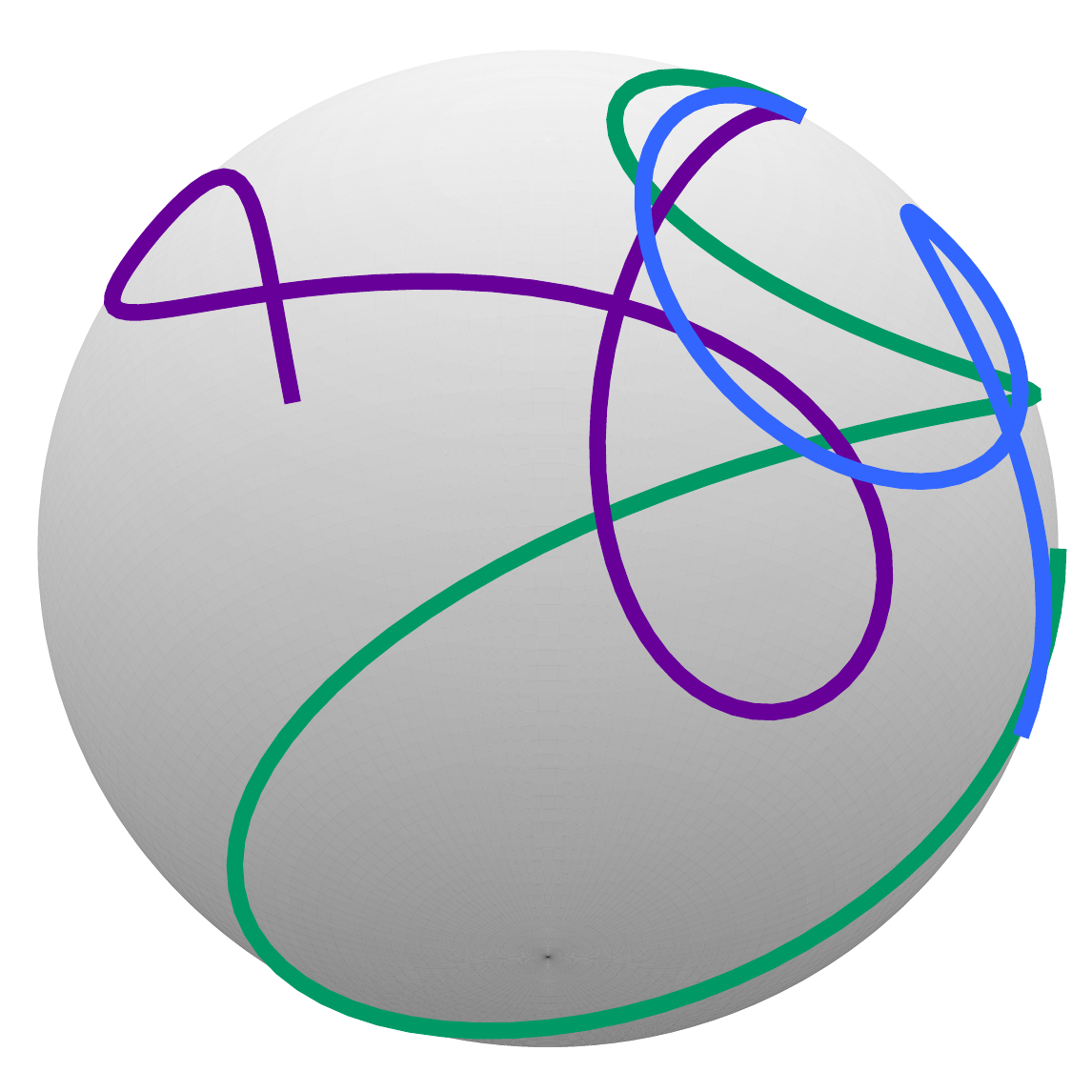}
\includegraphics[width=0.33\textwidth]{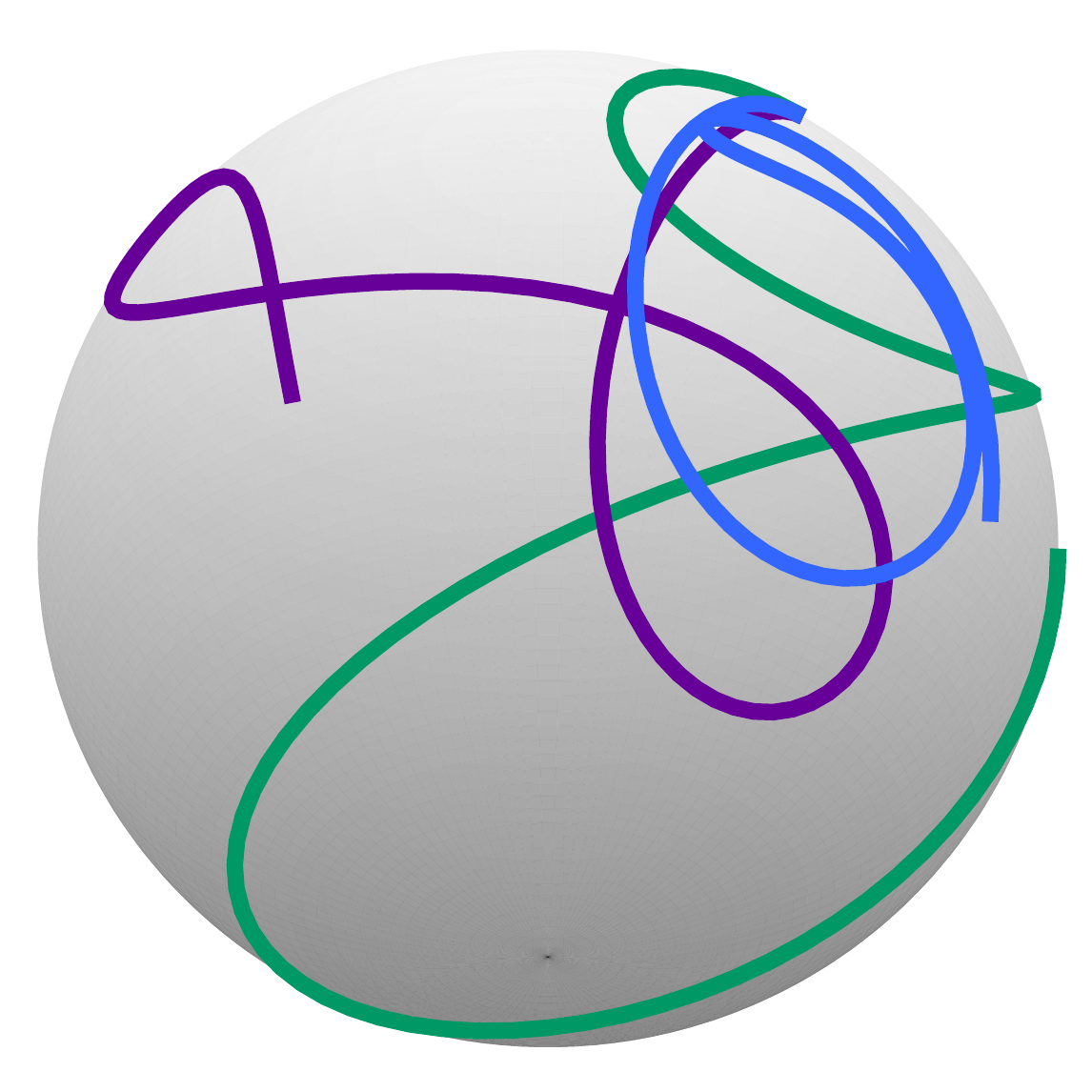}}\\
\subfloat[The interpolated curves at times $t = \left\{\frac{1}{4},\frac{1}{2},\frac{3}{4}\right\}$, from left to right, after reparametrisation.]{
\includegraphics[width=0.33\textwidth]{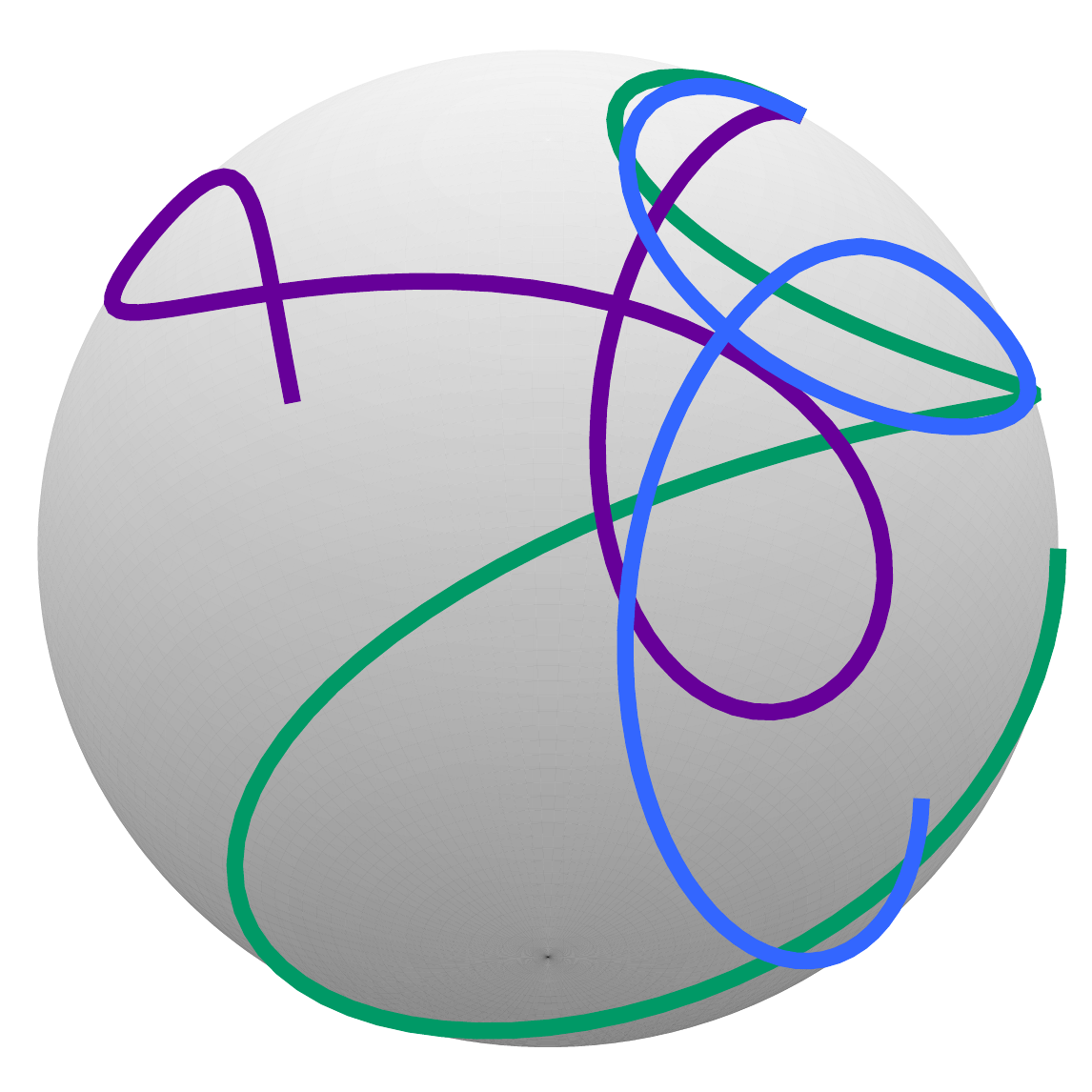}
\includegraphics[width=0.33\textwidth]{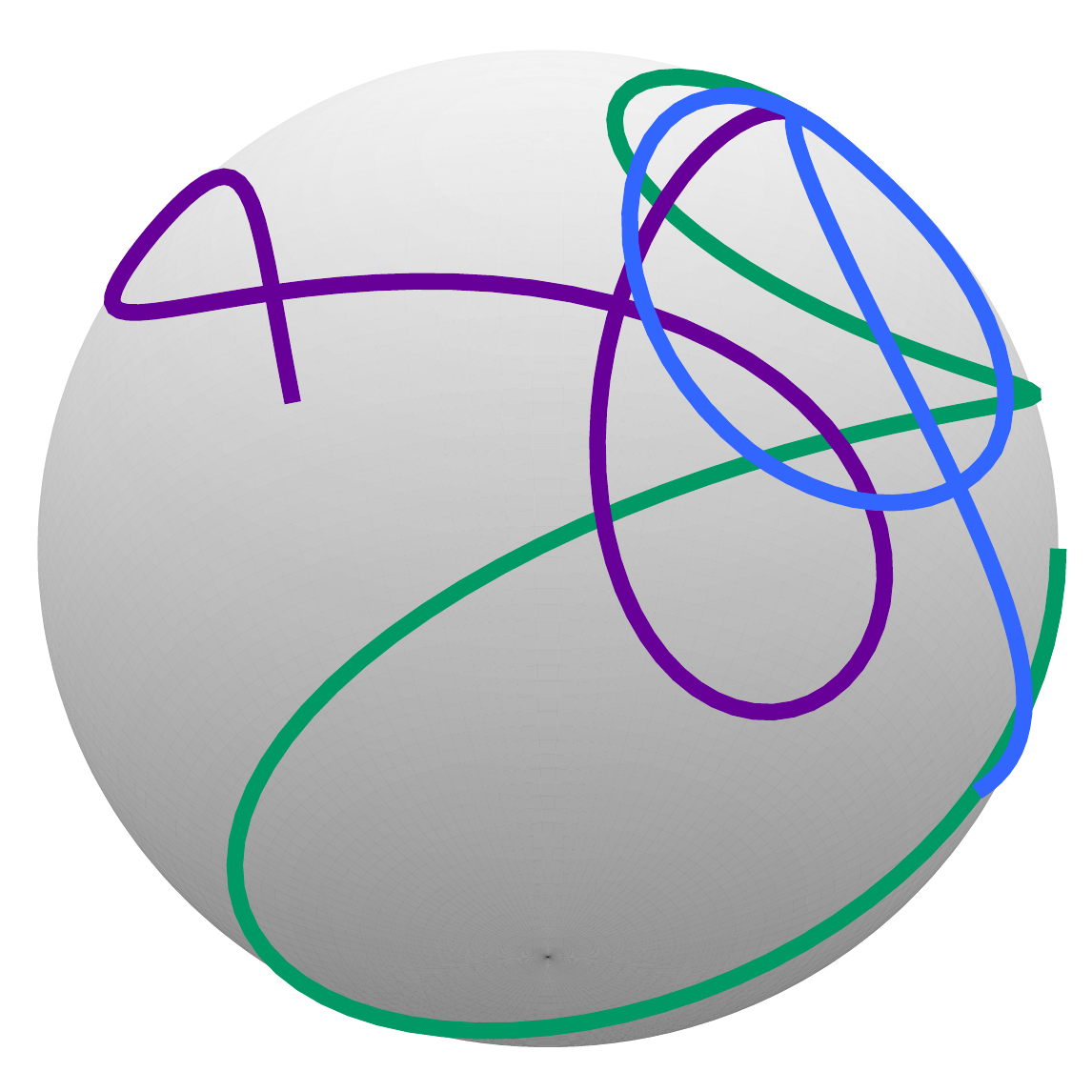}
\includegraphics[width=0.33\textwidth]{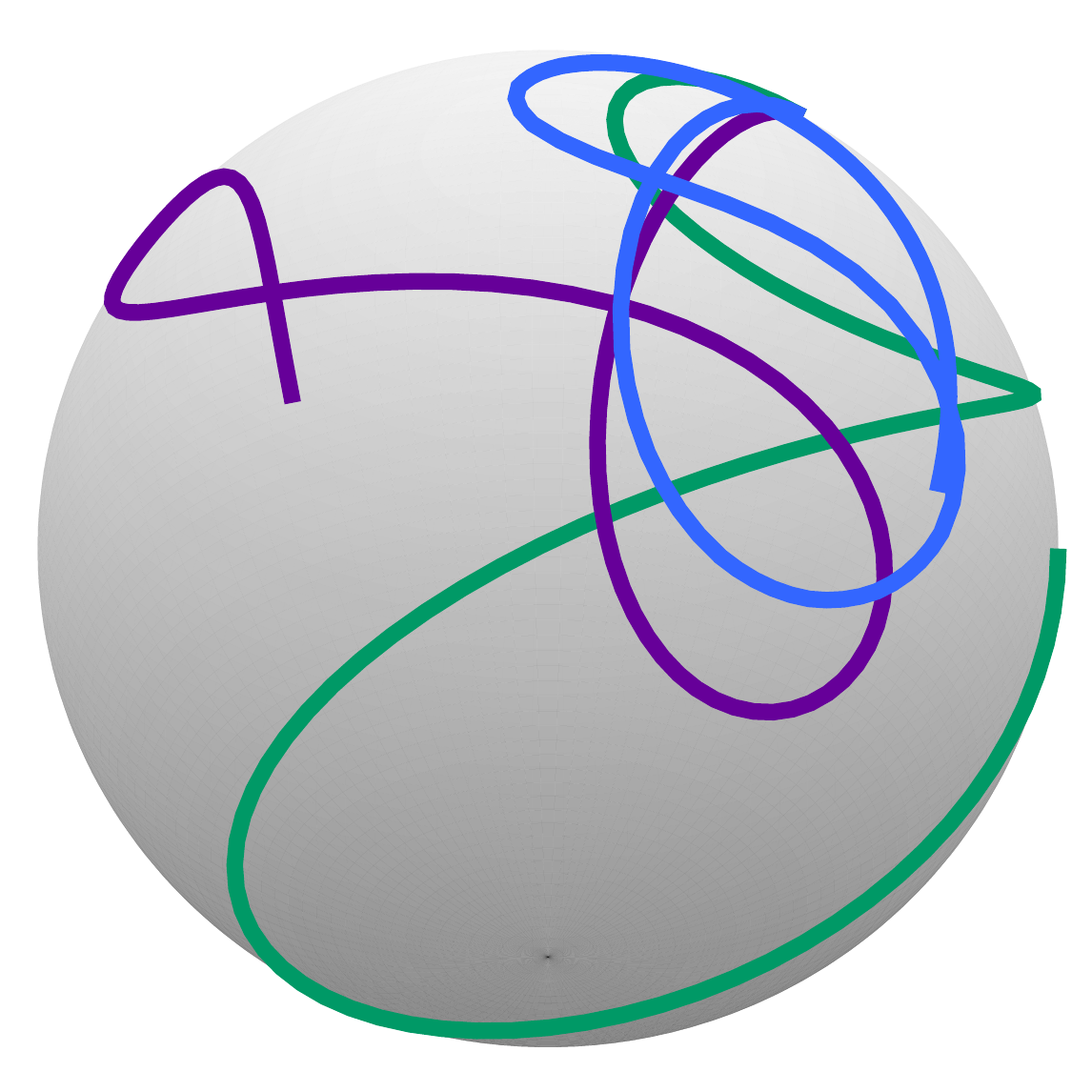}}
\caption{Interpolation between two curves on $\text{S}^2$, with and without reparametrisation, found by the reductive SRVT (\ref{eq:twistedSRVT}). The curves are $c^1(t) =  R_x(2 \pi t) R_y(2 \pi t) R_z(\pi t) \cdot (0,1,1)^\text{T}/\sqrt{2}$ and $c^2(t) =  R_z(2 \pi t) R_x(2 \pi t) R_y(\pi t/2)\cdot (0,1,1)^\text{T}/\sqrt{2}$ for $t \in [0,1]$, where $R_x(t)$, $R_y(t)$ and $R_z(t)$ are the rotation matrices in $\text{SO}(3)$ corresponding to rotation of an angle $t$ around the $x$-, $y$- and $z$-axis, respectively.}
\label{fig:curves2}
\end{center}
\end{figure}

\begin{figure}[htbp]
\begin{center}
\subfloat[From left to right: The original parametrisations of the curves to be interpolated, the reparametrisation minimizing the distance in $\text{SO}(3)$ and the reparametrisation minimizing the distance in $S^\text{2}$, using the SRVT (\ref{eq:numSRVT}).]{
\includegraphics[width=0.33\textwidth]{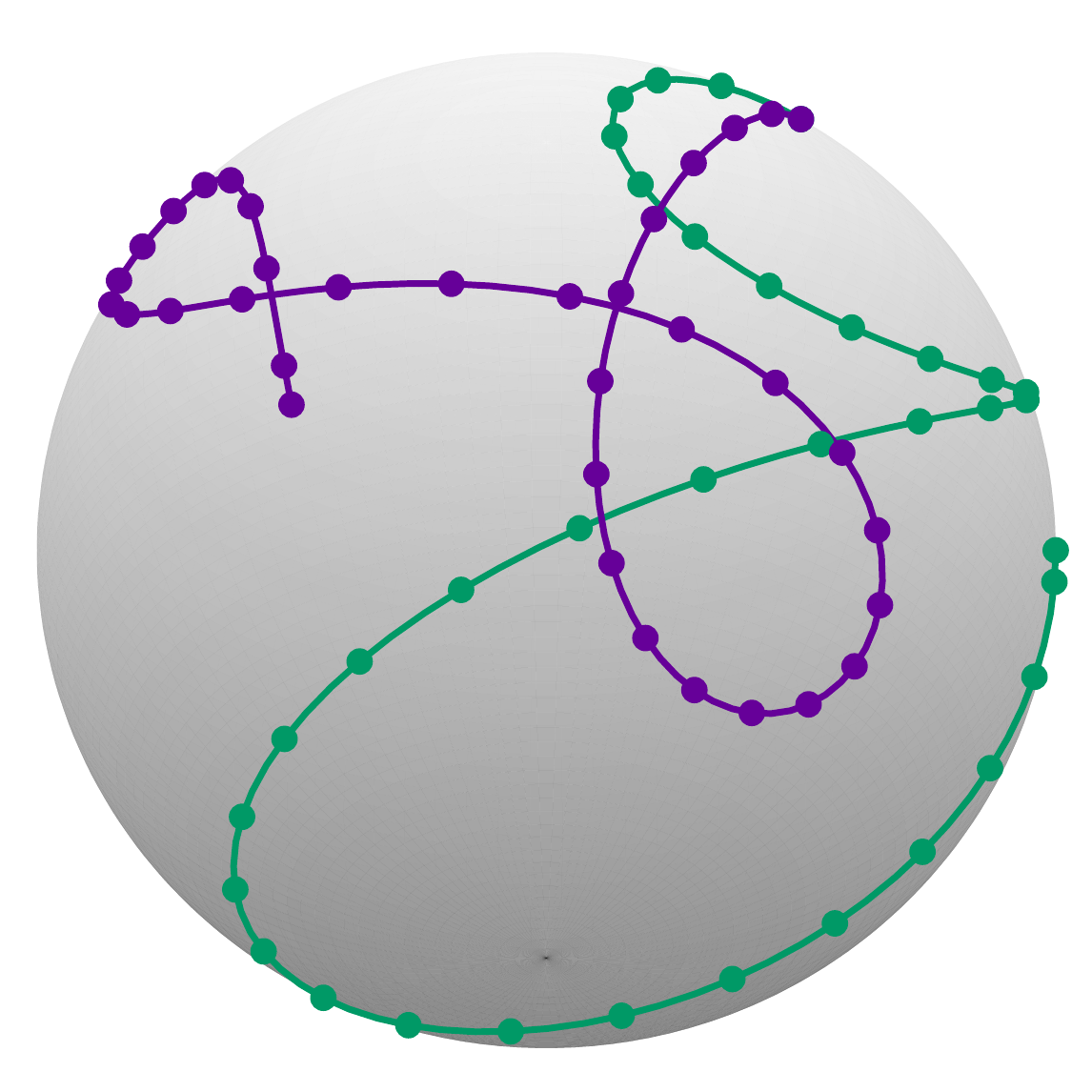}
\includegraphics[width=0.33\textwidth]{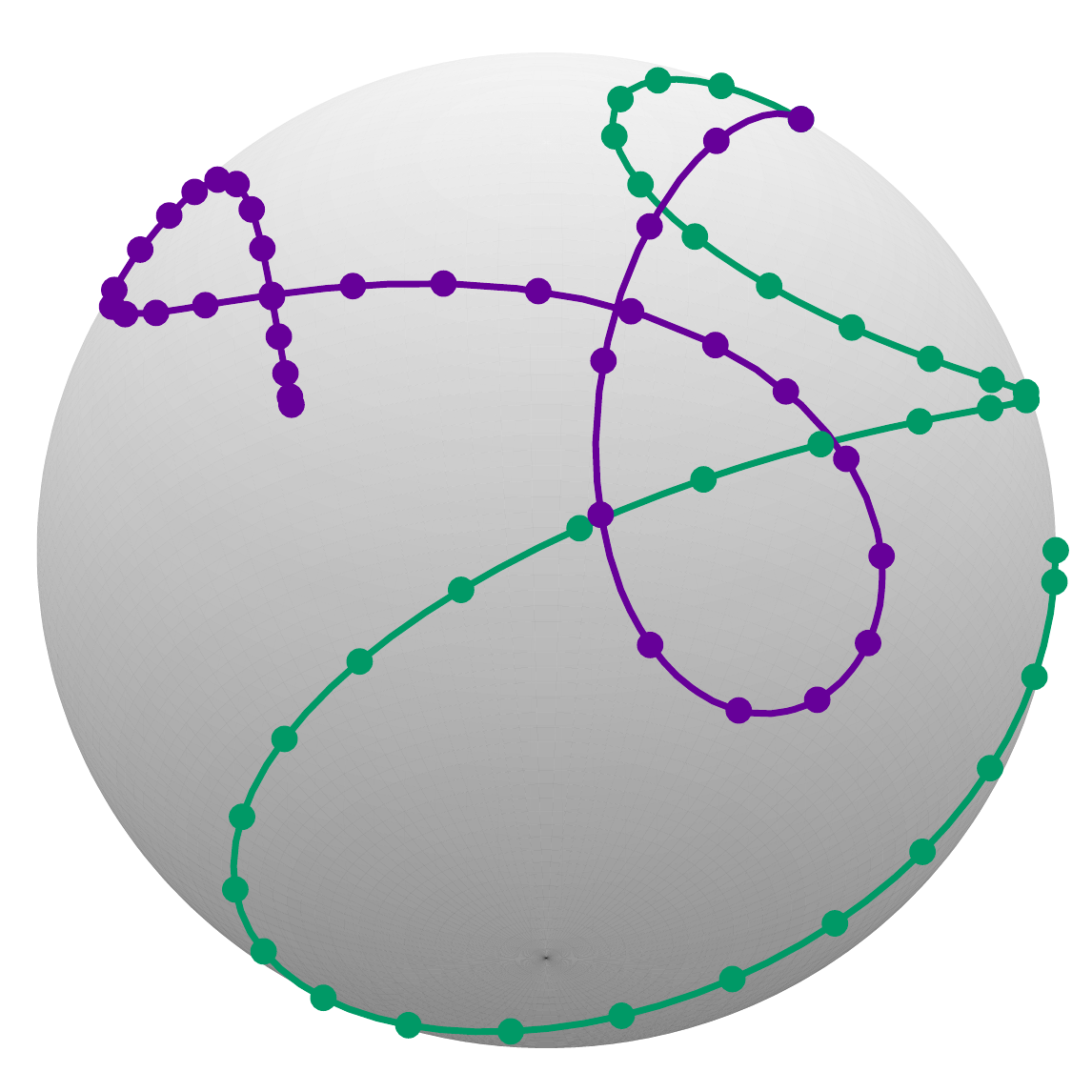}
\includegraphics[width=0.33\textwidth]{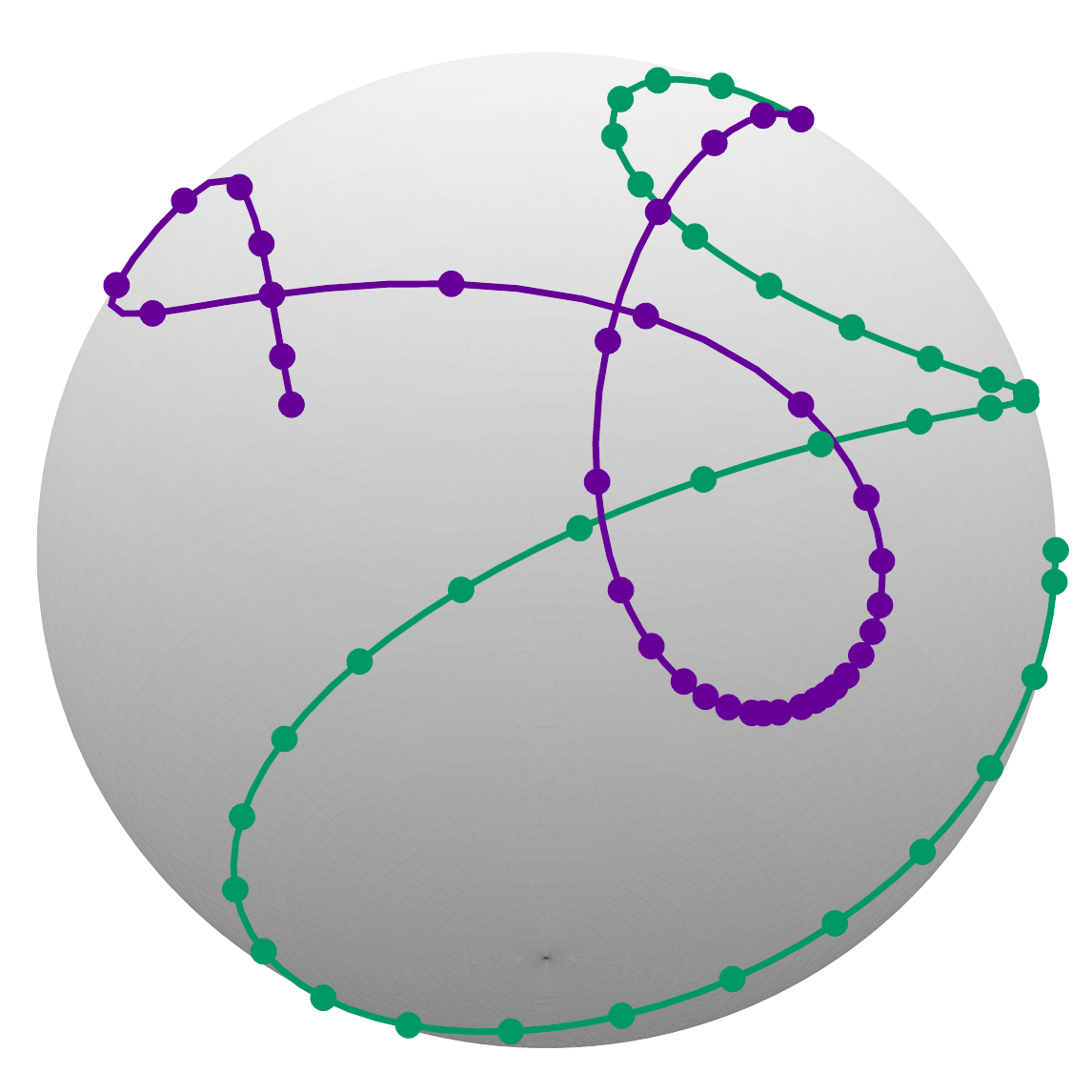}}\\
\subfloat[The interpolated curves at times $t = \left\{\frac{1}{4},\frac{1}{2},\frac{3}{4}\right\}$, from left to right, before reparametrisation, on $\text{SO}(3)$ (yellow, dashed line) and $\text{S}^2$ (red, solid line).]{
\includegraphics[width=0.33\textwidth]{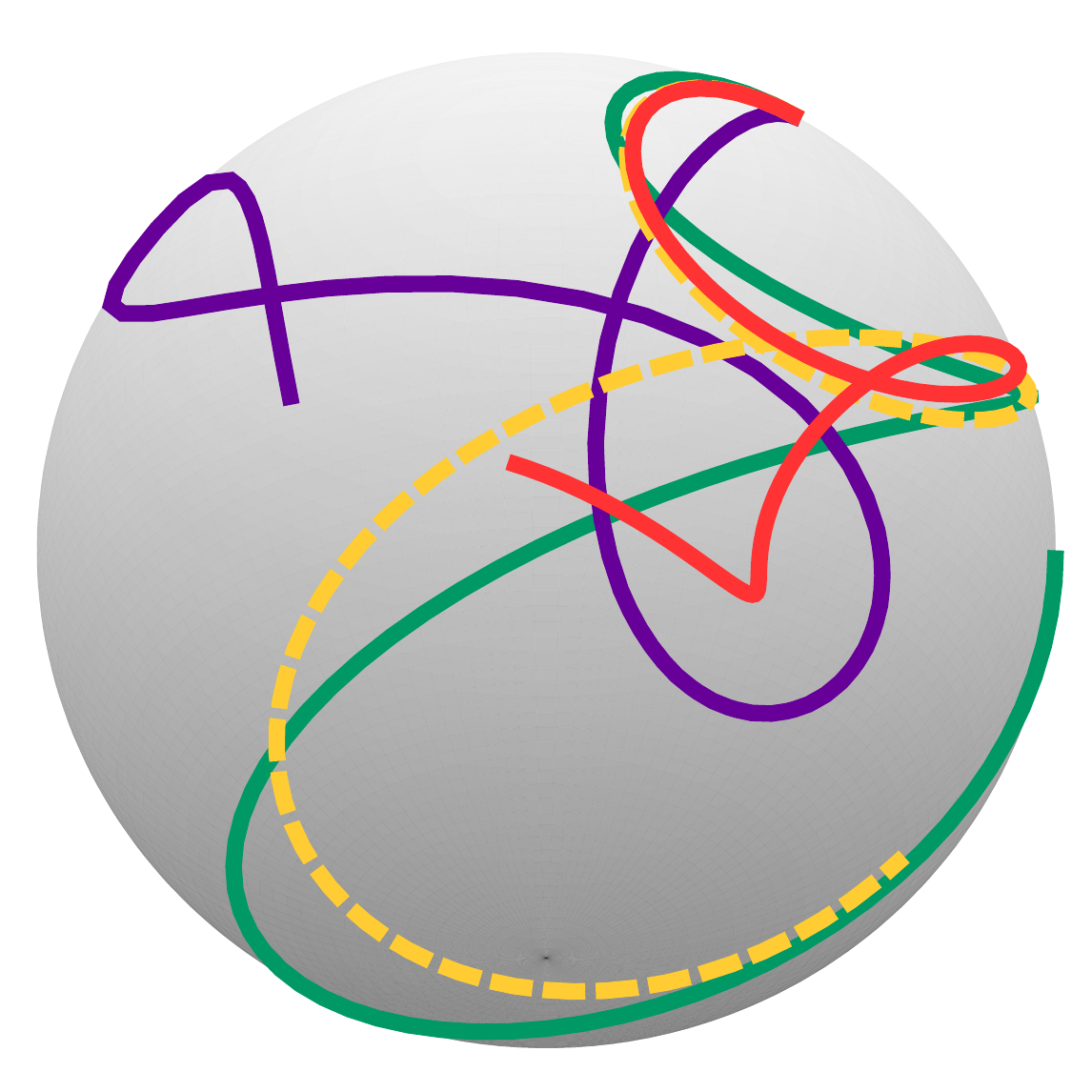}
\includegraphics[width=0.33\textwidth]{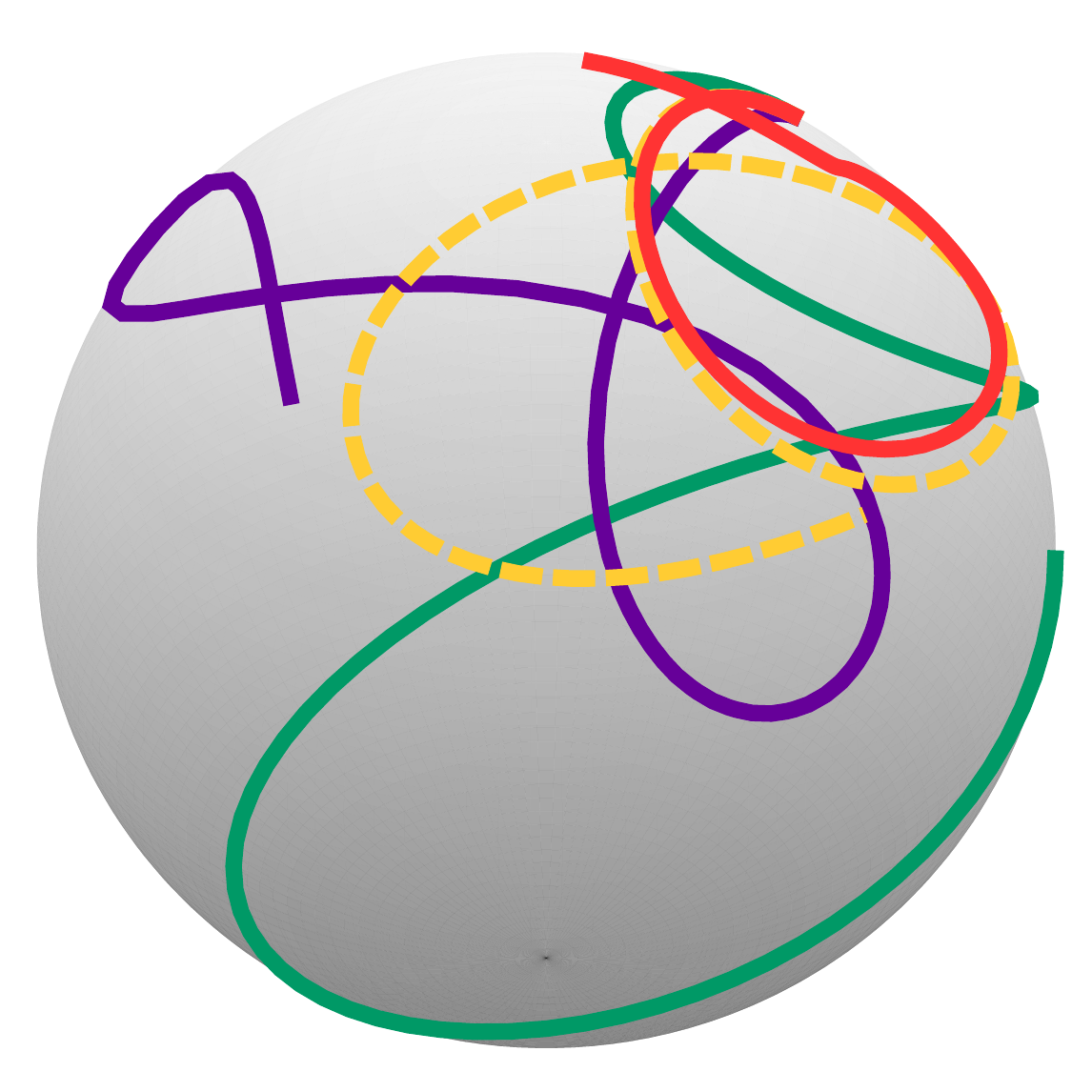}
\includegraphics[width=0.33\textwidth]{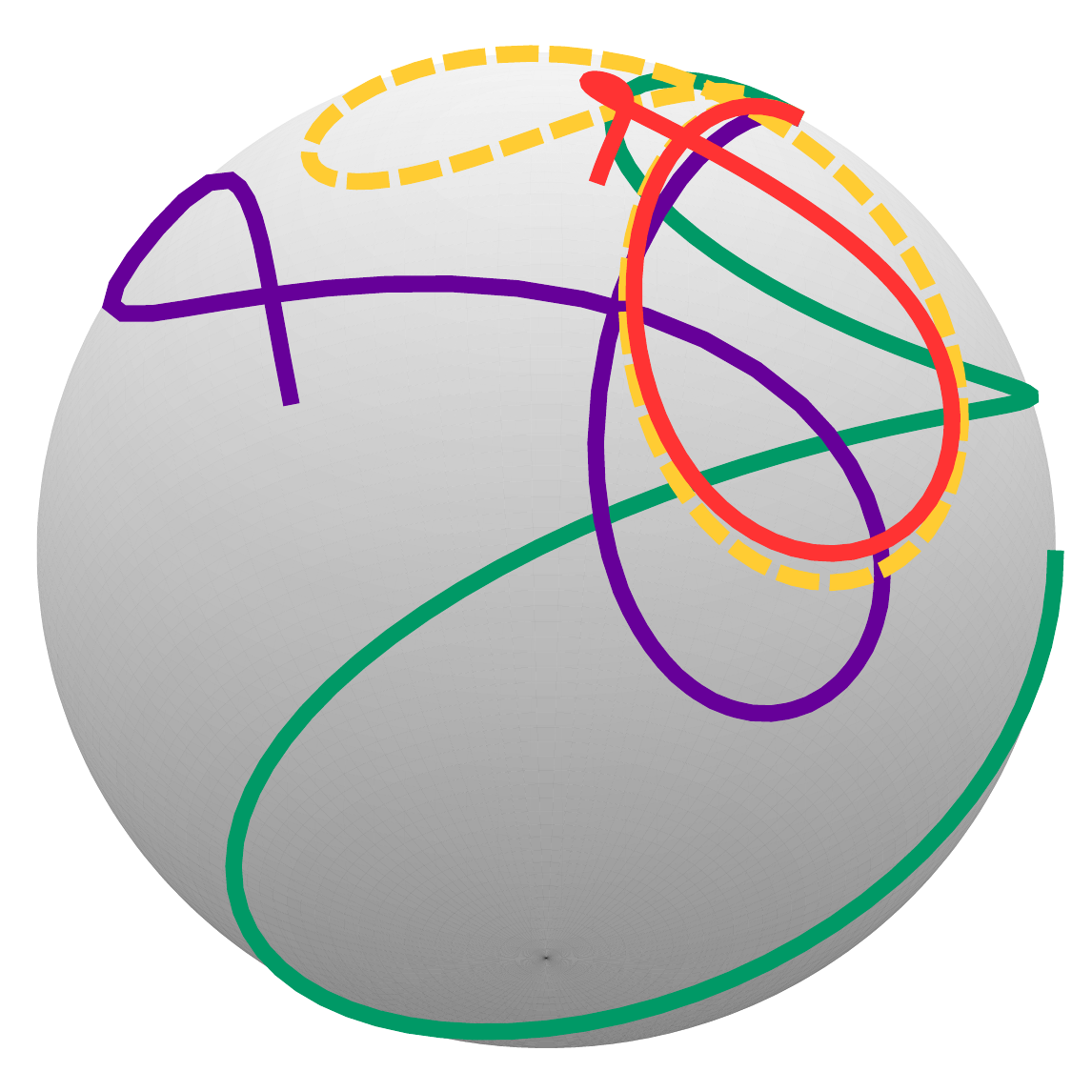}}\\
\subfloat[The interpolated curves at times $t = \left\{\frac{1}{4},\frac{1}{2},\frac{3}{4}\right\}$, from left to right, after reparametrisation, on $\text{SO}(3)$ (yellow, dashed line) and $\text{S}^2$ (red, solid line).]{
\includegraphics[width=0.33\textwidth]{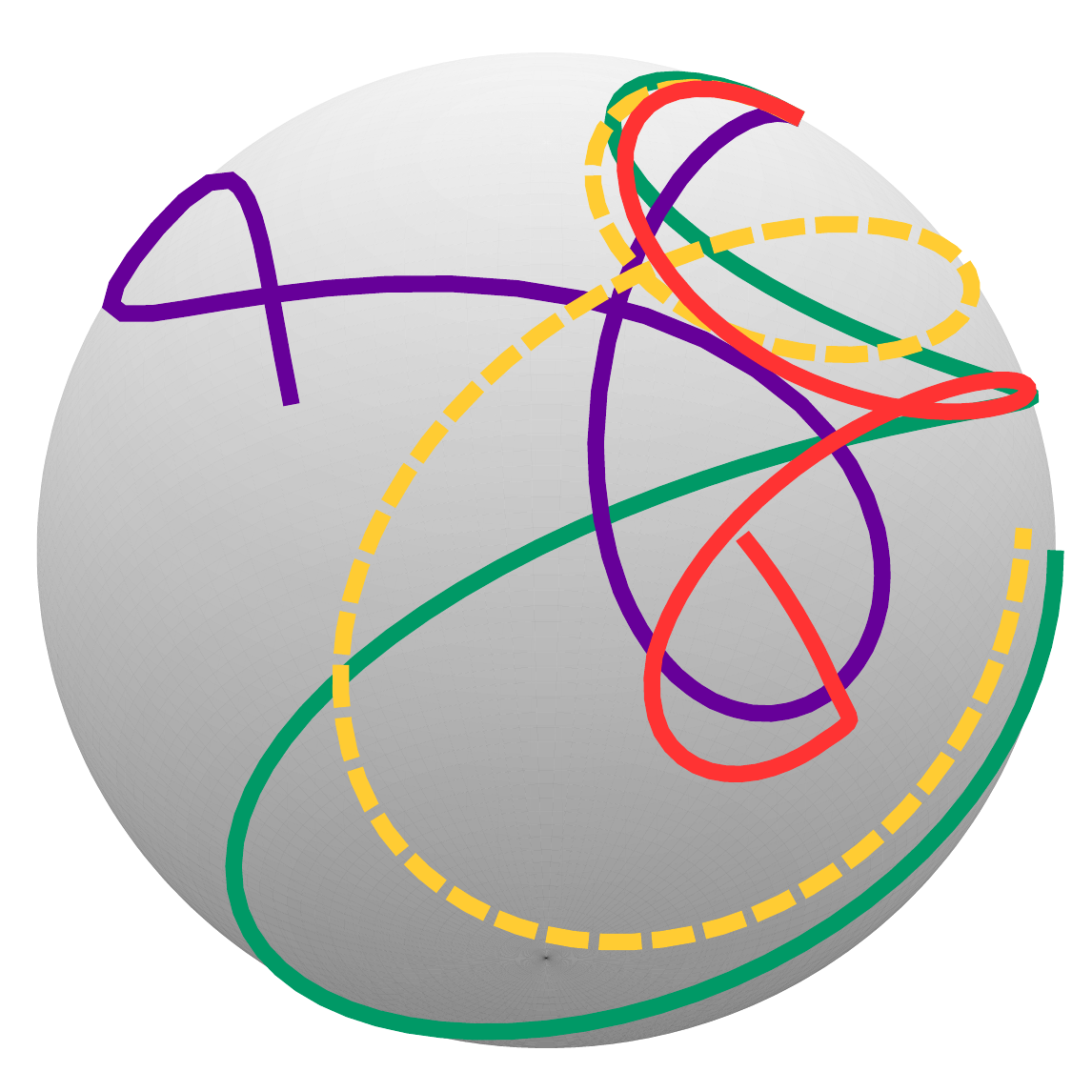}
\includegraphics[width=0.33\textwidth]{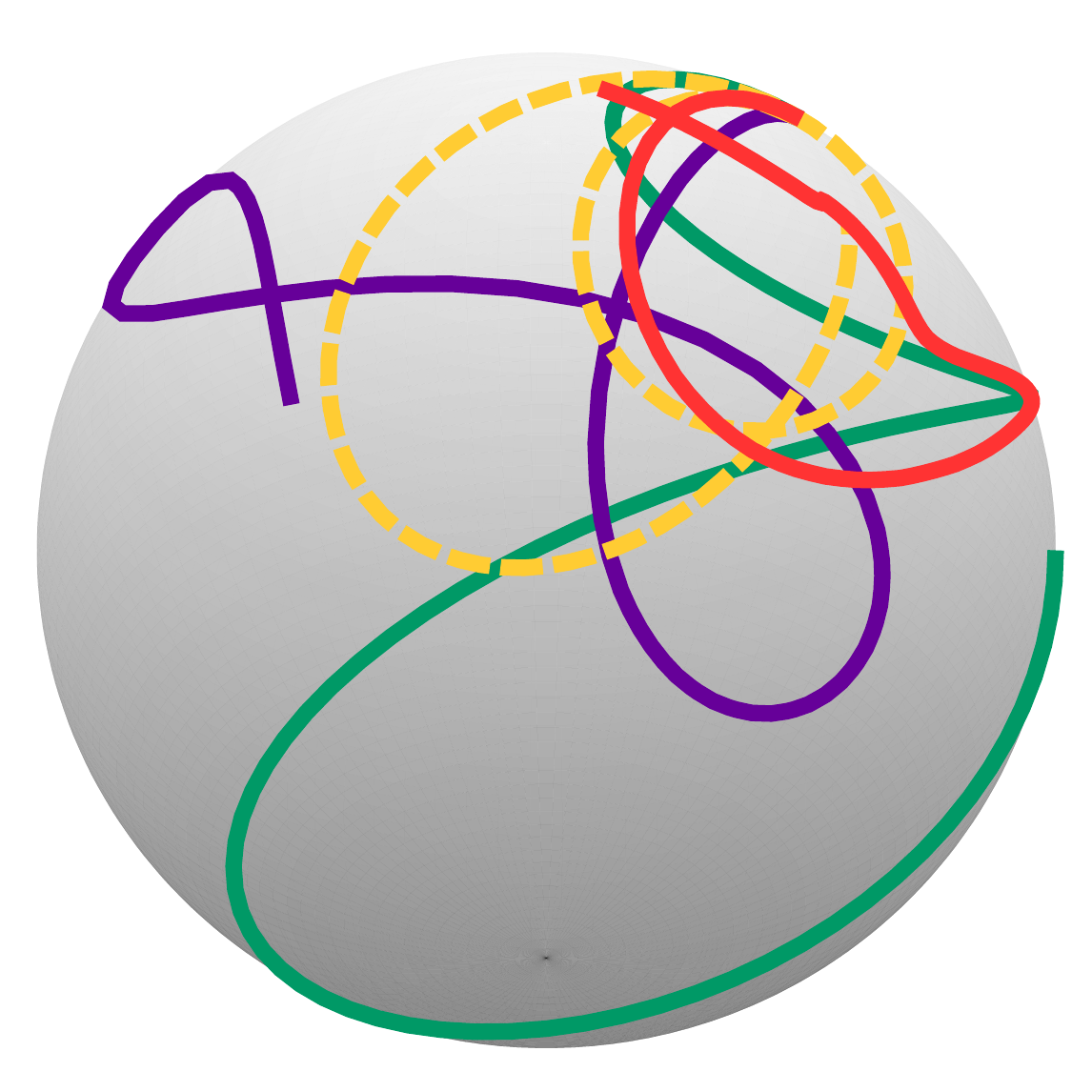}
\includegraphics[width=0.33\textwidth]{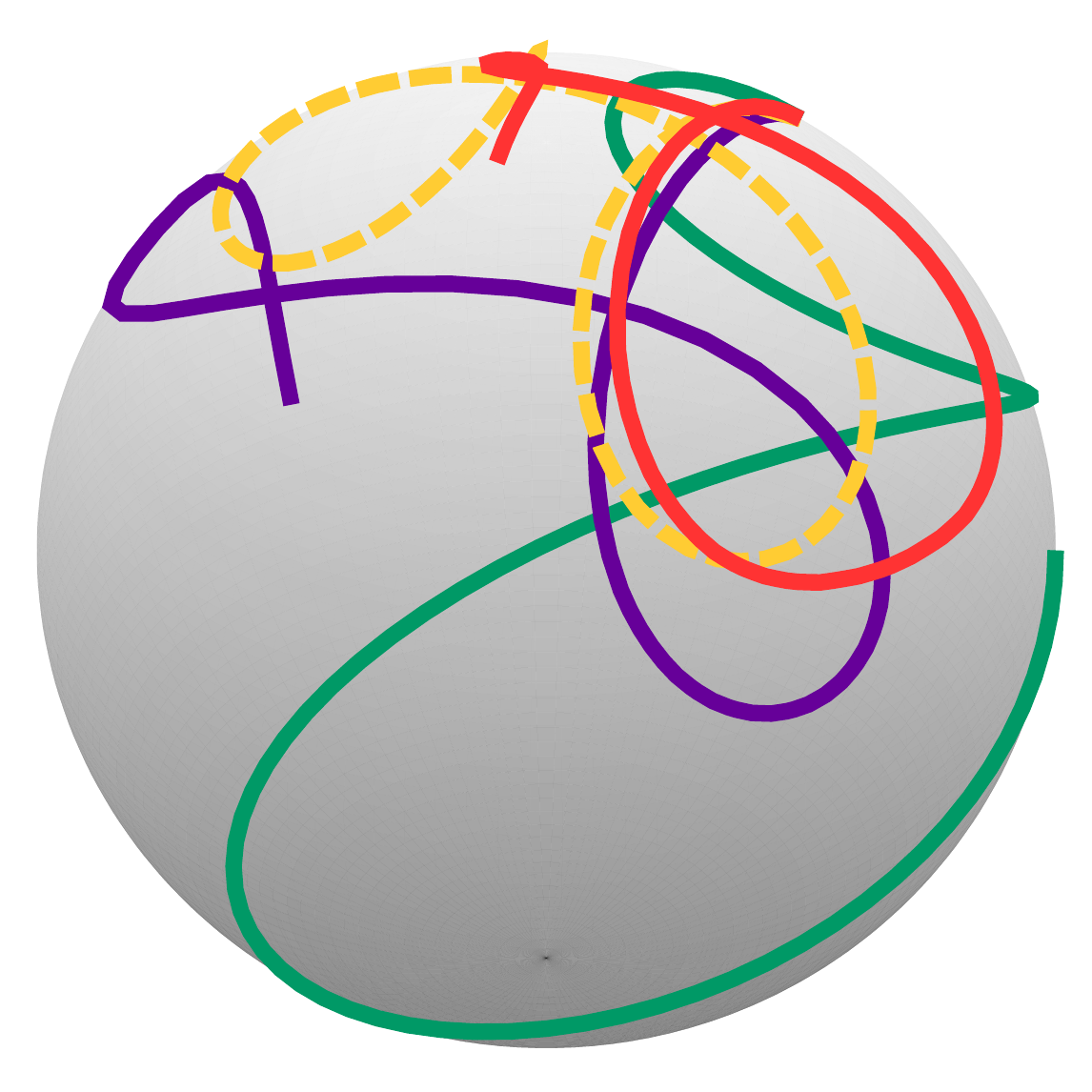}}
\caption{Interpolation between the same curves as in Figure \ref{fig:curves2}, with and without reparametrisation, obtained here with the SRVT (\ref{eq:numSRVT}), compared to the corresponding interpolation between curves on $\text{SO}(3)$ mapped to $\text{S}^2$ by multiplication with the vector $(0,1,1)^\text{T}/\sqrt{2}$.}
\label{fig:curves3}
\end{center}
\end{figure}

\begin{appendix} \section{Appendix}
\begin{setup}[Auxiliary results for Section \ref{sect: SRVT:reductive}]\label{app: aux}
  \begin{lem}\label{lem: smoothness}
 For the homogeneous space $\mathcal{M} = G/H$ with projection $\pi \colon G \rightarrow G/H$ the derivation map 
$
  D_{\mathcal{M}} \colon C^\infty (I, G/H) \rightarrow C^\infty (I, T(G/H)),  c \mapsto \dot{c}
$
 is smooth.
 \end{lem}

 \begin{proof}
 The map $D_G \colon C^\infty (I,G) \rightarrow C^\infty (I,TG),\ \gamma \mapsto \dot{\gamma}$ is a smooth group homomorphism by \cite[Lemma 2.1]{hgreg2015}.
 As $\pi \colon G \rightarrow G/H$ is a smooth submersion, $\theta_\pi \colon C^\infty (I, G) \rightarrow C^\infty (I, G/H), c \mapsto \pi \circ c$ is a smooth submersion \cite[Lemma 2.4]{1706.04816v1}.
 Write $\theta_{T\pi} \circ D_G = D \circ \theta_{\pi}$, whence by \cite[Lemma 1.8]{glockner15fos} $D_{\mathcal{M}}$ is smooth.
\end{proof} 

\begin{lem}\label{lem: identity}
With $\SRp \coloneq \Pomega \circ D$ The identity \eqref{claim} $\id_{C^\infty_{eH} (I,\mathcal{M})} = \pi \circ \Evol \circ \SRp$ holds. 
\end{lem}

\begin{proof}
 Let $c \colon I \rightarrow \mathcal{M}$ be smooth with $c(0)=eH$ and choose $g \colon I \rightarrow G$ smooth with $g(0)=e$ and $\pi \circ g = c$.
	Set $\gamma (t) := \Evol (\SRp (c)) (t)$.
	It suffices to prove that $\gamma (t)^{-1}g(t) \in H$ for all $t \in I$. Then $\pi \circ \gamma  = \pi \circ g = c$ and the assertion follows.
	
	As $\gamma (0)^{-1}g(0) = e \in H$, we only have to prove that $\frac{\dif}{\dif t}\pi (\gamma (t)^{-1}g(t))$ vanishes everywhere to obtain $\gamma (t)^{-1}g(t) \in H$.
	Before we compute the derivative of $\pi (\gamma (t)^{-1}g(t))$, let us first collect some facts concerning the logarithmic derivatives $\delta^r (f) = \dot{f}.f^{-1}$ and $\delta^l (f)= f^{-1}. \dot{f}$.
	By definition $\delta^r (\gamma) = \delta^r (\Evol (\SRp (c))) = \SRp (c)$. 
	Further,\cite[Lemma 38.1]{MR1471480} yields for smooth $f,h \colon I \rightarrow G$: 
	\begin{align}\label{log: deriv}
	\delta^r (f \cdot h) &= \delta^r (f) + \Ad(f).\delta^r (h) \quad \text{ and } \quad \delta^r  (f^{-1}) = - \delta^{l} (f), \quad \text{whence}\\
	\frac{\dif}{\dif t}(\gamma (t)^{-1}g(t)) &\stackrel{\hphantom{\eqref{log: deriv}}}{=} (\gamma (t)^{-1}g(t)) \cdot \delta^l (\gamma^{-1}g)(t) \stackrel{\eqref{log: deriv}}{=} -(\gamma (t)^{-1}g(t)) \cdot  \delta^r(g^{-1}\gamma)(t) \notag \\
	&\stackrel{\eqref{log: deriv}}{=} (\gamma (t)^{-1}g(t)) \cdot (\delta^l(g)(t) - \Ad (g(t)^{-1}).\SRp (c)(t)) \notag
	\end{align}
	Recall that by definition, $\SRp (c)(t) = \omega (\dot{c}(t)) = \Ad (g(t)).\omega_e (T\Lambda^{g(t)^{-1}} (\dot{c} (t)))$ (here $\pi \circ g = c$ is used).
	Inserting this into the above equation we obtain 
	\begin{equation}
	\label{Dgammainvg}
	\frac{\dif}{\dif t}(\gamma (t)^{-1}g(t)) = (\gamma (t)^{-1}g(t)) \cdot (\delta^l(g)(t) - \omega_e (T \Lambda^{g(t)^{-1}} \circ \dot{c}(t))).
	\end{equation}
	Observe that $T_e\pi (\delta^l (g)(t)) = T\Lambda^{g(t)^{-1}} T\pi \dot{g}(t) = T\Lambda^{g(t)^{-1}} \dot{c}(t)$ since $\pi \circ g=c$.
	As $\omega_e$ is a section of $T_e \pi$, $T_e \pi (\delta^l(g)(t) - \omega_e (T \Lambda^{g(t)^{-1}} \circ \dot{c}(t))) = 0 \in T_{eH} \mathcal{M}$.
	Summing up
	\begin{align*}
	\frac{\dif}{\dif t}\pi (\gamma (t)^{-1}g(t)) 
	&\stackrel{\eqref{Dgammainvg}}{=} T \pi ((\gamma (t)^{-1}g(t)) \cdot (\delta^l(g)(t) - \omega_e (T \Lambda^{g(t)^{-1}} \circ \dot{c}(t))) \\
	&\stackrel{\eqref{eq: comm}}{=} T\Lambda^{\gamma (t)^{-1}g(t)} T_e\pi (\delta^l(g)(t) - \omega_e (T \Lambda^{g(t)^{-1}} \circ \dot{c}(t))) \qedhere
	\stackrel{\hphantom{\eqref{eq: comm}}}{=} 0.
	\end{align*}
\end{proof}
\end{setup}
\begin{setup}[A chart for the image of the SRVT]
Let $G$ be a Lie group with Lie algebra $\g$. 
Using the adjoint action of $G$ on $\g$ and the evolution $\Evol \colon C^\infty (I,\g) \rightarrow C^\infty (I,G)$, we define the map 
  \begin{displaymath}
   \Psi \colon C^\infty (I,\g) \rightarrow C^\infty (I,\g),\quad  q \mapsto -\Ad (\Evol (q)^{-1} ) . q,
  \end{displaymath}
where the dot denotes pointwise application of the linear map $\Ad (\Evol(q)^{-1})$.
Observe that $\Psi$ (co)restricts to a mapping $C^\infty (I,\g \setminus \{0\}) \rightarrow C^\infty (I,\g \setminus \{0\})$.

\begin{lem}\label{lem: involution}
 The map $\Psi \colon C^\infty (I,\g) \rightarrow C^\infty (I,\g)$ is a smooth involution.
\end{lem}

\begin{proof}
 To establish smoothness of $\Psi$, consider the commutative diagram 
 \begin{displaymath}
  \begin{xy}
  \xymatrix{
       C^\infty (I,\g) \ar[rrr]^\Psi \ar[d]_{(\Evol, \id_{C^\infty (I,\g)}}   &   & & C^\infty (I,\g) \ar@2{-}[d]  \\
       C^\infty (I,G) \times  C^\infty (I,\g) \ar[rrr]^-{(f,g) \mapsto \Ad (f).g}   &  & &  C^\infty (I,\g)    
  }
\end{xy}.
 \end{displaymath}
 As $\Ad \colon G \times \g \rightarrow \g$ is smooth, so is $(f,g) \mapsto \Ad (f).g$ (cf.\ \cite[Proof of Proposition 6.2]{hgreg2015}) and $\Psi$ is smooth as a composition of smooth maps.
 Compute for $q \in C^\infty (I,\g)$ 
 \begin{align*}
  \Psi (\Psi (q)) &= -\Ad (\Evol (\Psi (q))^{-1} ) . \Psi (q) 
		  = -\Ad (\Evol (-\Ad (\Evol (q)^{-1} ) . q)^{-1} ) . (-\Ad (\Evol (q)^{-1} ) . q)\\
		  &= \Ad ((\Evol (q)\Evol (-\Ad (\Evol (q)^{-1} ) . q))^{-1}).q.
		\end{align*}
 To see that $\Psi (\Psi (q))=q$ , we prove that $\gamma_q \coloneq \Evol (q)\Evol (-\Ad (\Evol (q)^{-1} ) . q)$ is a constant path. 
 Recall that $\Evol (q)$ and $\Evol (-\Ad (\Evol (q)^{-1} ) . q)$ are smooth paths starting at the identity in $G$. Hence it suffices to prove $\delta^r (\gamma_q) =0$.
 To this end, apply the product formula \eqref{log: deriv} and $\delta^r (\Evol (q))=q$: 
  \begin{align*}
   \delta^r (\gamma_q ))
 &= \delta^r (\Evol (q)) + \Ad (\Evol (q)). \delta^r (\Evol (-\Ad (\Evol (q)^{-1}).q))\\
 &= q + \Ad (\Evol (q)).(-\Ad (\Evol(q)^{-1}).q) = q-q =0. \qedhere
  \end{align*}
\end{proof}

 To account for the initial point $c_0 \in \M$, fix $g_0 \in \pi^{-1} (c_0)$ and define 
 \begin{displaymath}
   \Psi_{g_0} \colon C^\infty (I,\g) \rightarrow C^\infty (I,\g), \quad \Psi_{g_0} (q) := \Ad (g_0). \Psi (q)=-\Ad (g_0\Evol (q)^{-1}).q.
 \end{displaymath}
 For $k$ in the center of $G$, $\Psi_k = \Psi$ holds, but in general $\Psi_{g_0}$ will not be an involution. 
 
 \begin{lem}\label{lem: PSi0diff}
  For each $g_0 \in \G$, the map $\Psi_{g_0}$ is a diffeomorphism with inverse $\Psi_{g_0^{-1}}$.
 \end{lem}
 \begin{proof}
  From the definition of $\Psi_{g_0}$ and Lemma \ref{lem: involution}, it is clear that $\Psi_{g_0}$ is a smooth diffeomorphism.
  We use that $\Ad \colon \G \rightarrow \text{GL} (\g)$ is a group morphism and compute 
  \begin{align*}
   \Psi_{g_0^{-1}} (\Psi_{g_0} (q)) &= \Ad (g_0^{-1}).\Psi (\Psi_{g_0} (q)) = \Ad (g_0).\Psi (\Ad (g_0).\Psi(q))\\
				    &= \Ad (g_0^{-1}).\left( -\Ad (\Evol (\Ad (g_0).\Psi(q)))^{-1}).\Ad (g_0).\Psi(q)\right)\\
				    &= -\Ad (g_0^{-1} g_0 \Evol (\Psi (q))^{-1} g_0^{-1}g_0).\Psi (q) = \Psi (\Psi (q)) =q.
  \end{align*}
 Here we used that $\Evol (\Ad (g).f) = g\Evol (f)g^{-1}, \text{ for } g\in \G.$ \end{proof}

 \begin{lem}\label{lem: straightening}
 Fix $c_0 \in \mathcal{M}$ and choose $g_0 \in \G$ with $\pi (g_0) = c_0$. 
Assume that $\mathcal{M}$ is reductive with $\g = \h \oplus \m$, then
 $$\Psi_{g_0} (C^\infty (I,\m \setminus \{0\})) = \{ f \in C^\infty (I,\g) \mid f = \theta_\omega (\dot{c}) \quad \text{for some } c \in \textup{Imm}_{c_0} (I,\mathcal{M})\}.$$
With $\Pomega$ as in \ref{setup: alphared} the formula $\theta_\omega \circ D(\rho_{c_0} \circ \Psi_{g_0} (q)) = \Psi_{g_0} (q)$ holds.
 \end{lem}

 \begin{proof}
  Consider $c \in \text{Imm}_{c_0} (I,\mathcal{M})$ and recall from Proposition \ref{prop: alphaomega:prop} the identity $\Lambda_{c_0} (\Evol (\theta_{\omega} (\dot{c})) = \pi (\Evol (\theta_{\omega} (\dot{c}))g_0) = c$.
  Choose $\hat{c} = \Evol (\theta_{\omega} (\dot{c}))g_0$ as a smooth lift of $c$ to $\G$ and compute as follows:
    \begin{align*}
     \Psi_{g_0^{-1}} (\theta_\omega (\dot{c}))  &= \Ad(g_0^{-1}).\left(-\Ad (\Evol(\theta_\omega (\dot{c}))^{-1}).(\theta_\omega (\dot{c}))\right) \\
				     &= \Ad(g_0^{-1}).\left(-\Ad (\Evol(\theta_\omega (\dot{c}))^{-1}).\Ad (\hat{c}).\omega_e (T\Lambda^{\hat{c}^{-1}} (\dot{c}))\right)\\
				     &= \Ad(g_0^{-1}).\left(-\Ad (\Evol(\theta_\omega (\dot{c}))^{-1}).\Ad (\Evol(\theta_\omega (\dot{c}))g_0).\omega_e (T\Lambda^{\hat{c}^{-1}} (\dot{c}))\right)\\
				     &= -\omega_e (T\Lambda^{\hat{c}^{-1}} (\dot{c})) \in \m \setminus \{0\}.
    \end{align*}
  Conversely, let us show that $\Psi_{g_0} (C^\infty (I,\m \setminus \{0\}))$ is contained in the image of $\theta_\omega \circ D|_{\text{Imm}_{c_0} (I,\mathcal{M})}$.
  To this end, consider $q = \Psi_{g_0} (v)$ for $v \in C^\infty (I, \m \setminus \{0\})$. We compute 
    \begin{equation} \label{comp: immersion} \begin{aligned}
     \rho_{c_0} (q) &= \Lambda_{c_0} (\Evol (\Ad (g_0). \Psi (v))) = \pi (\Evol (\Ad (g_0). \Psi (v)g_0)) \\ 
     &= \pi (g_0 \Evol (\Psi (v))) = \pi (g_0 \Evol (-\Ad (\Evol (v)^{-1}).v)) = \Lambda^{g_0} (\pi (\Evol (\Psi (v)))).
     \end{aligned}
    \end{equation}
 Since $\Lambda^{g_0}$ is a diffeomorphism, $\rho_{c_0} (q) \colon I \rightarrow \mathcal{M}$ is an immersion if and only if the curve $\pi (\Evol (\Psi (v)))$ has a non-vanishing derivative everywhere.
 Recall from the proof of Lemma \ref{lem: involution} that $\Evol (v) \Evol (\Psi (v)) = e$, whence we compute the derivative  
 \begin{equation} \label{comp: derivative}\begin{aligned}
  \frac{\di}{\di t} \pi (\Evol (\Psi (v))(t)) &= T\pi \left( \frac{\di}{\di t} \Evol (\Psi (v))(t) \right ) = T\pi (\Psi (v) \Evol(\Psi(v)))(t)\\
					      &= T\pi (-\Ad (\Evol (v)^{-1}).v \Evol (\Psi(v)))(t)\\
					      &= -T\pi \circ (L_{\Evol (v)^{-1}(t)})_* \circ (R_{\Evol (v)(t)\Evol (\Psi (v))(t)})_* (v(t)) \\
					      &= -T\Lambda^{\Evol (v)^{-1}(t)} \circ T_e\pi (v(t)).
					      \end{aligned}
 \end{equation}
 In passing to the last line, we used that $\pi$ commutes with the left action. 
 Since $T\Lambda^{g}$ is an isomorphism, $\frac{\di}{\di t} \pi (\Evol (\Psi (v))(t))$ vanishes if and only if $v(t) \in \text{ker} T_e\pi = \h$.
 However, $v(t) \in \m\setminus \{0\}$, whence $\rho_{c_0} (\Psi_{g_0} (v)) \in \text{Imm}_{c_0} (I,\mathcal{M})$ and we can apply $\theta_\omega \circ D$ to $\rho_{c_0} (q)$. 
 A combination of \eqref{comp: immersion} and \eqref{comp: derivative} yields 
  $
   \frac{\di}{\di t} \rho_{c_0} (\Psi_{g_0} (v))(t) = -T\Lambda^{g_0(\Evol (v))^{-1}(t)} \circ T_e\pi (v(t)).
  $
With $\pi (g_0\Evol (v)^{-1}) = \rho_{c_0} (\Psi_{g_0} (v)(t))$, this yields 
 \begin{align*}
  &\theta_\omega \left( \frac{\di}{\di t} \rho_{c_0} (v(t))\right) = \theta_\omega \left( \frac{\di}{\di t} \rho_{c_0} (\Psi_{g_0} (v))(t)\right) = \theta_\omega (-T\Lambda^{g_0(\Evol (v))^{-1}(t)} \circ T_e\pi (v(t))) \\
	=& \Ad (g_0(\Evol (v))^{-1}).\omega_e (-T\Lambda^{(g_0(\Evol (v))^{-1}(t))^{-1}} T\Lambda^{g_0(\Evol (v))^{-1}(t)} \circ T_e\pi (v(t)))\\
	=& -\Ad (g_0(\Evol (v))^{-1}).\omega_e (T_e \pi (v(t)) = -\Ad (g_0(\Evol (v))^{-1}).v(t) = \Psi_{g_0} (v)(t).
 \end{align*}
 Note that as $\omega_e = (T_e\pi|_{\m})^{-1}$, we have $\omega_e (T_e \pi (v(t)) = v(t)$.
 \end{proof}
 \end{setup}
\end{appendix}

\bibliographystyle{new}
\bibliography{hom_shape}

\end{document}